\documentclass[11pt,twoside,a4paper]{article}
\usepackage{color,amscd,amsmath,amssymb,amsthm,latexsym,stmaryrd,authblk,mathabx,shuffle,mathrsfs,hyperref,tikz,tkz-tab} 
\usepackage[latin1]{inputenc}
\usepackage[T1]{fontenc} 
\usepackage[english]{babel}
\providecommand{\keywords}[1]{\textbf{\textit{Keywords.}} #1}
\providecommand{\AMSclass}[1]{\textbf{\textit{AMS classification.}} #1}

\input{xy}
\xyoption{all}

\theoremstyle{plain}
\newtheorem{theo}{Theorem}[section]
\newtheorem{lemma}[theo]{Lemma}
\newtheorem{cor}[theo]{Corollary}
\newtheorem{prop}[theo]{Proposition}
\newtheorem{defi}[theo]{Definition}

\theoremstyle{remark}
\newtheorem{remark}{Remark}[section]
\newtheorem{notation}{Notations}[section]
\newtheorem{example}{Example}[section]

\newcommand{\bfG}{\mathbf{G}}
\newcommand{\bfH}{\mathbf{H}}

\newcommand{\bfM}{\mathbf{M}}

\newcommand{\calE}{\mathcal{E}}
\newcommand{\calF}{\mathcal{F}}

\newcommand{\calH}{\mathcal{H}}

\newcommand{\calP}{\mathcal{P}}
\newcommand{\calQ}{\mathcal{Q}}

\newcommand{\K}{\mathbb{K}}
\renewcommand{\L}{\mathbb{L}}
\newcommand{\N}{\mathbb{N}}
\newcommand{\Z}{\mathbb{Z}}

\newcommand{\sym}{\mathfrak{S}}
\newcommand{\cl}{\mathrm{cl}}
\newcommand{\id}{\mathrm{Id}}

\newcommand{\tdelta}{\tilde{\Delta}}

\renewcommand{\ker}{\mathrm{Ker}}

\newcommand{\rk}{\mathrm{rk}}

\newcommand{\bfbool}{\mathbf{Bool}}
\newcommand{\rmbool}{\mathrm{Bool}}
\newcommand{\parti}{\mathrm{Part}}
\newcommand{\ic}{\mathrm{ic}}
\newcommand{\cc}{\mathrm{cc}}
\newcommand{\rmHG}{\mathrm{HG}}
\newcommand{\bfHG}{\bfH\bfG}
\newcommand{\GM}{\bfG\bfM}

\newcommand{\rmbools}{\rmbool_{\mathrm{r}}}
\newcommand{\bfbools}{\bfbool_{\mathrm{r}}}
\newcommand{\rmboolhs}{\rmbool_{\mathrm{hr}}}
\newcommand{\bfboolhs}{\bfbool_{\mathrm{hr}}}
\newcommand{\rmboolc}{\rmbool_{\mathrm{cou}}}
\newcommand{\bfboolc}{\bfbool_{\mathrm{cou}}}
\newcommand{\rmboolt}{\rmbool_{\mathrm{t}}}
\newcommand{\bfboolt}{\bfbool_{\mathrm{t}}}
\newcommand{\rmboolm}{\rmbool_{\mathrm{mat}}}
\newcommand{\bfboolm}{\bfbool_{\mathrm{mat}}}
\newcommand{\rmboolgm}{\rmbool_{\mathrm{gmat}}}
\newcommand{\bfboolgm}{\bfbool_{\mathrm{gmat}}}
\newcommand{\rmboollm}{\rmbool_{\mathrm{lmat}_\L}}
\newcommand{\bfboollm}{\bfbool_{\mathrm{lmat}_\L}}
\newcommand{\rmboolmax}{\rmbool_{\mathrm{max}}}
\newcommand{\bfboolmax}{\bfbool_{\mathrm{max}}}
\newcommand{\Phichr}{\Phi_{\mathrm{chr}}}

\newcommand{\starq}{\star_{q_1,q_2}}

\begin{document}

\title{Bialgebraic structures on boolean functions}
\date{}
\author{Lo\"ic Foissy}
\affil{\small{Univ. Littoral Côte d'Opale, UR 2597
LMPA, Laboratoire de Mathématiques Pures et Appliquées Joseph Liouville
F-62100 Calais, France}.\\ Email: \texttt{foissy@univ-littoral.fr}}

\maketitle

\begin{abstract}
We study several bialgebraic structures on boolean functions, that is to say maps defined on the set of subsets of a finite set $X$, taking the value $0$ on $\emptyset$. Examples of boolean functions are given by the indicator function of the hyperedges of a given hypergraph, or the rank function of a matroid. 
We give the species of boolean functions a two-parameters family of products and a coproduct, and this defines a two-parameters family of twisted bialgebras. 

We then try to define a second coproduct on boolean functions, based on contractions, in order to obtain a double bialgebra. We show that this is not possible on the whole species of boolean functions, but that there exists a maximal subspecies where this is possible.
This subspecies being rather mysterious, we introduce rigid boolean functions and show that this subspecies has indeed a second coproduct, as wished, and that it contains rank functions of matroids and indicator functions associated to hypergraphs.
As a consequence, we obtain a unique polynomial invariant on rigid boolean functions, which is a generalization of the chromatic polynomial of graphs. 
\end{abstract}

\\

\AMSclass{16T05 16T30 06E30 05B35 05C65}

\tableofcontents

\section*{Introduction}

In the past decade, in the theory of combinatorial Hopf algebras, an important interest has been taken on what we call here double bialgebras, a shorthand for pairs of bialgebras in cointeraction or, equivalently, bialgebra in the category of comodules over another bialgebra. 
These objects are pairs of bialgebras $(A,m,_A\delta)$ and $(B,m_B,\Delta)$, with a right coaction $\rho:B\longrightarrow B\otimes A$ of $A$ over $B$, such that $m_B$, $\Delta$, the unit and the counit of $B$ are comodules morphisms. With more details:
\begin{itemize}
\item $\rho$ is a right coaction:
\begin{align*}
(\rho\otimes\id_A)\circ\rho&=(\id_B\otimes\delta)\circ\rho,&(\id_B\otimes\varepsilon_A)\circ\rho&=\id_B.
\end{align*}
\item $m_B$ is a comodule morphism:
\[\rho\circ m_B=(m_B\otimes m_A)\circ (\id_B\circ\tau_{A,B}\circ\id_A)\circ (\rho\otimes\rho),\]
where $\tau_{A,B}:A\otimes B\longrightarrow B\otimes A$ is the usual flip, sending $a\otimes b$ to $b\otimes a$. 
\item The unit $1_B$ of $B$ is a comodule morphism from the base field $\K$ to $B$:
\[\rho(1_B)=1_B\otimes 1_A.\]
\item $\Delta$ is a comodule morphism:
\[(\Delta\otimes\id_A)\circ\rho= (\id_B\otimes\id_B\otimes m_A)\circ (\id_B\circ\tau_{A,B}\circ\id_A)\circ (\rho\otimes\rho)\circ\Delta.\]
\item The counit $\varepsilon_\Delta$ of $B$ is a comodule morphism from $B$ to $\K$:
\begin{align*}
&\forall b\in B,& (\varepsilon_\Delta\otimes\id_A)\circ\rho(x)=\varepsilon_B(x)1_A.
\end{align*}
\end{itemize}
In particular, the second and third items are equivalent to the fact that $\rho$ is an algebra morphism from $B$ to $B\otimes A$.
The most interesting examples are when $A$ and $B$ share the same underlying space and product, with the coaction given by $\delta$ itself.
The first known non trivial combinatorial example seems to be described in \cite{Calaque2011}: it is based on rooted trees, and $(B,m,\Delta)$ is the Connes-Kreimer Hopf algebra \cite{Connes1998}, whereas $\delta$ is based on a procedure of contraction of edges. 
A similar (but much more intricate, with important decorations on vertices and edges) is used in the theory of regularity structure \cite{Hairer2014,Bruned2019}, used to solve a class of stochastic differential equations. 
Other examples can be found on various families of graphs \cite{Foissy36,Manchon2012}, posets \cite{Foissy37}, hypergraphs \cite{Ebrahimi-Fard2022,Foissy44}, noncrossing partitions \cite{Ebrahimi-Fard2020,Foissy43}, multi-indices \cite{Foissy57}$\ldots$
In all these cases, the product is rather simple (usually, a disjoint union of combinatorial objects), the first coproduct $\Delta$ is given by splitting the objects into two parts through a restriction operation, maybe with constraints on the possible choices,
and the second one is based on contractions, that can be seen as quotienting the objects through an equivalence.\\

So far, matroids did not fit into this frame of double bialgebras, though attempts have been made \cite{Catoire2024}. A matroid is a formalization of linear independence, due to Whitney \cite{Whitney35} in terms of independent sets.
Equivalent formulations in terms of bases or circuits can be found in the literature. We shall refer to the online lecture notes \cite{Goemans,Vyuka} for classical definitions and results on matroids. 
Disjoint unions of matroids and restriction of matroids are classical operations, leading to a first Hopf-algebraic structure on them; the problem is to define properly contractions  (or quotients) of matroids. 

In order to avoid this problem, we choose to work with a larger class of objects, namely boolean functions. If $X$ is a finite set, a boolean function on $X$ is a map $f:\calP(X)\longrightarrow\Z$, such that $f(\emptyset)=0$. 
We shall use here the characterization of matroids by their rank functions. The rank function of a matroid on a finite set $X$ is a map $f:\calP(X)\longrightarrow\N$ such that:
\begin{itemize}
\item For any $A\subseteq X$, $f(A)\leq |A|$. In particular, $f(\emptyset)=0$. 
\item $f$ is increasing: for any $A\subseteq B\subseteq X$, $f(A)\leq f(B)$.
\item $f$ is submodular: for any $A,B\subseteq X$, $f(A\cup B)+f(A\cap B)\leq f(A)+f(B)$.
\end{itemize}
Therefore, matroids can be seen as boolean functions. Another class of boolean functions are associated to hypergraphs: if $H$ is a hypergraph, it is associated to the boolean function $\gamma(H)$ on the set of vertices $V(H)$ of $H$ by
\begin{align}
\label{EQ1}\gamma(H):\left\{\begin{array}{rcl}
\calP(X)&\longrightarrow&\Z\\
A&\longmapsto&|\{Y\in E(H)\mid Y\subseteq A\}|.
\end{array}\right.
\end{align}
These two classes of examples of boolean functions take their value in $\N$, but we choose here to work with boolean functions taking their values in $\Z$ instead, which will allow us to define several transformations on boolean functions without the concern of verifying positivity conditions. 
We work in the frame of (linear) species and twisted bialgebras, which will avoid us tedious discussions on isomorphisms between boolean functions, through the bosonic Fock functor defined in \cite{Aguiar2010}.
The species of boolean functions is denoted by $\rmbool$, and its linearization by $\bfbool$. The correspondences $\gamma$ and $\rk$ define species morphisms from the species of hypergraphs or the species of matroids to $\rmbool$.
We first define a two-parameters family $\starq$ of products on $\bfbool$ in Theorem \ref{theo2.2}. If $X$ and $Y$ are two disjoint finite sets, $f\in\rmbool(X)$ and $g\in\rmbool(Y)$, then $f\starq g$ is the boolean function on $X\sqcup Y$ defined by
\begin{align*}
\forall A\subseteq X\sqcup Y,&&f\starq g(A)&=q_1f(A\cap X)+q_2g(A\cap Y).
\end{align*}
In particular, if $q_1=q_2$, the product $\star_{q_1,q_2}$ is commutative. We give several results on these products. A very classical inductive argument shows that any boolean function $f$ can be decomposed as a product 
\[f=f_1\starq\cdots\starq f_k,\]
where $f_1,\ldots,f_k$ are $(q_1,q_2)$-indecomposable boolean functions. Moreover, this decomposition is unique, up to a set of permutations depending on the commutation relations (for $\starq$) between the $(q_1,q_2)$-indecomposable boolean functions (Lemma \ref{lemma2.7} and Proposition \ref{prop2.8}). 
When $q_1=q_2$, as $\star_{q_1,q_2}$ is commutative, then all the permutations are of course possible. When $q_1\neq q_2$, they are restricted to permutations of boolean functions of the form
\[\left\{\begin{array}{rcl}
\calP(\{x\})&\longrightarrow&\Z\\
\emptyset&\longmapsto&0\\
\{x\}&\longrightarrow&\lambda.
\end{array}\right.\]
We also give a few results of characterizations of  $(q_1,q_2)$-indecomposable boolean functions. For example, if $|X|=1$, any boolean function on $X$ is $(q_1,q_2)$-indecomposable;
if $X=\{x,y\}$ is a pair, then $f\in\rmbool(X)$ is $(q_1,q_2)$-indecomposable if, and only if,
\[f(\{x,y\})\notin\{q_1f(\{x\})+q_2f(\{y\}),\:q_2f(\{x\})+q_1f(\{y\})\}.\]
We give more results in the case $q_1=q_2=1$. In this case, we will simply write indecomposable instead of $(1,1)$-indecomposable. We give a sufficient condition for a boolean function to be indecomposable (Proposition \ref{prop2.17}),
which we will often use in to produce examples and counterexamples. We also characterize indecomposable boolean functions on a set of cardinality 3, see Proposition \ref{prop2.18}. 
We then define a restriction coproduct on boolean functions in Theorem \ref{theo2.19}, such that for any $(q_1,q_2)\in\Z^2$, $(\bfbool,\starq,\Delta)$ is a twisted bialgebra. 
The application of the bosonic Fock functor \cite{Aguiar2010} gives a two-parameters family of bialgebras $(\calH_{\bfbool},\starq,\Delta)$, where $\calH_{\bfbool}$ is the space of isoclasses of boolean functions, and
\begin{align*}
\forall f_1\in\rmbool(X),\:\forall f_2\in\rmbool(Y),&&\overline{f_1}\starq\overline{f_2}&=\overline{f_1\starq f_2},\\
\forall f\in\rmbool(X),&&\Delta(\overline{f})&=\sum_{X'\sqcup X''=X}\overline{f_{\mid X'}}\otimes\overline{f_{\mid X''}}.
\end{align*}
Here, for any boolean function $f$, we denote by $\overline{f}$ its isomorphism class.\\

Complications arise when we try to define a second coproduct, based on contractions and restrictions, following the formalism of \cite{Foissy41}. We first define the contraction and the restriction of a boolean function $f$ by an equivalence $\sim$.
More precisely, if $X$ is a finite set, $f\in\rmbool(X)$ and $\sim$ is an equivalence on $X$, we define $f/{\sim}\in\rmbool(X/{\sim})$ by
\begin{align*}
&\forall A\subseteq X/{\sim},&f/{\sim}(A)&=f(\varpi_\sim^{-1}(A)),
\end{align*}
where $\varpi_\sim:X\longrightarrow X/{\sim}$ is the canonical surjection, and we define $f\mid\sim\in\rmbool(X)$ by
\begin{align*}
&\forall A\subseteq X,&f\mid\sim(A)&=\sum_{Y\in X/\sim}f(A\cap Y),
\end{align*}
see Definition \ref{defi3.1}. We look for a coproduct $\delta$ on the form
\[\delta^\calE_\sim(f)=\begin{cases}
f/{\sim}\otimes f\mid\sim\mbox{ if }\sim\in\calE(f),\\
0\mbox{ otherwise},
\end{cases}\]
where $\calE(f)$ is a family of equivalences depending on $f$. We give necessary and sufficient conditions on $\calE$ such that this coproduct $\delta$ is compatible with the product $\star_1$ (Proposition \ref{prop3.7}),
coassociative (Proposition \ref{prop3.9}), compatible with the coproduct $\Delta$ (Proposition \ref{prop3.9}), or counitary (Proposition \ref{prop3.11}).
Two natural families of equivalences occur: 
\begin{itemize}
\item Weak equivalences: for $f\in\rmbool(X)$, $\calE^W(f)$ is the set of equivalences $\sim$ on $X$ such that any $Y\in X/{\sim}$, $f_{\mid Y}$ is indecomposable.
\item Strong equivalences: for $f\in\rmbool(X)$, $\calE^S(f)$ is the set of weak equivalences $\sim$ on $X$ such that $f/{\sim}$ and $f$ have the same number of indecomposable components.
\end{itemize}
The associated coproducts are respectively denoted by $\delta^W$ and $\delta^S$. We obtain in Theorems \ref{theo3.14} and \ref{theo3.17} that:
\begin{itemize}
\item $\delta^W$ is compatible with the product $\star_1$ and the coproduct $\Delta$, but is not coassociative, and not counitary (it only has a right counit).
\item $\delta^S$ is coassociative and counitary, compatible with the product $\star_1$, but not with the coproduct $\Delta$. 
\end{itemize}
So, none of them satisfy all required properties to obtain a double bialgebra on the whole species of boolean functions. In fact, no family of equivalences allow to obtain all these properties, as it is shown in Proposition \ref{prop3.19}.
Let us detail this negative result: we restrict ourselves to a subspecies $\rmboolt$ of $\rmbool$, stable under the product $\star_1$, and the coproduct $\Delta$, in order to obtain a twisted subbialgebra $\bfboolt$ of $(\bfbool,\star_1,\Delta)$.
We also assume that this subspecies is stable by contractions associated to a certain family $\calE$ of equivalences, defining on $\bfboolt$ a contraction-restriction coproduct $\delta^\calE$.
If this coproduct is coassociative, counitary, compatible with the product $\star_1$ and the coproduct $\Delta$, then Proposition \ref{prop3.19} states that for any $f\in\rmboolt$,
\[\calE(f)=\calE^W(f)=\calE^S(f).\]
This explains why we cannot proceed on the whole species $\rmbool$, as there exist boolean functions $f$ with $\calE^S(f)\neq\calE^W(f)$. 
It is tempting to restrict ourselves to the subspecies $\rmboolc$ of boolean functions $f$ such that $\calE^W(f)=\calE^S(f)$. However, this subspecies is not stable under the coproduct $\Delta$.
This leads us to the definition of convenient subspecies of $\rmbool$ (Definition \ref{defi3.21}), which is a subspecies of $\rmbool$ stable under the product $\star_1$, the coproduct $\Delta$, and contractions, and such that $\calE^W(f)=\calE^S(f)$ for any boolean map $f$ in this subspecies. 
We prove that there exists a convenient subspecies $\rmboolmax$, maximal for the inclusion (Proposition \ref{prop3.24}). 
However, the elements of $\rmboolmax$ are difficult to characterize. We introduce two convenient subspecies, of rigid and of hyper-rigid boolean functions (Definition \ref{defi4.5} and \ref{defi4.6}), smaller than $\rmboolmax$ but easier to handle with.
In particular, if $H$ is a hypergraph, with set of vertices $X$, we prove in Proposition \ref{prop4.16} that $\gamma(H)$, defined in (\ref{EQ1}), is a rigid boolean function.
This gives an injective species morphism $\gamma$ from hypergraphs to rigid boolean functions. This map is a twisted bialgebra morphism, but not a twisted double bialgebra morphism (Theorem \ref{theo4.17}), as it is not compatible with $\delta$:
in fact, the subspecies of boolean functions $\gamma(H)$ is not stable under $\delta$. Similarly, rank functions associated to matroids (see Definition \ref{defi4.20}) are rigid boolean functions, which gives a twisted subbialgebra of matroids, 
again not stable under $\delta$ (Theorem \ref{theo4.25}). As particular examples, we consider graphical matroids and linear matroids, which form twisted subbialgebras, see Definitions \ref{defi4.26} and \ref{defi4.30}.

As a consequence, the space $\calH_{\bfboolmax}$ generated by isomorphism classes of boolean functions in $\rmboolmax$ inherits a double bialgebraic structure. This implies that there exists a unique double bialgebra morphism $\Phi$ from $\calH_{\bfboolmax}$ to $\K[T]$,
which is described in Theorem \ref{theo5.1} and Proposition \ref{prop5.2}. For any boolean function $f\in\rmboolmax(X)$, and for any $n\geq 1$, $\Phi(f)(n)$ is the number of maps $c:X\longrightarrow\{1,\ldots,n\}$ such that for any $i$, $f_{\mid c^{-1}(i)}$ is modular (Definition \ref{defi2.16}). In particular:
\begin{itemize}
\item If $f$ comes from a hypergraph $H$, then $\phi(f)$ is the chromatic polynomial of $H$, as defined in \cite{Helgason74}, see also \cite{Borowiecki2000,Dohmen95,Tomescu98,ZhangDong2017}, see Example \ref{ex1.3} for more details.
\item If $f$ is the rank function of a graphical matroid associated to a graph $G$, then $\Phi(f)(n)$ is the number of maps $c$ from the set of edges $E(G)$ of $G$ to $\{1,\ldots,k\}$, such that for any $i$, the graph formed by the edges sent to $i$ by $c$ and the attached vertices is a forest. 
\end{itemize}
This result is extended from $\rmboolmax$ to a larger species $\rmboolc$ on which $\delta^W$ and $\delta^S$ coincide, and we prove in Proposition \ref{prop5.6} that $\Phi$ is still compatible with $\delta^W$ or $\delta^S$.
This allows to give a formula for the antipode of elements of $\rmboolmax$, using the values of $\Phi(f)$ at $-1$, see Proposition \ref{prop5.5}.\\

This paper is organized as follows. The first section contains reminders on double bialgebras and the formalism of species that allows to construct examples of double bialgebras.
In the second section, the first bialgebraic structures on boolean functions are introduced: the two-parameters family of products, and the coproduct $\Delta$. 
Results on these products are also given, including a decomposition into indecomposable boolean functions and the unicity up to the orders of terms of such a decomposition. 
The study of a possible contraction-restriction coproducts is done in the third section, which leads to the notion of convenient subspecies of boolean functions. 
The next section gives several cases of convenient subspecies: rigid or hyper-rigid ones, hypergraphs, rank functions of matroids. In the fifth section, several results on the chromatic polynomial on matroids attached to this construction are given. 
Finally, a proof of a classical result on separators of matroids is made in the appendix.\\

\begin{notation}\begin{enumerate}
\item We denote by $\K$ a commutative field of characteristic zero. All the vector spaces of this text will be taken over this field.
\item For any $n\in\N$, we denote by $[n]$ the set $\{1,\ldots,n\}$. In particular, $[0]=\emptyset$.
\item Let $V,W$ be two vector spaces. We denote by $\tau_{V,W}$ the usual flip
\[\tau_{V,W}:\left\{\begin{array}{rcl}
V\otimes W&\longrightarrow&W\otimes V\\
v\otimes w&\longmapsto&w\otimes v.
\end{array}\right.\]
\item Let $X$ be a finite set. 
\begin{enumerate}
\item We denote by $\calE(X)$ the set of equivalence relations on $X$. If $\sim\in\calE(X)$, we denote by $\varpi_\sim:X\longrightarrow X/{\sim}$ the canonical surjection, and by $\cl(\sim)$ the number of classes of $\sim$, that is to say the cardinality of $X/{\sim}$.
\item For any bijection $\sigma:X\longrightarrow X'$ between two finite sets, and for any $\sim\in\calE(X)$, we define $\calE(\sigma)(\sim)=\sim_\sigma\in\calE(X')$ by
\begin{align*}
&\forall x',y'\in X',&x'\sim_\sigma y'\Longleftrightarrow\sigma^{-1}(x')\sim\sigma^{-1}(y').
\end{align*}
This defines a species $\calE$. 
\item If $\sim\in\calE(X)$, we identify $\{\sim'\in\calE(X)\mid\sim\subseteq\sim'\}$ and $\calE(X/{\sim})$, via the map sending $\sim'$ to $\overline{\sim'}$ defined by
\begin{align*}
&\forall\overline{x},\overline{y}\in X/{\sim},&\overline{x}\overline{\sim'}\overline{y}&\Longleftrightarrow x\sim' y.
\end{align*}
Note that $\varpi_{\sim'}=\varpi_{\overline{\sim'}}\circ\varpi_\sim$.
\end{enumerate}
\item Let $X$ be a finite set.
\begin{enumerate}
\item We denote by $\calP(X)$ the set of subsets of $X$. For $k\in\N$, we denote by $\calP_k(X)$ the set of subsets of $X$ of cardinality $k$.
\item For any finite set $X$, we denote by $\parti(X)$ the set of partitions of $X$. For example,
\begin{align*}
\parti(\{x\})&=\{\{\{x\}\}\},\\
\parti(\{x,y\})&=\{
\{\{x,y\}\},\{\{x\},\{y\}\}\},\\
\parti(\{x,y,z\})&=\{\{\{x,y,z\}\},\{\{x,y\},\{z\}\},\{\{x,z\},\{y\}\},\{\{y,z\},\{x\}\},\{\{x\},\{y\},\{z\}\}\}.
\end{align*}
\item For any bijection $\sigma:X\longrightarrow X'$ between two finite sets, and for any $\pi\in\parti(X)$, we define $\parti(\sigma)(\pi)\in\parti(X')$ by
\[\parti(\sigma)(\pi)=\{\sigma(b)\mid b\in\pi\}.\]
This defines a species $\parti$. 
\item The following defines a species isomorphism $\Psi$ from $\calE$ to $\parti$:
\begin{align*}
&\forall\sim\in\calE(X),&\Psi(\sim)&=X/{\sim}.
\end{align*}
\end{enumerate}\end{enumerate}\end{notation}

\section{Reminders}

\label{section1}

\subsection{Double bialgebras}

We refer to \cite{Foissy37,Foissy36,Foissy40} for details.

\begin{defi}
A double bialgebra is a family $(H,m,\Delta,\delta)$ such that:
\begin{enumerate}
\item $(H,m,\Delta)$ and $(H,m,\delta)$ are bialgebras. Their common unit is denoted by $1_H$. The counits of $\Delta$ and $\delta$ are respectively denoted by $\varepsilon_\Delta$ and $\epsilon_\delta$. We put
\[\eta_b:\left\{\begin{array}{rcl}
\K&\longrightarrow&H\\
\lambda&\longmapsto&\lambda 1_H.
\end{array}\right.\]
\item $(H,m,\Delta)$ is a bialgebra in the category of right comodules over $(H,m,\delta)$, with the coaction $\delta$, seen as a coaction over itself. This is equivalent to the two following assertions:
\begin{align*}
(\varepsilon_\Delta\otimes\id_H)\circ\delta&=\eta_H\circ\varepsilon_\Delta,\\
(\Delta\otimes\id_H)\circ\delta&=m_{1,3,24}\circ (\delta\otimes\delta)\circ\Delta,
\end{align*} 
where
\[m_{1,3,24}:\left\{\begin{array}{rcl}
H^{\otimes 4}&\longrightarrow& H^{\otimes 3}\\
x_1\otimes x_2\otimes x_3\otimes x_4&\longmapsto&x_1\otimes x_3\otimes x_2x_4.\\
\end{array}\right.\]
\end{enumerate}\end{defi}

\begin{example}\label{ex1.1}
An example of double bialgebra is given by the usual polynomial algebra $\K[T]$\footnote{The indeterminate will be denoted by $T$, as $X$ will be often a finite set in the sequel.}, with its usual product $m$ and the two (multiplicative) coproducts defined by
\begin{align*}
\Delta(T)&=T\otimes 1+1\otimes T,&\delta(T)&=T\otimes T.
\end{align*}
The counits are given by
\begin{align*}
\varepsilon_\Delta:&\left\{\begin{array}{rcl}
\K[T]&\longrightarrow&\K\\
P(T)&\longmapsto&P(0),
\end{array}\right.&
\epsilon_\delta:&\left\{\begin{array}{rcl}
\K[T]&\longrightarrow&\K\\
P(T)&\longmapsto&P(1).
\end{array}\right.\end{align*}\end{example}

The double structure allows to find the antipode for the first structure, whenever it exists:

\begin{theo} \label{theo1.2}\cite[Corollary 2.3]{Foissy40}
Let $(H,m,\Delta,\delta)$ be a double bialgebra. 
\begin{enumerate}
\item Then $(H,m,\Delta)$ is a Hopf algebra if, and only if, the character $\epsilon_\delta$ has an inverse $\mu_H$ for the convolution product $*$ dual to $\Delta$. 
Moreover, if this holds, the antipode of $(H,m,\Delta)$ is given by
\[S=(\mu_H\otimes\id_H)\circ\delta.\]
\item Let $\phi_H:(H,m,\Delta,\delta)\longrightarrow (\K[T],m,\Delta,\delta)$ be a double bialgebra morphism. Then $\epsilon_\delta$ has an inverse for the convolution product $*$, given by
\begin{align*}
\mu_H&:\left\{\begin{array}{rcl}
H&\longrightarrow&\K\\
x&\longmapsto&\phi_H(x)(-1).
\end{array}\right.
\end{align*}\end{enumerate}\end{theo}

We shall say that a double bialgebra $(H,m,\Delta,\delta)$ is connected if the bialgebra $(H,m,\Delta)$ is connected. If so, we obtain more results:

\begin{theo}\label{theo1.3} \cite[Theorem 3.9]{Foissy40}
\begin{enumerate}
\item Let $(H,m,\Delta)$ be a connected bialgebra, and $\lambda$ be a character of $H$ (that is to say, an algebra morphism from $m$ to $\K$). There exists a unique Hopf algebra morphism $\Phi_\lambda:(H,m,\Delta)\longrightarrow (\K[T],m,\Delta)$
such that for any $x\in H$, 
\[\Phi_\lambda(x)(1)=\lambda(x).\]
Moreover, for any $x\in\ker(\varepsilon_\Delta)$,
\[\Phi_\lambda(x)=\sum_{k=0}^\infty\lambda^{\otimes k}\circ\tdelta^{(k-1)}(x)\frac{T(T-1)\cdots (T-k+1)}{k!},\]
where for any $x\in\ker(\varepsilon_\Delta)$, $\tdelta(x)=\Delta(x)-x\otimes 1_H-1_H\otimes x$ and, for $k\geq 1$,
\[\tdelta^{(k-1)}(x)=\begin{cases}
x\mbox{ if }k=1,\\
(\tdelta^{(k-2)}\otimes\id_H)\circ\tdelta(x)\mbox{ if }k\geq 2.
\end{cases}\]
\item Let $(H,m,\Delta,\delta)$ be a connected double bialgebra.
Then $\Phi_{\epsilon_\delta}$ is the unique double bialgebra morphism from $(H,m,\Delta,\delta)$ to $(\K[T],m,\Delta,\delta)$. 
\end{enumerate}\end{theo}

\subsection{From species to double bialgebras}

We shall work in the context of (linear) species \cite{Joyal1981,Joyal1986}. Recall that a species $\calQ$ is a functor from the category of finite sets, with bijections, to the category of vector spaces, that is to say:
\begin{itemize}
\item For any finite set $X$, $\calQ(X)$ is a vector space.
\item For any bijection $\sigma:X\longrightarrow Y$ between two finite sets, $\calQ(\sigma)$ is a linear map
from $\calQ(X)$ to $\calQ(Y)$.
\item If $\sigma:X\longrightarrow Y$ and $\tau:Y\longrightarrow Z$ are two bijections between finite sets,
then $\calQ(\tau\circ\sigma)=\calQ(\tau)\circ\calQ(\sigma)$.
\item For any finite set $X$, $\calQ(\id_X)=\id_{\calQ(X)}$. 
\end{itemize}
We similarly define set species, as functor from the category of finite sets, with bijections, to the category of sets.

\begin{defi}
A twisted algebra is an algebra in the category of species, with the Cauchy tensor product. 
With more details, a twisted algebra is a pair $(\calQ,m)$, such that:
\begin{itemize}
\item $\calQ$ is a species.
\item For any pair $(X,Y)$ of disjoint finite sets, $m_{X,Y}:\calQ(X)\otimes\calQ(Y)\longrightarrow\calQ(X\sqcup Y)$ is a linear map,
\end{itemize}
with:
\begin{itemize}
\item If $(X,Y)$ and $(X',Y')$ are pairs of disjoint finite sets, $\sigma:X\longrightarrow X'$ and $\tau:Y\longrightarrow Y'$ are bijections, then the following diagram commutes:
\[\xymatrix{\calQ(X)\otimes\calQ(Y)\ar[r]^{m_{X,Y}}\ar[d]_{\calQ(\sigma)\otimes\calQ(\tau)}&\calQ(X\sqcup Y)\ar[d]^{\calQ(\sigma\sqcup\tau)}\\
\calQ(X')\otimes\calQ(Y')\ar[r]_{\hspace{2mm}m_{X',Y'}}&\calQ(X'\sqcup Y')}\]
where $\sigma\sqcup\tau$ is the bijection defined by
\[\sigma\sqcup\tau:\left\{\begin{array}{rcl}
X\sqcup Y&\longrightarrow&X'\sqcup Y'\\
x&\longmapsto&\begin{cases}
\sigma(x)\mbox{ if }x\in X,\\
\tau(x)\mbox{ if }x\in Y.
\end{cases}\end{array}\right.\]
\item If $(X,Y,Z)$ is a triple of pairwise disjoint finite sets, then the following diagram commutes:
\[\xymatrix{\calQ(X)\otimes\calQ(Y)\otimes\calQ(Z)\ar[rr]^{m_{X,Y}\otimes\id_{\calQ(Z)}}
\ar[d]_{\id_{\calQ(X)}\otimes m_{Y,Z}}&&\calQ(X\sqcup Y)\otimes\calQ(Z)\ar[d]^{m_{X\sqcup Y,Z}}\\
\calQ(X)\otimes\calQ(Y\sqcup Z)\ar[rr]_{m_{X,Y\sqcup Z}}&&\calQ(X\sqcup Y\sqcup Z)}\]
\item There exists $1_\calQ\in\calQ(\emptyset)$, such that for any finite set $X$, for any $x\in\calQ(X)$,
\[m_{\emptyset,X}(1_\calQ\otimes x)=m_{X,\emptyset}(x\otimes 1_\calQ)=x.\]
\end{itemize}\end{defi}

Dually:

\begin{defi}
A twisted coalgebra is a coalgebra in the category of species, with the Cauchy tensor product. 
With more details, a twisted coalgebra is a pair $(\calQ,\Delta)$, with:
\begin{itemize}
\item $\calQ$ is a species.
\item For any pair $(X,Y)$ of disjoint finite sets, $\Delta_{X,Y}:\calQ(X\sqcup Y)\longrightarrow\calQ(X)\otimes\calQ(Y)$ is a linear map,
\end{itemize}
such that:
\begin{itemize}
\item If $(X,Y)$ and $(X',Y')$ are pairs of disjoint finite sets, $\sigma:X\longrightarrow X'$ and $\tau:Y\longrightarrow Y'$ are bijections, then the following diagram commutes:
\[\xymatrix{\calQ(X)\otimes\calQ(Y)\ar[d]_{\calQ(\sigma)\otimes\calQ(\tau)}&\calQ(X\sqcup Y)\ar[l]_{\Delta_{X,Y}}\ar[d]^{\calQ(\sigma\sqcup\tau)}\\
\calQ(X')\otimes\calQ(Y')&\calQ(X'\sqcup Y')\ar[l]^{\hspace{2mm}\Delta_{X',Y'}}}\]
\item If $(X,Y,Z)$ is a triple of pairwise disjoint finite sets, then the following diagram commutes:
\[\xymatrix{\calQ(X)\otimes\calQ(Y)\otimes\calQ(Z)
&&\calQ(X\sqcup Y)\otimes\calQ(Z)\ar[ll]_{\Delta_{X,Y}\otimes\id_{\calQ(Z)}}\\
\calQ(X)\otimes\calQ(Y\sqcup Z)\ar[u]^{\id_{\calQ(X)}\otimes\Delta_{Y,Z}}&&\calQ(X\sqcup Y\sqcup Z)\ar[ll]^{\Delta_{X,Y\sqcup Z}}\ar[u]_{\Delta_{X\sqcup Y,Z}}}\]
\item There exists a linear map $\varepsilon_\Delta:\calQ(\emptyset)\longrightarrow\K$, such that for any finite set $X$, 
\[\left(\varepsilon_\Delta\otimes\id_{\calQ(X)}\right)\circ\Delta_{\emptyset,X}=\left(\id_{\calQ(X)}\otimes\varepsilon_\Delta\right)\circ\Delta_{X,\emptyset}=\id_{\calQ(X)}.\]
\end{itemize}\end{defi}

And finally:

\begin{defi}
A twisted bialgebra is a bialgebra in the category of species, with the Cauchy tensor product. 
With more details, a twisted coalgebra is a triple $(\calQ,m,\Delta)$, such that:
\begin{itemize}
\item $(\calQ,m)$ is a twisted algebra.
\item $(\calQ,\Delta)$ is a twisted coalgebra.
\item For any $x,y\in\calQ(\emptyset)$, $\varepsilon_\Delta\circ m_{\emptyset,\emptyset}(x\otimes y)=\varepsilon_\Delta(x)\varepsilon_\Delta(y)$. Moreover, $\varepsilon_\Delta(1_\calQ)=1$.
\item If $(X,Y)$ and $(X',Y')$ are two pairs of disjoint finite sets such that $X\sqcup Y=X'\sqcup Y'$, then
the following diagram commutes:
\[\xymatrix{\calQ(X)\otimes\calQ(Y)\ar[rr]^{m_{X,Y}}\ar[d]_{\Delta_{X\cap X',X\cap Y'}\otimes\Delta_{Y\cap X,Y\cap Y'}}&&\calQ(X\sqcup Y)\ar[d]^{\Delta_{X',Y'}}\\
\calQ(X\cap X')\otimes\calQ(X\cap Y')\otimes\calQ(Y\cap X')\otimes\calQ(Y\cap Y')\ar[d]_{\id_{\calQ(X\cap X')}
\otimes\tau_{\calQ(X\cap Y'),\calQ(Y\cap X')}\otimes\id_{\calQ(Y\cap Y')}}&&\calQ(X')\otimes\calQ(Y')\\
\calQ(X\cap X')\otimes\calQ(Y\cap X')\otimes\calQ(X\cap Y')\otimes\calQ(Y\cap Y')
\ar[rru]_{\hspace{1.6cm}m_{X\cap X',Y\cap X'}\otimes m_{X\cap Y',Y\cap Y'}}&&}\]
\end{itemize}\end{defi}

Let us consider the example of hypergraphs \cite{Ebrahimi-Fard2022,Foissy44}:

\begin{defi}
Let $X$ be a finite set. A hypergraph $H$ on $X$ is a subset $E(H)$ of $\calP(X)$ which does not contain $\emptyset$. The elements of $E(H)$ are called the hyperedges of $H$. 
The set of hypergraphs on $X$ is denoted by $\rmHG(X)$ and the vector space generated by $\rmHG(X)$ is denoted by $\bfHG(X)$: this defines a species $\bfHG$. 
\end{defi}

The species $\bfHG$ is a twisted bialgebra \cite{Foissy44}; let us recall its structure. Let $X,Y$ be two disjoint sets. If $G\in\rmHG(X)$ and $H\in\rmHG(Y)$, then $GH\in\rmHG(X\sqcup Y)$ is given by
\[E(GH)=E(G)\sqcup E(H).\]
If $G\in\rmHG(X)$ and $X'\subseteq X$, we denote by $G_{\mid X'}\in\rmHG(X')$ the hypergraph defined by
\begin{align*}
E(G_{\mid X'})&=E(G)\cap\calP(X')=\{Y\in E(G)\mid Y\subseteq X'\}.
\end{align*}
The coproduct is then given by
\begin{align*}
&\forall G\in\rmHG(X\sqcup Y),&\Delta_{X,Y}(G)&=G_{\mid X}\otimes G_{\mid Y}.
\end{align*}

Other examples can be found on rooted trees \cite{Calaque2011}, various family of graphs, oriented or not \cite{Manchon2012,Foissy36,Foissy40}, mixed graphs \cite{Foissy45}, noncrossing partitions \cite{Foissy43}, posets and quasi-posets \cite{Foissy37}$\ldots$
\begin{defi}
The bosonic Fock functor $\calF$ is defined in \cite{Aguiar2010}. For any species $\calQ$,
\[\calF(\calQ)=\bigoplus_{n=0}^\infty\frac{\calQ([n])}{\langle x-\calQ(\sigma)(x)\mid\sigma\in\sym_n,\: x\in\calQ([n])\rangle}.\]
The class of $x\in\calQ([n])$ in $\calF(\calQ)$ is denoted by $\overline{x}$ and called the isomorphism class (shortly, isoclass) of $x$.
\end{defi}

If $(\calQ,m,\Delta)$ is a twisted bialgebra, then $\calF(\calQ)$ is a bialgebra. Its product is given by
\begin{align*}
&\forall x\in\calQ([k]),\:\forall y\in\calQ([l]),&m(\overline{x}\otimes\overline{y})
&=\overline{\calQ(\sigma_{k,l})\circ m_{[k],[l]}(x\otimes y)},
\end{align*}
where $\sigma_{k,l}:[k]\sqcup [l]\longrightarrow [k+l]$ is the bijection defined by 
\begin{align*}
\sigma_{k,l}(i)&=\begin{cases}
i\mbox{ if }i\in [k],\\
k+i\mbox{ if }i\in [l].
\end{cases}\end{align*}
The unit is $\overline{1_\calQ}$.
Its coproduct is given by
\begin{align*}
&\forall x\in\calQ([n]),&\Delta(\overline{x})&=\sum_{I\subseteq [n]}\overline{\left(\calQ(\sigma_I)\otimes\calQ(\sigma_{[n]\setminus I})\right)\circ\Delta_{I,[n]\setminus I}(x)},
\end{align*}
where for any $I\subseteq\N$, $\sigma_I:I\longrightarrow [|I|]$ is the unique increasing bijection.\\

\begin{example}
We denote by $\calH_{\bfHG}$ the image of $\bfHG$ by the functor $\calF$: this is the vector space generated by isoclasses of hypergraphs. If $H\in\rmHG(X)$ is a hypergraph, we denote by $\overline{H}$ its isoclass,
that is to say the element $\overline{\rmHG(\sigma)(H)}$ in $\calH_{\bfHG}$, where $\sigma:X\longrightarrow [|X|]$ is any bijection. We obtain that $\calH_{\bfHG}$ is a bialgebra, with
\begin{align*}
\forall G\in\rmHG(X),\:\forall H\in\rmHG(Y),&&\overline{G}\cdot\overline{H}&=\overline{GH},\\
\forall H\in\rmHG(X),&&\Delta(\overline{H})&=\sum_{X=Y\sqcup Z}\overline{H_{\mid Y}}\otimes\overline{H_{\mid Z}},\\
&&\varepsilon_\Delta(\overline{H})&=\delta_{G,1}.
\end{align*}\end{example}

Let us now describe a formalism to obtain double bialgebras, exposed in \cite{Foissy41}. 
Let $(\calQ,m,\Delta)$ be a twisted bialgebra. A contraction-restriction coproduct is a family of maps
\[\delta_\sim:\calQ(X)\longrightarrow\calQ(X/\sim)\otimes\calQ(X),\]
for any finite set $X$ and any $\sim\in\calE(X)$, such that:
\begin{itemize}
\item For any bijection $\sigma:X\longrightarrow Y$ between two finite sets, the following diagram commutes:
\[\xymatrix{\calQ(X)\ar[d]_{\calQ(\sigma)}\ar[r]^{\hspace{-5mm}\delta_\sim}&\calQ(X/\sim)\otimes\calQ(X)\ar[d]^{\calQ(\sigma/\sim)\otimes\calQ(\sigma)}\\
\calQ(Y)\ar[r]_{\hspace{-5mm}\delta_{\sim_\sigma}}&\calQ(Y/\sim_\sigma)\otimes\calQ(Y)}\]
\item (Coassociativity of $\delta$). For any finite set $X$, and any $\sim,\sim'\in\calE(X)$, such that $\sim\subseteq\sim'$,
\[\left(\delta_{\overline{\sim}}\otimes\id_{\calQ(X)}\right)\circ\delta_\sim=\left(\id_{\calQ(X/\sim')}\otimes\delta_\sim\right)\circ\delta_{\sim'}.\]
For any finite set $X$, and any $\sim,\sim'\in\calE(X)$, such that we do not have $\sim\subseteq\sim'$,
\[\left(\id_{\calQ(X/\sim')}\otimes\delta_\sim\right)\circ\delta_{\sim'}=0.\]
\item (Counity). For any finite set $X$, there exists a map $\epsilon_X\in\calQ(X)^*$ such that
\begin{align*}
&\forall\sim\in\calE(X),&\left(\id_{\calQ(X/\sim)}\otimes\epsilon_X\right)\circ\delta_\sim
&=\begin{cases}
\id_{\calQ(X)}\mbox{ if $\sim$ is the equality of $X$},\\
0\mbox{ otherwise}.
\end{cases}\\
&&\sum_{\sim\in\calE(X)}\left(\epsilon_{X/\sim}\otimes\id_{\calQ(X)}\right)\circ\delta_\sim&=\id_{\calQ(X)}.
\end{align*}
\item (Compatibility with the product $m$). For any pair $(X,Y)$ of disjoint finite sets, for any $x\in\calQ(X)$, for any $y\in\calQ(Y)$,
for any $\sim\in\calQ(X\sqcup Y)$, putting $\sim_X=\sim\cap X^2$ and $\sim_Y=\sim\cap Y^2$,
\[\delta_\sim (xy)
=\begin{cases}
\delta_{\sim_X}(x)\delta_{\sim_Y}(y)\mbox{ if }\sim=\sim_X\sqcup\sim_Y,\\
0\mbox{ otherwise}.
\end{cases}\]
Moreover, if $\sim_\emptyset$ is the unique equivalence on $\emptyset$, $\delta_{\sim_\emptyset}(1_\calQ)=1_\calQ\otimes 1_\calQ$. 
\item (Compatibility with the coproduct $\Delta$). For any pair $(X,Y)$ of disjoint finite sets, for any $\sim_X\in\calE(X)$ and $\sim_Y\in\calE(Y)$,
\[\left(\Delta_{X/{\sim_X},Y/{\sim_Y}}\otimes\id_{\calQ(X)}\right)\circ\delta_{\sim_X\sqcup \sim_Y}=
m_{1,3,24}\circ\left(\delta_{\sim_X}\otimes\delta_{\sim_Y}\right)\circ\Delta_{X,Y},\] 
where
\[m_{1,3,24}:\left\{\begin{array}{rcl}
\calQ(X/{\sim_X})\otimes\calQ(X)\otimes\calQ(Y/{\sim_Y})\otimes\calQ(Y)&\longrightarrow&\calQ(X/{\sim_X})\otimes\calQ(Y/{\sim_Y})\otimes\calQ(X\sqcup Y)\\
x_1\otimes x_2\otimes x_3\otimes x_4&\longmapsto&x_1\otimes x_3\otimes m_{X,Y}(x_2\otimes x_4).
\end{array}\right.\]
\item (Compatibility with the counit $\varepsilon_\Delta$). For any $x\in\calQ(\emptyset)$, 
\[\left(\varepsilon_\Delta\otimes\id_{\calQ(\emptyset)}\right)\circ\delta_{\sim_\emptyset}(x)=\varepsilon_\Delta(x)1_\calQ.\]
\end{itemize}

If $(\calQ,m,\Delta)$ is given a such contraction-restriction coproduct, then $\calF(\calQ)$ is a double bialgebra, with the second coproduct $\delta$ defined by
\begin{align*}
&\forall x\in\calQ([n]),&\delta(\overline{x})&=\sum_{\sim\in\calE([n])}
\overline{\left(\calQ(\sigma_\sim)\otimes\id_{\calQ([n])}\right)\circ\delta_\sim(x)},
\end{align*}
where $\sigma_\sim:[n]/\sim\longrightarrow [\cl(\sim)]$ is any bijection. 
The counit of this coproduct $\delta$ is given by 
\begin{align*}
&\forall x\in\calQ([n]),&\epsilon(\overline{x})&=\epsilon_{[n]}(x).
\end{align*}

\begin{example}\label{ex1.3}
The species $\bfHG$ has also a contraction-restriction coproduct.
Let $X$ be a finite set, $H\in\rmHG(X)$, and $\sim\in\calE(X)$. 
We define $H/{\sim}\in\rmHG(X/{\sim})$ and $H\mid\sim\in\rmHG(X)$ by
\begin{align*}
E(H/{\sim})&=\{\varpi_\sim(Y)\mid Y\in E(H)\},&E(H\mid\sim)&=\{Y\in E(H)\mid\forall x,y\in Y,\: x\sim y\}. 
\end{align*}
We shall write that $\sim\in\calE^C(H)$ if for any $Y\in X/{\sim}$, $H_{\mid Y}$ is connected. Then
\[\delta_\sim(H)=\begin{cases}
H/{\sim}\otimes H\mid\sim\mbox{ if }\sim\in\calE^C(H),\\
0\mbox{ otherwise}.
\end{cases}\]
Its counit $\epsilon_\delta$ is given by
\begin{align*}
&\forall H\in\rmHG(X),&\epsilon_\delta(H)&=\begin{cases}
1\mbox{ if }\forall Y\in E(H),\: |Y|=1,\\
0\mbox{ otherwise}.
\end{cases}\end{align*}
Then $\calH_{\bfHG}$ is a double bialgebra, with
\begin{align*}
\forall H\in\rmHG(X),&&\delta(\overline{H})&=\sum_{\sim\in\calE^C(H)}\overline{H/{\sim}}\otimes\overline{H\mid\sim},\\
&&\epsilon_\delta(\overline{H})&=\begin{cases}
1\mbox{ if }\forall Y\in E(H),\: |Y|=1,\\
0\mbox{ otherwise}.
\end{cases}\end{align*}
The unique double bialgebra morphism from $\calH_{\bfHG}$ to $\K[T]$ is denoted by$\Phichr$. For any hypergraph $H$ and for any $n\in\N$,
$\Phichr(\overline{H})(n)$ is the number of maps $c:V(G)\longrightarrow [n]$ such that for any $Y\in E(H)$, of cardinality $\geq 2$, $c_{\mid Y}$ is not constant. 
Another way to determine $\Phichr$ is to characterize it as the unique bialgebra morphism from $(\calH_{\bfHG},m,\Delta)$ to $(\K[T],m,\Delta)$ such that for any hypergraph $H$, $\Phichr(\overline{H})(1)=\epsilon_\delta(H)$.
\end{example}

\section{Hopf-algebraic structures on boolean functions}

\subsection{A two-parameters family of products}

\begin{defi}
Let $X$ be a finite set. A boolean function on $X$ is a map $f:\calP(X)\longrightarrow\Z$, such that $f(\emptyset)=0$.
The set of boolean functions on $X$ is denoted by $\rmbool(X)$.
\end{defi}

This defines a set species $\rmbool$: if $\sigma:X\longrightarrow Y$ is a bijection, we define a map 
\[\rmbool(\sigma):\left\{\begin{array}{rcl}
\rmbool(X)&\longrightarrow&\rmbool(Y)\\
f&\longmapsto&\rmbool(\sigma)(f):\left\{\begin{array}{rcl}
\calP(Y)&\longrightarrow&\Z\\
A&\longmapsto&f(\sigma^{-1}(A)).
\end{array}\right.\end{array}\right.\]
Its linearization is denoted by $\bfbool$: for any finite set $X$, $\bfbool(X)$ is the space generated by $\rmbool(X)$.

\begin{theo}\label{theo2.2}
Let $X, Y$ be two disjoint finite sets, $f\in\rmbool(X)$ and $g\in\rmbool(Y)$. 
\begin{enumerate}
\item Let $q_1,q_2\in\Z$. 
We define $f\star_{q_1,q_2}g$ by
\begin{align*}
&\forall A\subseteq X\sqcup Y,&f\star_{q_1,q_2} g(A)&=q_1^{|A\cap Y|}f(A\cap X)+q_2^{|A\cap X|}g(A\cap Y).
\end{align*}
The bilinear extension of $\star_{q_1,q_2}$ to $\bfbool$ makes it an associative twisted algebra. The unit is the unique boolean function $1\in\rmbool(\emptyset)$.
It is commutative if, and only if, $q_1=q_2$. In this case, we shall simply take $q=q_1=q_2$ and write $\star_q$ instead of $\star_{q,q}$. 
\item Let $q\in\Z$. We consider the species morphism $\theta_q:\bfbool\longrightarrow\bfbool$ defined by
\begin{align*}
&\forall f\in\rmbool(X),\:\forall A\subseteq X,&\theta_q(f)(A)&=\sum_{B\subseteq A}q^{|A|-|B|}f(B).
\end{align*}
Then:
\begin{enumerate}
\item For any $q,q'\in\Z$, $\theta_q\circ\theta_{q'}=\theta_{q+q'}$.
\item For any $q_1,q_2,q'_1,q'_2\in\Z$, $\theta_q$ is a twisted algebra isomorphism from $(\bfbool,\star_{q_1,q_2})$ to $(\bfbool,\star_{q_1+q,q_2+q})$. 
\end{enumerate}\end{enumerate}\end{theo}

\begin{example}
Let $X, Y$ be two disjoint finite sets, $f\in\rmbool(X)$ and $g\in\rmbool(Y)$. Then
\begin{align*}
&\forall A\subseteq X\sqcup Y,&f\star_1 g(A)&=f(A\cap X)+g(A\cap Y),\\
&&f\star_0 g(A)&=\begin{cases}
f(A)\mbox{ if }A\subseteq X,\\
g(A)\mbox{ if }A\subseteq Y,\\
0\mbox{ otherwise},
\end{cases}\\
&&f\star_{0,1}g(A)&=\begin{cases}
f(A)\mbox{ if }A\subseteq X,\\
g(A\cap Y)\mbox{ otherwise}.
\end{cases}\end{align*}\end{example}

\begin{proof}
Observe firstly that $f\star_{q_1,q_2} g$ is indeed a boolean function: $f\star_{q,_1,q_2} g(\emptyset)=0$, as $f(\emptyset)=g(\emptyset)=0$.\\

1. Let $X,Y,Z$ be pairwise disjoint sets, $f\in\rmbool(X)$, $g\in\rmbool(Y)$ and $h\in\rmbool(Z)$. For any $A\subseteq X\sqcup Y\sqcup Z$,
\begin{align*}
(f\star_{q_1,q_2} g)\star_{q_1,q_2} h(A)&=q_1^{|A\cap Y|+|A\cap Z|}f(A\cap X)+q_1^{|A\cap X|}q_2^{|A\cap Z|}g(A\cap Y)
+q_2^{|A\cap X|+|A\cap Y|}h(A\cap Z)\\
&=f\star_{q_1,q_2} (g\star_{q_1,q_2} h)(A),
\end{align*}
so $\star_{q_1,q_2}$ is associative.\\

Obviously, the opposite of $\star_{q_1,q_2}$ is $\star_{q_2,q_1}$. Consequently, if $q_1=q_2$, $\star_{q_1,q_2}$ is commutative.
Let $f\in\rmbool(\{1\})$, defined by $f(\{1\})=1$ and $g\in\rmbool(\{2\})$, defined by $g(\{2\})=0$. Then 
\begin{align*}
f\star_{q_1,q_2} g(\{1,2\})&=q_1f(\{1\})+q_2g(\{2\})=q_1,&
g\star_{q_1,q_2} f(\{1,2\})&=q_1g(\{1\})+q_2f(\{2\})=q_2.
\end{align*}
Therefore, if $\star_{q_1,q_2}$ is commutative, then $q_1=q_2$.\\

2. (a) Let $X$ be a finite set, $f\in\rmbool(X)$, and $q,q'\in\Z$. For any $A\subseteq X$,
\begin{align*}
\theta_q\circ\theta_{q'}(f)(A)&=\sum_{C\subseteq A} q^{|A|-|C|}\theta_{q'}(f)(C)\\
&=\sum_{B\subseteq C\subseteq A} q^{|A|-|C|} q'^{|C|-|B|} f(B)\\
&=\sum_{B\subseteq A}\left(\sum_{B\subseteq C\subseteq A} q^{|A|-|C|} q'^{|C|-|B|}\right) f(B)\\
&=\sum_{B\subseteq A}\left(\sum_{k=0}^{|A|-|B|}\binom{|A|-|B|}{k}q^k q'^{|A|-|B|-k}\right)f(B)\\
&=\sum_{B\subseteq A}(q+q')^{|A|-|B|}f(B)\\
&=\theta_{q+q'}(f)(A),
\end{align*}
so $\theta_q\circ\theta_{q'}=\theta_{q+q'}$. For the fourth equality, $k=|A|-|C|$.\\

Let $X$ be a finite set, and $f\in\rmbool(X)$. For any $A\subseteq X$,
\[\theta_0(f)(A)=\sum_{B\subseteq A} 0^{|A|-|B|} f(B)=f(A)+0,\]
so $\theta_0=\id_{\rmbool}$. As a consequence, for any $q\in\Z$, $\theta_q$ is invertible, of inverse $\theta_{-q}$.\\

2. (b) Let $X,Y$ be two disjoint finite sets, $f\in\rmbool(X)$, $g\in\rmbool(Y)$. For any $A\subseteq X\sqcup Y$,
\begin{align*}
\theta_q(f\star_{q_1,q_2} g)&=\sum_{B\subseteq A} q^{|A|-|B|}f\star_{q_1,q_2} g(A)\\
&=\sum_{B\subseteq A}\left(q^{|A|-|B|}q_1^{|B\cap Y|} f(B\cap X)+ q^{|A|-|B|} q_2^{|B\cap X|}g(B\cap Y)\right)\\
&=\sum_{\substack{B'\subseteq A\cap X,\\ B''\subset A\cap Y}} q^{|A|-|B'|-|B''|}
\left(q_1^{|B''|}f(B')+q_2^{|B'|}g(B'')\right)\\
&=\sum_{B'\subseteq A\cap X}\left(\sum_{B''\subseteq A\cap Y}q^{|A|-|B'|-|B''|}q_1^{|B''|}\right)f(B')\\
&+\sum_{B''\subseteq A\cap Y}\left(\sum_{B'\subseteq A\cap X}q^{|A|-|B'|-|B''|}q_2^{|B'|}\right)g(B'')\\
&=\sum_{B'\subseteq A\cap X}q^{|A\cap X|-|B'|}\left(\sum_{B''\subseteq A\cap Y}q^{|A\cap Y|-|B''|}q_1^{|B''|}\right)f(B')\\
&+\sum_{B''\subseteq A\cap Y}q^{|A\cap Y|-|B''|}\left(\sum_{B'\subseteq A\cap Y}q^{|A\cap X|-|B'|}q_2^{|B'|}\right)g(B'')\\
&=\sum_{B'\subseteq A\cap X}q^{|A\cap X|-|B'|}(q_1+q)^{|A\cap Y|}f(B')\\
&+\sum_{B''\subseteq A\cap Y}q^{|A\cap Y|-|B''|}(q_2+q)^{|A\cap X|}g(B'')\\
&=(q_1+q)^{|A\cap Y|}\theta_q(f)(A\cap X)+(q_2+q)^{|A\cap X|}\theta_q(fg(A\cap Y)\\
&=\theta_q(f)\star_{q_1+q,q_2+q}\theta_q(g)(A).
\end{align*}
For the third equality, $B'=B\cap X$ and $B''=B\cap Y$. 
So $\theta_q(f\star_{q_1,q_2} g)=\theta_q(f)\star_{q_1+q,q_2+q}\theta_q(g)$.
\end{proof}

\begin{notation}
For any $q\in\Z$, we shall use the symbols $\displaystyle\prod^{\star_q}$ for an iterated product of boolean functions using $\star_q=\star_{q,q}$. As it is commutative, the order of the factors does not matter.
\end{notation}

\begin{remark}
As a consequence, if $q_1-q_2=q'_1-q'_2$, then $(\bfbool,\star_{q_1,q_2})$ and $(\bfbool,\star_{q'_1,q'_2})$ are isomorphic, through the morphism $\theta_{q'_1-q_1}$. 
In particular, for any $q,q'\in\Z$, $(\bfbool,\star_q)$ and $(\bfbool,\star_{q'})$ are isomorphic, through $\theta_{q'-q}$. 
\end{remark}

\subsection{Indecomposable boolean functions}

We here study these products $\star_{q_1,q_2}$, before defining the coproduct on boolean functions. 

\begin{defi}
Let $X\subseteq Y$ be two finite sets. For any $f\in\rmbool(Y)$, we shortly write $f_{\mid X}$ instead of $f_{\mid\calP(X)}$.
Observe that $f_{\mid X}\in\rmbool(X)$.
\end{defi}

\begin{lemma}\label{lemma2.4}
Let $X,Y$ be two disjoint finite sets, $q_1,q_2\in\Z$, $f\in\rmbool(X)$, $g\in\rmbool(Y)$ and $A\subseteq X\sqcup Y$. Then
\begin{align*}
(f\star_{q_1,q_2} g)_{\mid A}&=f_{\mid A\cap X}\star_{q_1,q_2} g_{\mid A\cap Y}.
\end{align*}
In particular, 
\begin{align*}
(f\star_{q_1,q_2} g)_{\mid X}&=f,&(f\star_{q_1,q_2} g)_{\mid Y}&=g.
\end{align*}
\end{lemma}

\begin{proof}
Let $B\subseteq A$. 
\begin{align*}
(f\star_{q_1,q_2} g)_{\mid A}(B)&=q_1^{|B\cap Y|}f(B\cap X)+q_2^{|B\cap X|}g(B\cap Y)\\
&=q_1^{|B\cap A\cap Y|}f_{\mid A\cap X}(B\cap X)+q_2^{|B\cap A\cap X|}g_{\mid A\cap Y}(B\cap Y)\\
&=f_{\mid A\cap X}\star_{q_1,q_2} g_{\mid A\cap Y}(B),
\end{align*}
so $(f\star_{q_1,q_2} g)_{\mid A}=f_{\mid A\cap X}\star_{q_1,q_2} g_{\mid A\cap Y}$. In particular, if $A=X$,
\[(f\star_{q_1,q_2} g)_{\mid X}=f_{\mid X}\star_{q_1,q_2} g_{\mid\emptyset}=f\star_{q_1,q_2}1_{\bfbool}=f.\]
Similarly, $(f\star_{q_1,q_2} g)_{\mid Y}=g$.
\end{proof}

\begin{defi}
Let $X$ be a nonempty finite set and $f\in\rmbool(X)$. We shall say that $f$ is $(q_1,q_2)$-indecomposable if for any $Y\subseteq X$, for any $f'\in\rmbool[X\setminus Y]$, and for any $f''\in\rmbool[Y]$,
\[f=f'\starq f''\Longrightarrow Y=\emptyset\mbox{ or }Y=X.\]
\end{defi}

If $|X|=1$, then any $f\in\rmbool(X)$ is obviously $(q_1,q_2)$-indecomposable, for any $(q_1,q_2)\in\Z^2$.
Let us consider the case where $|X|=2$.

\begin{prop}\label{prop2.6}
Let $X=\{x,y\}$ be a set of cardinality two, $f\in\rmbool(X)$ and $(q_1,q_2)\in\Z^2$. Then $f$ is $(q_1,q_2)$-indecomposable if, and only if, 
\begin{align*}
&f(\{x,y\})\neq q_1f(\{x\})+q_2f(\{y\})\mbox{ and }f(\{x,y\})\neq q_1f(\{y\})+q_2f(\{x\}).
\end{align*}\end{prop}

\begin{proof}
As $|X|=2$, $f$ is $(q_1,q_2)$-decomposable if, and only if, $f=f'\starq f''$ or $f''\starq f'$, for certain $f'\in\rmbool(\{x\})$ and $f''\in\rmbool(\{y\})$. 
By Lemma \ref{lemma2.4}, necessarily $f'=f_{\mid\{x\}}$ and $f''=f_{\mid\{y\}}$. Therefore, 
\begin{align*}
f'\starq f''(\emptyset)&=0,&f''\starq f'(\emptyset)&=0,\\
f'\starq f''(\{x\})&=f(\{x\}),&f''\starq f'(\{x\})&=f(\{x\}),\\
f'\starq f''(\{y\})&=f(\{y\}),&f''\starq f'(\{y\})&=f(\{y\}),\\
f'\starq f''(\{x,y\})&=q_1f(\{x\})+q_2f(\{y\}),&f''\starq f'(\{x,y\})&q_1f(\{y\})+q_2f(\{x\}).
\end{align*}
Therefore, 
\begin{align*}
f=f'\starq f''&\Longleftrightarrow f(\{x,y\})=q_1f(\{x\})+q_2f(\{y\}),\\
f=f''\starq f'&\Longleftrightarrow f(\{x,y\})=q_1f(\{y\})+q_2f(\{x\}).
\end{align*}
The result follows.
\end{proof}

\begin{example}
Let $f\in\rmbool(\{x,y\})$.
\begin{align*}
f\mbox{ is $(q,q)$-indecomposable}&\Longleftrightarrow f(\{x,y\})\neq q\left(f(\{x\})+f(\{y\})\right),\\
f\mbox{ is $(0,1)$-indecomposable}&\Longleftrightarrow f(\{x,y\})\notin\left\{ f(\{x\}), f(\{y\})\right\}.
\end{align*}\end{example}

\begin{lemma}\label{lemma2.7}
Let $X$ be a nonempty finite set and $f\in\rmbool(X)$. There exists a composition $(X_1,\ldots,X_k)$ of $X$ such that:
\begin{itemize}
\item For any $i\in [k]$, $f_{\mid X_i}$ is $(q_1,q_2)$-indecomposable.
\item $f=f_{\mid X_1}\starq\cdots\starq f_{\mid X_k}$.
\end{itemize}
Such a decomposition will be called a decomposition of $f$ into $(q_1,q_2)$-indecomposable components.
\end{lemma}

\begin{proof}
Note that if $f$ is $(q_1,q_2)$- indecomposable, the result holds with the composition $(X)$. We proceed by induction on $|X|$. The indecomposable case covers the case $|X|=1$. 
Let us assume the result at all ranks $<|X|$. If $f$ is $(q_1,q_2)$-indecomposable, we take the composition $(X)$. Otherwise, there exists a composition $(X_1,X_2)$ such that $f=f_1\starq f_2$, for a pair $(f_1,f_2)\in\rmbool(X_1)\times\rmbool(X_2)$. 
By Lemma \ref{lemma2.4}, $f_i=f_{\mid X_i}$ for $i\in [2]$. 

We apply the induction hypothesis on $f_1$ and $f_2$; we obtain two compositions $(X_{11},\ldots,X_{1k_1})$ and $(X_{21},\ldots,X_{2k_2})$ of respectively $X_1$ and $X_2$, such that, for any $i\in [2]$, for any $j\in [k_i]$, 
\[f_{ij}=\left(f_i\right)_{\mid X_{i,j}}=\left(f_{\mid X_i}\right)_{\mid X_{i,j}}=f_{\mid X_{i,j}}\]
is $(q_1,q_2)$-indecomposable,
\begin{align*}
f_1&=f_{11}\starq\cdots\starq f_{1k_1},&f_2&=f_{21}\starq\cdots\starq f_{2k_2}.
\end{align*}
Then
\[f=f_{11}\starq\cdots f_{1k_1}\starq f_{21}\starq\cdots f_{2k_2}.\]
We then take the composition $(X_{11},\ldots,X_{1k_1},X_{21},\ldots,X_{2k_2})$.\end{proof}

This decomposition is not unique, but it is up to a permutation.

\begin{prop}\label{prop2.8}
Let $X$ be a nonempty finite set, and $f\in\rmbool(X)$. Let us consider two decomposition of $f$ into $(q_1,q_2)$-indecomposable components:
\begin{align*}
f&=f_{\mid X_1}\starq\cdots\starq f_{\mid X_k}=f_{\mid Y_1}\starq\cdots\starq f_{\mid Y_l}.
\end{align*}
Then $k=l$ and there exists $\sigma\in\sym_k$, such that $Y_i=X_{\sigma(i)}$ for any $i\in [k]$.
\end{prop}

\begin{proof}
We proceed by induction on $|X|$. If $|X|=1$, then $f$ is $(q_1,q_2)$-indecomposable: necessarily, $k=l=1$, and $X_1=Y_1=X$. Let us assume the result at all ranks $<|X|$.
Restricting to $Y_l$, with Lemma \ref{lemma2.4},
\[f_{\mid Y_l}=f_{\mid X_1\cap Y_l}\starq\cdots\starq f_{\mid X_k\cap Y_l}.\]
As $f_{\mid Y_l}$ is $(q_1,q_2)$-indecomposable, there exists a unique $\sigma(l)\in [k]$, such that $Y_l\subseteq X_{\sigma(l)}$. By symmetry between the $X$'s and the $Y$'s, $Y_l=X_{\sigma(l)}$.
Restricting to 
\[Y_1\sqcup\cdots\sqcup Y_{l-1}=X_1\sqcup\cdots\sqcup X_{\sigma(l)-1}\sqcup X_{\sigma(l)+1}\sqcup\cdots\sqcup X_k,\]
we obtain with Lemma \ref{lemma2.4} that
\[f_{\mid Y_1}\starq\cdots\starq f_{\mid Y_{l-1}}=f_{\mid X_1}\starq\cdots\starq f_{\mid X_{\sigma(l)-1}}\starq f_{\mid X_{\sigma(l)+1}}\starq\cdots\starq f_{\mid X_k}.\]
The induction hypothesis gives that $k-1=l-1$ (so $k=l$) and the image of the permutation $\sigma$ for the values in $[k-1]$. 
\end{proof}

\begin{prop}
Let $X$ be a finite set, $f\in\rmbool(X)$, which we decompose into $(q_1,q_2)$-indecomposable components into two different ways:
\[f=f_1\starq\cdots\star f_k=f_{\sigma(1)}\starq\cdots\starq f_{\sigma(k)}.\]
Then one can go from the first one to the second one by successive permutations of consecutive commuting factors.
\end{prop}

\begin{proof}
Let us consider a decomposition $\sigma=\tau_1\circ\cdots\circ\tau_p$ of $\sigma$ as a composition of transpositions $(j\: j+1)$, with $j\in [k-1]$, of minimal length. Let us put $\tau_p=(i,\: i+1)$.
As this is a decomposition of minimal length, $\sigma(i)>\sigma(i+1)$. Therefore, 
$f_{i+1}$ appears before $f_i$ in the second decomposition. By Lemma \ref{lemma2.4},
\begin{align*}
f_{\mid X_i\sqcup X_{i+1}}&=(f_1)_{\mid\emptyset}\starq\cdots\starq (f_i)_{\mid X_i}\starq (f_{i+1})_{\mid X_{i+1}}\starq\cdots\starq (f_k)_{\mid\emptyset}\\
&=f_i\starq f_{i+1}.
\end{align*}
On the other hand,
\begin{align*}
f_{\mid X_i\sqcup X_{i+1}}&=(f_{\sigma(1)})_{\mid\emptyset}\starq\cdots\starq (f_{i+1})_{\mid X_{i+1}}\starq\cdots\starq (f_i)_{\mid X_i}\starq\cdots\starq (f_{\sigma(k)})_{\mid\emptyset}\\
&=f_{i+1}\starq f_i.
\end{align*}
So $f_i$ and $f_{i+1}$ commute, and, consequently,
\[f=f_1\starq\cdots\starq f_{i+1}\starq f_i\starq\ldots\starq f_k.\]
Iterating the process, we obtain a way to go from the first decomposition to the second one by using 
commutation relations between two consecutive factors.
\end{proof}

As a consequence:

\begin{cor}
As a twisted algebra, $(\bfbool,\starq)$ is generated by $(q_1,q_2)$-indecomposable elements, with the commutation relations between them. 
\end{cor}

Let us now describe the commutation relations. When $q_1=q_2$, the product is commutative. We now consider the case $q_1\neq q_2$. 

\begin{defi}
We assume that $q_1\neq q_2$. Let $X$ be a finite set and $\lambda\in\Z$. We define $f_{X,\lambda}^{(q_1,q_2)}\in\rmbool(X)$ by
\begin{align*}
&\forall A\subseteq X,&f_{X,\lambda}^{(q_1,q_2)}(A)&=\lambda\dfrac{q_1^{|A|}-q_2^{|A|}}{q_1-q_2}.
\end{align*}\end{defi}

\begin{remark}
For any $n\geq 1$,
\[\dfrac{q_1^n-q_2^n}{q_1-q_2}=q_1^{n-1}+q_1^{n-2}q_2+\cdots+q_1q_2^{n-2}+q_2^{n-1}\in\Z.\]
So, indeed, $f_{X,\lambda}^{(q_1,q_2)}\in\rmbool(X)$.
\end{remark}

\begin{lemma}
Let $X,Y$ be two disjoint finite sets, and $\lambda\in\Z$. Then
\[f^{(q_1,q_2)}_{X,\lambda}\starq f^{(q_1,q_2)}_{Y,\lambda}=f^{(q_1,q_2)}_{X\sqcup Y,\lambda}.\]
\end{lemma}

\begin{proof}
Let $A\subseteq X\sqcup Y$.
\begin{align*}
ff^{(q_1,q_2)}_{X,\lambda}\starq f^{(q_1,q_2)}_{Y,\lambda}(A)&=
\lambda q_1^{|A\cap Y|}\dfrac{q_1^{|A\cap X|}-q_2^{|A\cap X|}}{q_1-q_2}+
\lambda q_2^{|A\cap X|}\dfrac{q_1^{|A\cap Y|}-q_2^{|A\cap Y|}}{q_1-q_2}\\
&=\lambda\dfrac{q_1^{|A\cap X|+|A\cap Y|}-q_2^{|A\cap X|+|A\cap Y|}}{q_1-q_2}\\
&=\lambda\dfrac{q_1^{|A|}-q_2^{|A|}}{q_1-q_2}\\
&=f^{(q_1,q_2)}_{X\sqcup Y,\lambda}(A).\qedhere
\end{align*}
\end{proof}

\begin{remark}
As direct consequences:
\begin{enumerate}
\item If $X$ and $Y$ are disjoint finite sets, and $\lambda\in\Z$, $f^{(q_1,q_2)}_{X,\lambda}$ and $f^{(q_1,q_2)}_{Y,\lambda}$ $\starq$-commute.
\item $f^{(q_1,q_2)}_{X,\lambda}$ is $(q_1,q_2)$-indecomposable if, and only if, $|X|=1$.
\end{enumerate}\end{remark}

\begin{prop}
We assume that $q_1\neq q_2$. Let $X,Y$ be two nonempty disjoint finite sets, $f\in\rmbool(X)$ and $g\in\rmbool(Y)$, such that $f\starq g=g\starq f$. Then, there exists $\lambda\in\Z$ such that
$f=f^{(q_1,q_2)}_{X,\lambda}$ and $g=f^{(q_1,q_2)}_{Y,\lambda}$. 
\end{prop}

\begin{proof}
Let us choose $y\in Y$. For any $A\subseteq X$,
\begin{align*}
f\starq g(A\sqcup\{y\})&=q_1 f(A)+q_2^{|A|}g(\{y\})\\
=g\starq f(A\sqcup\{y\})&=q_2 f(A)+q_1^{|A|}g(\{y\}).
\end{align*}
As $q_1\neq q_2$, we obtain that for any $A\subseteq X$,
\[f(A)=g(\{y\})\dfrac{q_1^{|A|}-q_2^{|A|}}{q_1-q_2},\]
so $f=f^{(q_1,q_2)}_{X,\lambda}$, with $\lambda=g(\{y\})$. Similarly, choosing $x\in X$ and putting $\mu=f(\{x\})$,
$g=f^{(q_1,q_2)}_{Y,\mu}$. Moreover, for $A=\{x\}$, we obtain
\[f(\{x\})=\mu=\lambda\dfrac{q_1-q_2}{q_1-q_2}=\lambda,\]
so $\lambda=\mu$. 
\end{proof}

\begin{cor}
If $q_1\neq q_2$, as a twisted algebra, $(\bfbool,\starq)$ is generated by $(q_1,q_2)$-indecomposable boolean maps, with the relations 
\[f^{(q_1,q_2)}_{X,\lambda}\starq f^{(q_1,q_2)}_{Y,\lambda}=f^{(q_1,q_2)}_{Y,\lambda}\starq f^{(q_1,q_2)}_{X,\lambda},\]
where $\lambda\in\Z$ and $X,Y$ are two distinct singletons. 
\end{cor}

\subsection{The commutative case $q_1=q_2=1$}

When $q_1=q_2$, by commutativity, we can permute the order in the decomposition into $(q,q)$-indecomposable components in any way. 
As the algebras $(\calH_{\bfbool},\star_q)$ are all isomorphic trough the isomorphisms $\theta_q$, we restrict ourselves to the case $q_1=q_2=1$. We shall simply write indecomposable instead of $(1,1)$-indecomposable.
By Lemma \ref{lemma2.7} and Proposition \ref{prop2.8}:

\begin{prop}
Let $X$ be a nonempty set and let $f\in\rmbool(X)$. There exists a unique $\sim\in\calE(X)$ such that:
\begin{itemize}
\item $\displaystyle f=\prod_{Y\in X/{\sim}}^{\star_1} f_{\mid Y}$.
\item For any $Y\in X/{\sim}$, $f_{\mid Y}$ is indecomposable.
\end{itemize}
This equivalence on $X$ is denoted by $\sim_f^i$. The number of indecomposable components of $f$ is denoted by $\cl(\sim_f^i)=\ic(f)$. The elements of $X/{\sim}_f^i$ are called the indecomposable components of $f$. 
\end{prop}

\begin{remark}\begin{enumerate}
\item Obviously, if $f\in\rmbool(X)$ and $g\in\rmbool(Y)$, where $X$ and $Y$ are disjoint finite sets, 
then $\ic(f\star_1 g)=\ic(f)+\ic(g)$, and $\sim_{f\star_1 g}^i=\sim_f^i\sqcup\sim_g^i$. 
\item If $f\in\rmbool(X)$, with $X$ nonempty, then $\ic(f)\geq 1$, and $\ic(1)=0$ by convention. 
\end{enumerate}\end{remark}

\begin{defi}\label{defi2.16}
Let $X$ be a finite set and $f\in\rmbool(X)$. The following conditions are equivalent:
\begin{enumerate}
\item $\ic(f)=|X|$.
\item $\sim_f^i$ is the equality $=_X$ of $X$.
\item The indecomposable components of $f$ are the singletons of $X$.
\item For any $A\subseteq X$, $\displaystyle f(A)=\sum_{x\in A} f(\{x\})$.
\item For any $A,B\subseteq X$, $f(A\cup B)=f(A)+f(B)-f(A\cap B)$.
\end{enumerate}
Such a boolean function is called modular.
\end{defi}

\begin{proof}
Left to the reader.
\end{proof}

The two following propositions give sufficient criteria to determine indecomposable boolean functions:

\begin{prop}\label{prop2.17}
Let $X$ be a finite set and $f\in\rmbool(X)$. We assume that there exists $x\in X$ such that
\begin{align*}
&\forall y\in X\setminus\{x\},&f(\{x,y\})&\neq f(\{x\})+f(\{y\}).
\end{align*}
Then $f$ is indecomposable.
\end{prop}

\begin{proof}
Let us assume that $f$ is decomposable. There exists $\{X_1,X_2\}\in\parti(X)$ such that $f=f_{\mid X_1}\star_1 f_{\mid X_2}$. Up to a permutation of $X_1$ and $X_2$, let us assume that $x\in X_1$, 
and let us choose $y\in X_2$. Then
\[f(\{x,y\})=f_{\mid X_1}(\{x\})+f_{\mid X_2}(\{y\})=f(\{x\})+f(\{y\}).\]
This is a contradiction with the hypothesis on $x$. So $f$ is indecomposable. 
\end{proof}
 
The converse is false, as shown by the following example.

\begin{example}
Let $f\in\rmbool([3])$ such that
\begin{align*}
&\forall X\subseteq [3],&f(X)&=\begin{cases}
1\mbox{ if }X=[3],\\
0\mbox{ otherwise}.
\end{cases}\end{align*}
The incoming Proposition \ref{prop2.18} shows that $f$ is indecomposable. However, for any $x,y\in [3]$, with $x\neq y$, $f(\{x,y\})=f(\{x\})+f(\{y\})=0$.
\end{example}

\begin{prop}\label{prop2.18}
\begin{enumerate}
\item Let $X=\{x,y\}$ be a set of cardinality 2 and $f\in\rmbool(X)$. Then
\begin{align*}
f\mbox{ is indecomposable}&\Longleftrightarrow f(\{x,y\})\neq f(\{x\})+f(\{y\}).
\end{align*}
\item Let $X=\{x,y,z\}$ be a set of cardinality 3 and $f\in\rmbool(X)$. Then 
\begin{align*}
f\mbox{ is indecomposable}&\Longleftrightarrow\begin{cases}
\left(\begin{array}{c}
f(\{x,y\})\neq f(\{x\})+f(\{y\})\\
\mbox{ or }f(\{x,z\})\neq f(\{x\})+f(\{z\})\\
\mbox{ or }f(\{x,y,z\})\neq f(\{x\})+f(\{y,z\})
\end{array}\right)\\
\mbox{and }\left(\begin{array}{c}
f(\{x,y\})\neq f(\{x\})+f(\{y\})\\
\mbox{ or }f(\{y,z\})\neq f(\{y\})+f(\{z\})\\
\mbox{ or }f(\{x,y,z\})\neq f(\{y\})+f(\{x,z\})
\end{array}\right)\\
\mbox{and }\left(\begin{array}{c}
f(\{x,z\})\neq f(\{x\})+f(\{z\})\\
\mbox{ or }f(\{y,z\})\neq f(\{y\})+f(\{z\})\\
\mbox{ or }f(\{x,y,z\})\neq f(\{z\})+f(\{x,y\})
\end{array}\right).\\
\end{cases}\end{align*}\end{enumerate}\end{prop}

\begin{proof}
1. Comes from Proposition \ref{prop2.6}, with $q_1=q_2=1$.\\

2. $\Longrightarrow$. Let $X_1,X_2$ be two nonempty subsets such that $X=X_1\sqcup X_2$, with $|X_1|=1$ and $|X_2|=2$.
Then $f\neq f_{\mid X_1}\star_1 f_{\mid X_2}$. As they coincide on subsets of $X_1$ and $X_2$, there exists $Y$ such that $X_1\subseteq Y\subseteq X$, with $f(Y)\neq f(X_1)+f(X_2\cap Y)$. 
If $X_1=\{x\}$ and $X_2=\{y,z\}$, then $Y=\{x,y\}$, $\{x,z\}$ or $\{x,y,z\}$, and we obtain the first parenthesis of conditions; if $X_2=\{y\}$ and $X_2=\{z\}$, we obtain the two other parentheses of conditions. 

$\Longleftarrow$. Let us assume that there exists $\{X_1,X_2\}\in\parti(X)$ such that $f\neq f_{\mid X_1}\star_1 f_{\mid X_2}$. 
 Necessarily one of these two sets, say for example $X_1$, is a singleton. The first parenthesis of conditions gives that $X_1\neq\{x\}$, the second one one that $X_1\neq\{y\}$ and the last one that $X_1\neq\{z\}$.
So $f$ is indecomposable. 
\end{proof}

\subsection{The restriction coproduct}

\begin{theo}\label{theo2.19}
We define a coproduct $\Delta$ on $\bfbool$ as follows: if $X,Y$ are disjoint finite sets,
\begin{align*}
&\forall f\in\rmbool(X\sqcup Y),&\Delta_{X,Y}(f)&=f_{\mid X}\otimes f_{\mid Y}.
\end{align*} 
\begin{enumerate}
\item For any $q_1,q_2\in\Z$, $(\bfbool,\star_{q_1,q_2},\Delta)$ is a cocommutative twisted bialgebra. 
\item For any $q,q_1,q_2\in\Z$, $\theta_q$ is a twisted bialgebra isomorphism from $(\bfbool,\star_{q_1,q_2},\Delta)$ to $(\bfbool,\star_{q_1+q,q_2+q},\Delta)$.
\end{enumerate}\end{theo}

\begin{proof}
1. Let us prove the coassociativity of $\Delta$. Let us consider three pairwise disjoint sets $X,Y,Z$ and $f\in\rmbool(X\sqcup Y\sqcup Z)$.
\begin{align*}
\left(\Delta_{X,Y}\otimes\id_{\bfbool(Z)}\right)\circ\Delta_{X\sqcup Y,Z}(f)&=(f_{\mid X\sqcup Y})_{\mid X}\otimes (f_{\mid X\sqcup Y})_{\mid Y}\otimes f_{\mid Z}\\
&=f_{\mid X}\otimes f_{\mid Y}\otimes f_{\mid Z}\\
&=f_{\mid X}\otimes (f_{\mid Y\sqcup Z})_{\mid Y}\otimes (f_{\mid Y\sqcup Z})_{\mid Z}\\
&=\left(\id_{\bfbool(X)}\otimes\Delta_{Y,Z}\right)\circ\Delta_{X,Y\sqcup Z}(f).
\end{align*}
so $\Delta$ is coassociative. For any $f\in\rmbool(X)$,
\begin{align*}
\Delta_{\emptyset,X}(f)&=1\otimes f,&\Delta_{X,\emptyset}(f)&=f\otimes 1.
\end{align*}
So $\Delta$ is counitary, and the counit $\varepsilon_\Delta$ sends $1\in\rmbool(\emptyset)$ to $1$. Moreover, if $X,Y$ are two disjoint sets and $f\in\rmbool(X\sqcup Y)$,
\begin{align*}
\tau_{\bfbool(X),\bfbool(Y)}\circ\Delta_{X,Y}(f)&=f_{\mid Y}\otimes f_{\mid X}=\Delta_{Y,X}(f),
\end{align*}
so $\Delta$ is cocommutative.\\

Let $X,Y,X',Y'$ be four finite sets such that $X\sqcup Y=X'\sqcup Y'$. For any $f\in\rmbool(X)$ and $g\in\rmbool(Y)$. By Lemma \ref{lemma2.4},
\begin{align*}
\Delta_{X',Y'}(f\star_{q_1,q_2} g)&=(f\star_{q_1,q_2} g)_{\mid X'}\otimes (f\star_{q_1,q_2} g)_{\mid Y'}\\
&=f_{\mid X\cap X'}\star_{q_1,q_2} g_{\mid Y\cap X'}\otimes f_{\mid X\cap Y'}\star_{q_1,q_2} g_{\mid Y\cap Y'}\\
&=\Delta_{X\cap X',X\cap Y'}(f)\star_{q_1,q_2}\Delta_{Y\cap X',Y\cap Y'}(g).
\end{align*}
So $(\bfbool,\star_{q_1,q_2},\Delta)$ is a twisted bialgebra.\\

2. By Theorem \ref{theo2.2}, it remains to show that $\theta_q$ is a twisted coalgebra morphism. 
Let $Y\subseteq X$ be two finite sets and $f\in\rmbool(X)
$. For any $A\subseteq Y$,
\begin{align*}
\theta_q(f)_{\mid Y}(A)&=\theta_q(f)(A)=\sum_{B\subseteq A}q^{|A|-|B|}f(B)=\sum_{B\subseteq A}q^{|A|-|B|}f_{\mid Y}(A)=\theta_q(f_{\mid Y})(A),
\end{align*}
so $\theta_q(f)_{\mid Y}=\theta_q(f_{\mid Y})$. This directly implies that $\theta_q$ is a coalgebra morphism.
\end{proof}

\subsection{Application of the bosonic Fock functor}

Consequently, the application of the bosonic Fock functor $\calF$ of \cite{Aguiar2010} gives a two-parameters family of bialgebras sharing the same coproduct. Let us describe then. The space $\calH_\bfbool$ has for basis the set of isoclasses of boolean functions.
If $f\in\rmbool(X)$, its isoclass is denoted by $\overline{f}$. The products $\starq$, for $q_1,q_2\in\Z$, the coproduct and the counit are given by
\begin{align*}
\forall f_1\in\rmbool(X),\:\forall f_2\in\rmbool(Y),&&\overline{f_1}\starq\overline{f_2}&=\overline{f_1\starq f_2},\\
\forall f\in\rmbool(X),&&\Delta(\overline{f})&=\sum_{X'\sqcup X''=X}\overline{f_{\mid X'}}\otimes\overline{f_{\mid X''}},\\
&&\varepsilon_\Delta(\overline{f})&=\delta_{f,1}.
\end{align*}
The unit is the isoclass of $1$, also denoted by $1$. Applying the functor $\calF$ to $\theta_q$, we obtain bialgebra isomorphisms, also denoted by $\theta_q$:
\begin{align*}
\theta_q&:\left\{\begin{array}{rcl}
(\calH_\bfbool,\starq,\Delta)&\longrightarrow&(\calH_\bfbool,\star_{q_1+q,q_2+q},\Delta)\\
\overline{f}&\longmapsto&\overline{\theta_q(f)}.
\end{array}\right.\end{align*}
The inverse of $\theta_q$ is $\theta_{-q}$.

\section{Coproduct associated to a family of equivalences}

\subsection{Contraction and restriction}

\begin{defi}\label{defi3.1}
Let $X$ be a finite set, $f\in\rmbool(X)$ and $\sim\in\calE(X)$. We put
\begin{align*}
f/{\sim}&:\left\{\begin{array}{rcl}
\calP(X/{\sim})&\longrightarrow&\Z\\
A&\longmapsto&f\left(\varpi_\sim^{-1}(A)\right),
\end{array}\right.&
f\mid\sim&:\left\{\begin{array}{rcl}
\calP(X)&\longrightarrow&\Z\\
A&\longmapsto&\displaystyle\sum_{Y\in X/{\sim}} f(A\cap Y).
\end{array}\right.\end{align*}
Then $f/{\sim}\in\rmbool(X/{\sim})$ and $f\mid\sim\in\rmbool(X)$.
\end{defi}

\begin{remark} 
For any $f\in\rmbool(X)$ and $\sim\in\calE(X)$, $\displaystyle f\mid\sim=\prod^{\star_1}_{Y\in X/{\sim}} f_{\mid Y}$.
\end{remark}

\begin{lemma}\label{lemma3.2}
Let $X$ be a finite set, $\sim\subseteq\sim'\in\calE(X)$ and $f\in\rmbool(X)$. 
\begin{align*}
(f/{\sim})/\overline{\sim'}&=f/{\sim}',&(f\mid\sim')\mid\sim&=f\mid\sim,&(f/{\sim})\mid\overline{\sim'}&=(f\mid\sim')/{\sim}.
\end{align*}\end{lemma}

\begin{proof}
Let $A\subseteq X/{\sim}'=(X/{\sim})/\overline{\sim'}$. 
\begin{align*}
(f/{\sim})/\overline{\sim'}(A)&=f/{\sim}\left(\varpi_{\overline{\sim'}}^{-1}(A)\right)\\
&=f\left(\varpi_\sim^{-1}\left(\varpi_{\overline{\sim'}}^{-1}(A)\right)\right)\\
&=f\left((\varpi_{\overline{\sim'}}\circ\varpi_\sim)^{-1}(A)\right)\\
&=f\left(\varpi_{\sim'}^{-1}(A)\right)=f/{\sim}'(A).
\end{align*}

Let $A\subseteq X$. 
\begin{align*}
(f\mid\sim')\mid\sim(A)&=\sum_{Y\in X/{\sim}} f\mid\sim'(A\cap Y)\\
&=\sum_{Y\in X/{\sim}}\sum_{Y'\in X/{\sim}'}f(A\cap Y\cap Y')\\
&=\sum_{Y'\in X/{\sim}'} f(A\cap Y')=f\mid\sim'(A).
\end{align*}
For the third equality, observe that for any $Y'\in X/{\sim}'$, as $\sim\subseteq\sim'$, there exists a unique $Y\in X/{\sim}$ such that $Y'\subseteq Y$.\\

Let $A\subseteq X/{\sim}$.
\begin{align*}
(f/{\sim})\mid\overline{\sim'}(A)&=\sum_{\overline{Y}\in (X/{\sim})/\overline{\sim'}} f/{\sim}(A\cap\overline{Y})\\
&=\sum_{\overline{Y}\in (X/{\sim})/\overline{\sim'}} f\left(\varpi_\sim^{-1}(A\cap\overline{Y})\right)\\
&=\sum_{\overline{Y}\in (X/{\sim})/\overline{\sim'}} f\left(\varpi_\sim^{-1}(A)\cap\varpi_\sim^{-1}(\overline{Y})\right)\\
&=\sum_{Y\in X/{\sim}'}f\left(\varpi_\sim^{-1}(A)\cap Y\right)\\
&=f\mid\sim'\left(\varpi_\sim^{-1}(A)\right)=(f\mid\sim')/{\sim}(A).\qedhere
\end{align*}
\end{proof}

\begin{lemma}
Let $X$ be a finite set, $f\in\rmbool(X)$, $\sim\in\calE(X)$, and $Y\subseteq X/{\sim}$. Then
\[(f/{\sim})_{\mid Y}=\left(f_{\mid\varpi_\sim^{-1}(Y)}\right)/{\sim_Y},\]
where $\sim_Y=\sim\cap\varpi_\sim^{-1}(Y)^2$.
\end{lemma}

\begin{proof}
Let $Z\subseteq Y$. 
\begin{align*}
(f/{\sim})_{\mid Y}(Z)&=f/{\sim}(Z)=f(\varpi_\sim^{-1}(Z))=f_{\mid\varpi_\sim^{-1}(Y)}(\varpi_{\sim_Y}^{-1}(Z))
=\left(f_{\mid\varpi_\sim^{-1}(Y)}\right)/{\sim_Y}(Z).\qedhere
\end{align*}\end{proof}

\begin{lemma}\label{lemma3.4}
Let $X_1,X_2$ be two disjoint finite sets, and for $i\in [2]$, $\sim_i\in\calE(X_i)$ and $f_i\in\rmbool(X_i)$. Then
\begin{align*}
(f_1\star_1 f_2)/{(\sim_1\sqcup\sim_2)}&=(f_1/{\sim}_1)\star_1 (f_2/{\sim}_2),&(f_1\star_1 f_2)|(\sim_1\sqcup\sim_2)&=(f_1|\sim_1)\star_1 (f_2|\sim_2).
\end{align*}
\end{lemma}

\begin{proof}
We put $X=X_1\sqcup X_2$, $\sim=\sim_1\sqcup\sim_2$ and $f=f_1\star_1 f_2$. Note that 
\[X/{\sim}=(X_1/{\sim}_1)\sqcup (X_2/{\sim}_2).\] 
Let $A\subseteq X/{\sim}$.
\begin{align*}
f/{\sim}(A)&=f\left(\varpi_\sim^{-1}(A)\right)\\
&=f\left(\varpi_{\sim_1}^{-1}(A\cap X_1/{\sim}_1)\sqcup (\varpi_{\sim_2}^{-1}(A\cap X_2/{\sim}_2)\right)\\
&=f_1\left(\varpi_{\sim_1}^{-1}(A\cap X_1/{\sim}_1)\right)+f_2\left(\varpi_{\sim_2}^{-1}(A\cap X_2/{\sim}_2)\right)\\
&=f_1/{\sim}_1(A\cap X_1/{\sim}_1)+f_2/{\sim}_2(A\cap X_2/{\sim}_2)\\
&=(f_1/{\sim}_1)\star_1 (f_2/{\sim}_2)(A).
\end{align*}
As $X/{\sim}=(X_1/{\sim}_1)\sqcup (X_2/{\sim}_2)$,
\begin{align*}
f\mid\sim&=\prod^{\star_1}_{Y_1\in X_1/{\sim}_1} f_{\mid Y_1}\star_1\prod^{\star_1}_{Y_2\in X_2/{\sim}_2}f_{\mid Y_1}
=\prod^{\star_1}_{Y_1\in X_1/{\sim}_1} {f_1}_{\mid Y_1}\star_1\prod^{\star_1}_{Y_2\in X_2/{\sim}_2}{f_2}_{\mid Y_1}
=(f_1|\sim_1)\star_1 (f_2|\sim_2).\qedhere
\end{align*}\end{proof}

\begin{lemma}\label{lemma3.5}
Let $X$ be a finite set, $\sim\in\calE(X)$ and $f\in\rmbool(X)$. 
\begin{enumerate}
\item $\ic(f\mid\sim)\geq\cl(\sim)$. It is an equality if, and only if, for any $Y\in X/{\sim}$, $f_{\mid Y}$ is indecomposable.
\item If $\ic(f\mid\sim)=\cl(\sim)$, then $\sim\subseteq\sim_f^i$ and $\ic(f/{\sim})\geq\ic(f)$.
\end{enumerate}\end{lemma}

\begin{proof}
1. By definition,
\[f\mid\sim=\prod_{Y\in X/{\sim}}^{\star_1} f_{\mid Y},\]
so
\[\ic(f\mid\sim)=\sum_{Y\in X/{\sim}}\ic(f_{\mid Y})\geq |X/{\sim}|=\cl(\sim).\]
Moreover, it is an equality if, and only if, for any $Y\in X/{\sim}$, $\ic(f_{\mid Y})=1$.\\

2. Let us assume that $\ic(f\mid\sim)=\cl(\sim)$. Then, for any $Y\in X/{\sim}$, $f_{\mid Y}$ is indecomposable. 
Let us denote by $X_1,\ldots,X_k$ the indecomposable components of $f$. Then, for any $Y\in X/{\sim}$,
\[f_{\mid Y}=(f_{\mid X_1}\star_1\cdots\star_1 f_{\mid X_k})_{\mid Y}=f_{\mid X_1\cap Y}\star_1\cdots\star_1 f_{\mid X_k\cap Y}.\]
As $f_{\mid Y}$ is indecomposable, one, and only one, of the sets $X_i\cap Y$ is nonempty. This implies that $\sim\subseteq\sim^i_f$. We then put $\sim_i=\sim\cap X_i^2$ for any $i\in [k]$, such that $\sim=\sim_1\sqcup\cdots\sqcup\sim_k$. By Lemma \ref{lemma3.4}, 
\[f/{\sim}=(f_1/{\sim}_1)\star_1\cdots\star_1 (f_k/{\sim}_k),\]
so $\ic(f/{\sim})=\ic(f_1/{\sim}_1)+\cdots+\ic(f_k/{\sim}_k)\geq k=\ic(f)$.
\end{proof}

\begin{lemma}\label{lemma3.6}
Let $X,Y$ be two disjoint finite set, $\sim_X\in\calE(X)$, $\sim_Y\in\calE(Y)$, $f\in\rmbool(X\sqcup Y)$. We put $\sim=\sim_X\sqcup\sim_Y$. Then
\begin{align*}
f_{\mid X}/{\sim}_X&=(f/{\sim})_{\mid X/{\sim}_X},&f_{\mid Y}/{\sim}_Y&=(f/{\sim})_{\mid Y/{\sim}_Y},&
f\mid\sim&=f_{\mid X}\mid\sim_X\star_1 f_{\mid Y}\mid\sim_Y.
\end{align*}\end{lemma}

\begin{proof}
Let $A\subseteq X/{\sim}_X$.
\begin{align*}
f_{\mid X}/{\sim}_X(A)&=f_{\mid X}(\varpi_{\sim_X}^{-1}(A))=f(\varpi_\sim^{-1}(A))=f/{\sim}(A)=(f/{\sim})_{\mid X/{\sim}_X}(A).
\end{align*}
The proof is similar for $Y$. As $(X\sqcup Y)/{\sim}=(X/{\sim}_X)\sqcup (Y/{\sim}_Y)$, we obtain the result for $f\mid\sim$. 
\end{proof}

\subsection{The frame}

\label{section3.2}

\begin{notation}
We now fix a set subspecies $\rmboolt$ of $\rmbool$ such that:
\begin{itemize}
\item $1\in\rmboolt(\emptyset)$.
\item If $X,Y$ are disjoint finite sets, $f\in\rmboolt(X)$, and $g\in\rmboolt(Y)$, then $f\star_1 g\in\rmboolt(X\sqcup Y)$.
\item If $Y\subseteq X$ are two finite sets and $f\in\rmboolt(X)$, then $f_{\mid Y}\in\rmboolt(Y)$.
\end{itemize}
We denote by $\bfboolt$ the linearization of $\rmboolt$, and by $\calH_{\bfboolt}$ the subspace of $\calH_{\bfbool}$ generated by isoclasses of boolean functions in $\rmboolt$. 
\end{notation}

The hypotheses on $\rmboolt$ insure (in fact, are equivalent to) that $\bfboolt$ is a twisted subbialgebra of $(\bfbool,\star_1,\Delta)$. Consequently, $\calH_{\bfboolt}$ is a subbialgebra of $(\calH_{\bfbool},\star_1,\Delta)$.\\
 
\begin{notation}
For any $f\in\rmboolt(X)$, we fix a set of equivalences $\calE(f)\subseteq\calE(X)$ which is compatible with the species structure. 
More precisely, if $\sigma:X\longrightarrow Y$ is a bijection between two finite sets and $f\in\rmboolt(X)$, then $f\circ\sigma^{-1}\in\rmboolt(Y)$ and $\calE(f\circ\sigma^{-1})=\{\sim_\sigma\mid\sim\in\calE(f)\}$.
We assume that:
\begin{itemize}
\item For any finite set $X$, for any $f\in\rmboolt(X)$, for any $\sim\in\calE(f)$, $f/{\sim}\in\rmboolt(X/{\sim})$. 
\end{itemize}
\end{notation}

The hypotheses on $\rmboolt$ insure that for any finite set $X$, for any $\sim\in\calE(X)$, for any $f\in\rmboolt(X)$, $f\mid\sim\in\rmboolt(X)$ and $f/{\sim}\in\rmbool(X/{\sim})$.
We then define a contraction-restriction coproduct $\delta^\calE$ on $\bfboolt$ as follows: for any finite set $X$, for any $f\in\rmboolt(X)$, for any $\sim\in\calE(X)$,
\[\delta^\calE_\sim(f)=\begin{cases}
f/{\sim}\otimes f\mid\sim\mbox{ if }\sim\in\calE(f),\\
0\mbox{ otherwise}.
\end{cases}\]

Of course, in general, this coproduct has no convenient property. Let us give conditions for the coassociativity, multiplicativity, existence of a counit and compatibility with $\Delta$ for $\delta^\calE$.

\subsection{General results}

\begin{prop}\label{prop3.7}
We shall say that $\calE$ satisfies the $\star_1$ condition if:
\begin{itemize}
\item $\calE(1)$ contains the unique equivalence $\sim_\emptyset\in\calE(\emptyset)$.
\item For any couple of disjoint finite sets $(X,Y)$, for any $f\in\rmboolt(X)$, $g\in\rmboolt(Y)$,
\[\calE(f\star_1 g)=\{\sim_X\sqcup\sim_Y\mid\sim_X\in\calE(f),\:\sim_Y\in\calE(g)\}.\]
\end{itemize}
The coproduct $\delta^\calE$ is compatible with the product $\star_1$ if, and only if, $\calE$ satisfies the $\star_1$ condition.
\end{prop}

\begin{proof}
By definition of $\delta$, 
\[\delta_{\sim_\emptyset}(1)=\begin{cases}
1\otimes 1\mbox{ if }\sim_\emptyset\in\calE(\emptyset),\\
0\mbox{ otherwise}.
\end{cases}\]
So $\delta_{\sim_\emptyset}(1)=1\otimes 1$ if, and only if, $\sim_\emptyset\in\calE(\emptyset)$.
This gives the first item of the $\star_1$ condition.\\

$\Longleftarrow$. Let $X$, $Y$ be disjoint finite sets and $f\in\rmboolt(X)$, $g\in\rmboolt(Y)$. Let $\sim\in\calE(X\sqcup Y)$. We put $\sim_X=\sim\cap X^2$ and $\sim_Y=\sim\cap Y^2$.
If $\sim\neq\sim_X\sqcup\sim_Y$, by the $\star_1$ condition, $\sim\notin\calE(f\star_1 g)$, so $\delta_\sim(f\star_1 g)=0$. If $\sim=\sim_X\sqcup\sim_Y$, by the $\star_1$ condition and Lemma \ref{lemma3.4},
\begin{align*}
\delta^\calE_\sim(f\star_1 g)&=\begin{cases}
(f\star_1 g)/{\sim}\otimes (f\star_1 g)\mid\sim\mbox{ if $\sim_X\in\calE(f)$ and $\sim_Y\in\calE(g)$},\\
0\mbox{ otherwise}
\end{cases}\\
&=\begin{cases}
(f/{\sim}_X)\star_1 (g/{\sim}_Y)\otimes (f\mid\sim_X)\star_1 (g\mid\sim_Y)\mbox{ if $\sim_X\in\calE(f)$ and $\sim_Y\in\calE(g)$},\\
0\mbox{ otherwise}
\end{cases}\\
&=\delta_\sim^\calE(f)\star_1\delta_\sim^\calE(g).
\end{align*}
So $\delta^\calE$ is compatible with the product.\\

$\Longrightarrow$. Let us assume that $\delta^\calE$ is compatible with the product.
Let $X,Y$ be two disjoint finite sets, $f\in\rmboolt(X)$, $g\in\rmboolt(Y)$, $\sim\in\calE(X\sqcup Y)$.
We put $\sim_X=\sim\cap X^2$ and $\sim_Y\in\cap Y^2$. If $\sim\neq\sim_X\sqcup\sim_Y$, then
$\delta_\sim^\calE(f\star_1 g)=0$, so $\sim\notin\calE(f\star_1 g)$. Let us assume that $\sim=\sim_X\sqcup\sim_Y$. Then 
\begin{align*}
\delta^\calE_\sim(f\star_1 g)&=\begin{cases}
(f\star_1 g)/{\sim}\otimes (f\star_1 g)\mid\sim\mbox{ if $\sim\in\calE(f\star_1 g)$},\\
0\mbox{ otherwise},
\end{cases}\\
\delta^\calE_{\sim_X}(f)\star_1\delta^\calE_{\sim_Y}(g)&=\begin{cases}
(f/{\sim}_X)\star_1 (g/{\sim}_Y)\otimes (f\mid\sim_X)\star_1 (g\mid\sim_Y)\mbox{ if $\sim_X\in\calE(f)$ and $g\in\sim_Y\in\calE(g)$},\\
0\mbox{ otherwise}.
\end{cases}
\end{align*}
So $\sim\in\calE(f\star_1 g)$ if, and only if, $\sim_X\in\calE(f)$ and $\sim_Y\in\calE(g)$,
which gives the second item of the $\star_1$ condition.
\end{proof}

\begin{lemma}\label{lemma3.8}
Let us assume that $\calE$ satisfies the $\star_1$ condition. Let $X$ be a finite set and $f\in\rmboolt(X)$. For any $\sim\in\calE(f)$, $\sim\subseteq\sim_f^i$.
\end{lemma}

\begin{proof}
We denote by $X_1,\ldots,X_k$ the indecomposable components of $f$. Then $f=f_{\mid X_1}\star_1\cdots\star_1 f_{\mid X_k}$. Note that $X/{\sim}_f^i=\{X_1,\ldots,X_k\}$. By the $\star_1$ condition,
\[\calE(f)=\{\sim_1\sqcup\cdots\sqcup\sim_k\mid\forall i\in [k],\:\sim_i\in\calE(f_{\mid X_i})\}.\]
 Consequently, if $\sim\in\calE(f)$, then
\[\sim=(\sim\cap X_1^2)\sqcup\cdots\sqcup (\sim\cap X_k^2)\subseteq X_1^2\sqcup\cdots\sqcup X_k^2=\sim_f^i.\qedhere\]
\end{proof}

\begin{prop}\label{prop3.9}
Let us assume that $\calE$ satisfies the $\star_1$ condition. We shall say that $\calE$ satisfies the $\delta$ condition 
if for any finite set $X$, for any $\sim\subseteq\sim'\in\calE(X)$, for any $f\in\rmboolt(X)$, the following assertions are equivalent:
\begin{enumerate}
\item $\sim\in\calE(f)$ and $\overline{\sim'}\in\calE(f/{\sim})$.
\item $\sim'\in\calE(f)$ and $\sim\in\calE(f\mid\sim')$. 
\end{enumerate}
The coproduct $\delta^\calE$ is coassociative if, and only if, $\calE$ satisfies the $\delta$ condition. 
\end{prop}

\begin{proof}
Let $X$ be a finite set, $f\in\rmboolt(X)$, and $\sim,\sim'\in\calE(X)$. Let us firstly assume that $\sim$ is not included in $\sim'$. As $\sim^i_{f\mid\sim'}\subseteq\sim'$, $\sim$ is not included in $\sim^i_{f\mid\sim'}$. 
By Lemma \ref{lemma3.8}, $\sim\notin\calE(f\mid\sim')$, so $\left(\id_{\bfboolt(X/{\sim}')}\otimes\delta_\sim^\calE\right)\circ\delta_{\sim'}^\calE(f)=0$. Let us now consider the case where $\sim\subseteq\sim'$. 
\begin{align*}
\left(\delta^\calE_{\overline{\sim'}}\otimes\id_{\bfboolt(X)}\right)\circ\delta^\calE_\sim(f)&=\begin{cases}
(f/{\sim})/\overline{\sim'}\otimes (f/{\sim})\mid\overline{\sim'}\otimes f\mid\sim\mbox{ if }\sim\in\calE(f)\mbox{ and }\overline{\sim'}\in\calE(f/{\sim}),\\
0\mbox{ otherwise}.
\end{cases}\\
\left(\id_{\bfboolt(X/{\sim'})}\otimes\delta^\calE_\sim\right)\circ\delta^\calE_{\sim'}(f)&=\begin{cases}
f/{\sim}'\otimes (f\mid\sim')/{\sim}\otimes (f\mid\sim'\mid\sim)\mbox{ if }\sim'\in\calE(f)\mbox{ and }\sim\in\calE(f\mid\sim'),\\
0\mbox{ otherwise}.
\end{cases}
\end{align*}
By Lemma \ref{lemma3.2},
\begin{align*}
(f/{\sim})/\overline{\sim'}&=f/{\sim}',&
(f/{\sim})\mid\overline{\sim'}&=\otimes (f\mid\sim')/{\sim},&
f\mid\sim&=(f\mid\sim'\mid\sim).
\end{align*}
So $\delta^\calE$ is coassociative if, and only if, $\calE$ satisfies the $\delta$ condition. 
\end{proof}

\begin{prop}\label{prop3.10}
We shall say that $\calE$ satisfies the $\Delta$ condition if for any couple $(X,Y)$ of disjoint finite sets, for any $f\in\rmboolt(X\sqcup Y)$, for any $\sim_X\in\calE(X)$ and $\sim_Y\in\calE(Y)$, the following assertions are equivalent:
\begin{enumerate}
\item $\sim_X\sqcup\sim_Y\in\calE(f)$.
\item $\sim_X\in\calE(f_{\mid X})$ and $\sim_Y\in\calE(f_{\mid Y})$.
\end{enumerate}
The coproduct $\delta$ is compatible with $\Delta$ if, and only if, $\calE$ satisfies the $\Delta$ condition.
\end{prop}

\begin{proof}
Let $X,Y$ be disjoint finite sets, $f\in\rmboolt(X\sqcup Y)$, $\sim_X\in\calE(X)$ and $\sim_Y\in\calE(Y)$. We put $\sim=\sim_X\sqcup\sim_Y$. 
\begin{align*}
\left(\Delta_{X/{\sim}_X,Y/{\sim}_Y}\otimes\id_{\bfboolt(X\sqcup Y)}\right)\circ\delta_\sim(f)&=\begin{cases}
(f/{\sim})_{\mid X/{\sim}_X}\otimes (f/{\sim})_{\mid Y/{\sim}_Y}\otimes f\mid\sim\mbox{ if }\sim\in\calE(f),\\
0\mbox{ otherwise},
\end{cases}\\
(\star_1)_{1,3,24}\circ (\delta_{\sim_X}\otimes\delta_{\sim_Y})\circ\Delta_{X,Y}(f)&=\begin{cases}
f_{\mid X}/{\sim}_X\otimes f_{\mid Y}/{\sim}_Y\otimes f_{\mid X}\mid\sim_X\star_1 f_{\mid Y}\mid\sim_Y\\
\hspace{2cm}\mbox{ if }\sim_X\in\calE(f_{\mid X})\mbox{ and }\sim_Y\in\calE(f_{\mid Y}),\\
0\mbox{ otherwise}. 
\end{cases}
\end{align*}
By Lemma \ref{lemma3.6},
\begin{align*}
f_{\mid X}/{\sim}_X&=(f/{\sim})_{\mid X/{\sim}_X},&f_{\mid Y}/{\sim}_Y&=(f/{\sim})_{\mid Y/{\sim}_Y},&
f\mid\sim&=f_{\mid X}\mid\sim_X\star_1 f_{\mid Y}\mid\sim_Y.
\end{align*}
So $\delta^\calE$ is compatible with $\Delta$ if, and only if, the $\Delta$ condition holds for $\calE$. 
\end{proof}

\begin{prop}\label{prop3.11}
Let us assume that $\calE$ satisfies the $\star_1$ condition. We shall say that $\calE$ satisfies the $\epsilon$ condition if:
\begin{enumerate}
\item For any $f\in\rmboolt(X)$, $\sim_f^i$ and $=_X$ belong to $\calE(f)$.
\item For any $f\in\rmboolt(X)$ and for any $\sim\in\calE(f)$, $f\mid\sim$ is modular if, and only if, $\sim$ is equal to $=_X$.
\item For any $f\in\rmboolt(X)$ and for any $\sim\in\calE(f)$, $f/{\sim}$ is modular if, and only if, $\sim=\sim_f^i$.
\end{enumerate}
The coproduct $\delta^\calE$ has a counit if, and only if, $\calE$ satisfies the $\epsilon_\delta$ condition.
If so, the counit is given by
\begin{align}
\label{EQ2}
&\forall f\in\rmboolt(X),&\epsilon_\delta(f)&=\begin{cases}
1\mbox{ if $f$ is modular},\\
0\mbox{ otherwise}.
\end{cases}\end{align}\end{prop}

\begin{proof}
$\Longrightarrow$. Let us denote by $\epsilon_{\delta^\calE}$ the counit of $\delta^\calE$. Let $f\in\rmboolt(X)$. 
If $=_X\notin\calE(f)$, then $\delta^\calE_{=_X}(f)=0$, so $\left(\id_{\bfboolt(X)}\otimes\epsilon_{\delta^\calE}\right)\circ\delta^\calE(f)=0$: this is a contradiction. So $=_X\in\calE(f)$. As $\calE$ satisfies the $\star_1$ condition, for any $\sim\in\calE(f)$, $\sim\subseteq\sim_f^i$. If $\sim\subsetneq\sim_f^i$, by Lemma \ref{lemma3.5},
\[\ic(f\mid\sim)\geq\cl(\sim)>\cl(\sim_f^i)=\ic(f),\]
so $f\mid\sim\neq f$. As
\[\sum_{\sim\in\calE(X)}\left(\epsilon_{\delta^\calE}\otimes\id_{\bfboolt(X)}\right)\circ\delta^\calE_\sim(f)
=\sum_{\sim\in\calE(f)}\epsilon_{\delta^\calE}(f/{\sim}) f\sim=f,\]
necessarily $\sim_f^i\in\calE(f)$.\\

If $f$ is modular and $\sim$ is $=_X$, then $f\mid\sim=f/{\sim}=f$, so
\[\left(\id_{\bfboolt(X)}\otimes\epsilon_{\delta^\calE}\right)\circ\delta_{=_X}^\calE(f)=\epsilon_{\delta^\calE}(f)f=f,\]
so $\epsilon_{\delta^\calE}(f)=1$. If $f$ is not modular, then $\sim_f^i$ is not the equality of $X$, so, as $f\mid\sim_f^i=f$,
\[\left(\id_{\bfboolt(X)}\otimes\epsilon_{\delta^\calE}\right)\circ\delta_{\sim_f^i}^\calE(f)=\epsilon_{\delta^\calE}(f)f/{\sim}_f^i=0,\]
and $\epsilon_{\delta^\calE}(f)=0$. We have just proved that the counit does not depend on $\calE$, and we denote now $\epsilon_\delta$
instead of $\epsilon_{\delta^\calE}$.\\

With all these partial results, we obtain that for any $\sim\in\calE(f)$,
\begin{align*}
\left(\id_{\bfboolt(X)}\otimes\epsilon_\delta\right)\circ\delta_\sim^\calE(f)&=\begin{cases}
f\mid\sim\mbox{ if $f\mid\sim$ is modular},\\
0\mbox{ otherwise},
\end{cases}\\
&=\begin{cases}
f\mbox{ if $\sim$ is the equality $=_X$ of $X$},\\
0\mbox{ otherwise}.
\end{cases}
\end{align*}
So the unique $\sim\in\calE(f)$ such that $f\mid\sim$ is modular is $=_X$. On the other side,
\begin{align*}
f&=\sum_{\sim\in\calE(X)}\left(\epsilon_\delta\otimes\id_{\bfboolt(X)}\right)\circ\delta^\calE_\sim(f)
=\sum_{\substack{\sim\in\calE(f),\\\mbox{\scriptsize $f/{\sim}$ modular}}} f\mid\sim.
\end{align*}
Hence, there exists a unique $\sim\in\calE(f)$ such that $f/{\sim}$ is modular, and it can only be $\sim_f^i$.\\

$\Longleftarrow$. With $\epsilon_\delta$ defined by (\ref{EQ2}), for any $f\in\rmboolt(X)$ and any $\sim\in\calE(X)$,
\begin{align*}
\left(\id_{\bfboolt(X)}\otimes\epsilon_\delta\right)\circ\delta_\sim^\calE(f)&=\begin{cases}
f\mid\sim\mbox{ if $f/{\sim}$ is the equality of $X$},\\
0\mbox{ otherwise},
\end{cases}\\
&=\begin{cases}
f\mbox{ if $f/{\sim}$ is the equality of $X$},\\
0\mbox{ otherwise},
\end{cases}\end{align*}
and
\begin{align*}
\sum_{\sim\in\calE(X)}\left(\epsilon_\delta\otimes\id_{\bfboolt(X)}\right)\circ\delta^\calE_\sim(f)=f\mid\sim_f^i=f,
\end{align*}
so $\epsilon_\delta$ is a counit of $\delta^\calE$.
\end{proof}

\begin{remark}
We automatically obtain that $\epsilon_\delta$ is an algebra morphism, as $f\star_1 g$ is modular if, and only if, $f$ and $g$ are modular. 
\end{remark}

Applying the functor $\calF$:

\begin{theo}
We define a second coproduct on $\calH_{\bfboolt}$ as follows: 
\begin{align*}
&\forall f\in\rmboolt([n]),&\delta^\calE(\overline{f})&=\sum_{\sim\in\calE(f)}\overline{f/{\sim}}\otimes\overline{f\mid\sim}.
\end{align*}
Then:
\begin{enumerate}
\item If $\calE$ satisfies the $\star_1$ condition, then $\delta^\calE$ is an algebra morphism. 
\item If $\calE$ satisfies the $\delta$ condition, then $\delta^\calE$ is coassociative.
\item If $\calE$ satisfies the $\epsilon$ condition, then $\delta^\calE$ has a counit, defined by
\begin{align*}
&\forall f\in\rmboolt([n]),&\epsilon_\delta(\overline{f})&=\begin{cases}
1\mbox{ if $f$ is modular},\\
0\mbox{ otherwise}.
\end{cases}\end{align*}
\item If $\calE$ satisfies the $\Delta$ condition, then
\[(\star_1)_{1,3,24}\circ\left(\delta^\calE\otimes\delta^\calE\right)\circ\Delta=\left(\Delta\otimes\id_{\calH_{\bfboolt}}\right)\circ\delta^\calE.\]
\item For any $x\in\calH_\bfbool$, 
\[\left(\varepsilon_\Delta\otimes\id_{\calH_{\bfboolt}}\right)\circ\delta^\calE(x)=\varepsilon_\Delta(x)1.\]
\end{enumerate}\end{theo}

\subsection{A first example: weak equivalences}

\begin{defi}
Let $f\in\rmbool(X)$. We put
\[\calE^W(f)=\{\sim\in\calE(X)\mid\ic(f\mid\sim)=\cl(\sim)\}.\]
The associated coproduct is denoted by $\delta^W$, instead of $\delta^{\calE^W}$.
\end{defi}

\begin{theo}\label{theo3.14}\begin{enumerate}
\item $\calE^W$ satisfies the $\star_1$ condition.
\item $\calE^W$ satisfies the $\Delta$ condition.
\item $\calE^W$ does not satisfy the $\delta$ condition.
\item $\calE^W$ does not satisfy the $\epsilon$ condition, but, however, $\epsilon_\delta$ defined by (\ref{EQ2}) is a right counit for $\delta^W$.
\end{enumerate}\end{theo}

\begin{proof}
1. Obviously, $\sim_\emptyset\in\calE^W(1)$. Let $X,Y$ be disjoint finite sets, $f\in\rmbool(X)$, $g\in\rmbool(Y)$, and $\sim\in\calE(X\sqcup Y)$. We put $\sim_X=\sim\cap X^2$ and $\sim_Y=\sim\cap Y^2$. 
If $\sim\neq\sim_X\sqcup\sim_Y$, let $Z\in (X\sqcup Y)/{\sim}$ such that both $Z\cap X$ and $Z\cap Y$ are nonempty. Therefore, 
\[(f\star_1 g)_{\mid Z}=f_{\mid X\cap Z}\star_1 g_{\mid Y\cap Z},\]
so $\ic\left((f\star_1 g)_{\mid Z}\right)\geq 2$. Consequently, $\ic(f\mid\sim)>\cl(\sim)$, and $\sim\notin\calE^W(f\star_1 g)$. If $\sim=\sim_X\star_1\sim_Y$, then, by Lemma \ref{lemma3.4},
\[(f\star_1 g)\mid\sim=(f\mid\sim_X)\star_1 (g\mid\sim_Y).\]
Moreover, $\ic(f\mid\sim_X)\geq\cl(\sim_X)$ and $\ic(f\mid\sim_Y)\geq\cl(\sim_Y)$, by Lemma \ref{lemma3.5}. We obtain that
\begin{align*}
\sim\in\calE^W(f\star_1 g)&\Longleftrightarrow\ic(f\mid\sim_X)+\ic(g\mid\sim_Y)=\cl(\sim_X)+\cl(\sim_Y)\\
&\Longleftrightarrow\ic(f\mid\sim_X)=\cl(\sim_X)\mbox{ and }\ic(g\mid\sim_Y)=\cl(\sim_Y)\\
&\Longleftrightarrow\sim_X\in\calE^W(f)\mbox{ and }\sim_Y\in\calE^W(g).
\end{align*}
So $\calE^W$ satisfies the $\star_1$ condition.\\

2. Let $f\in\rmbool(X\sqcup Y)$, $\sim_X\in\calE(X)$, and $\sim_Y\in\calE(Y)$. We put $\sim=\sim_X\sqcup\sim_Y$. Then
\[f\mid\sim=\prod^{\star_1}_{X'\in X/{\sim}_X} f_{\mid X'}\star_1\prod^{\star_1}_{Y'\in Y/{\sim}_Y} f_{\mid Y'}=\prod^{\star_1}_{X'\in X/{\sim}_X} (f_{\mid X})_{\mid X'}\star_1\prod^{\star_1}_{Y'\in Y/{\sim}_Y} (f_{\mid Y})_{\mid Y'}=(f_{\mid X}\mid\sim_X)\star_1 (f_{\mid Y}\mid\sim_Y).\]
By Lemma \ref{lemma3.5}, $\ic(f_{\mid X}\mid\sim_X)\geq\cl(\sim_X)$ and $\ic(f_{\mid Y}\mid\sim_Y)\geq\cl(\sim_Y)$. Therefore,
\begin{align*}
\sim\in\calE^W(f)&\Longleftrightarrow\ic(f_{\mid X}\mid\sim_X)+\ic(f_{\mid Y}\mid\sim_Y)=\cl(\sim_X)+\cl(\sim_Y)\\
&\Longleftrightarrow\ic(f_{\mid X}\mid\sim_X)=\cl(\sim_X)\mbox{ and }\ic(f_{\mid Y}\mid\sim_Y)=cl(\sim_Y)\\
&\Longleftrightarrow\sim_X\in\calE^W(f_{\mid X})\mbox{ and }\sim_Y\in\calE^W(f_{\mid Y}).
\end{align*} 
So $\calE^W$ satisfies the $\Delta$ condition.\\

3. Let $f\in\rmbool(X)$, and $\sim\subseteq\sim'\in\calE(X)$. In the $\delta$ condition for these elements, let us show that $1.\Longrightarrow 2$. 
Firstly, by Lemma \ref{lemma3.2}, $(f\mid\sim')\mid\sim=f\mid\sim$, so, as $\sim\in\calE^W(f)$, 
\[\ic((f\mid\sim')\mid\sim)=\ic(f\mid\sim)=\cl(\sim),\]
so $\sim\in\calE^W(f\mid\sim')$. Moreover, $\ic(f\mid\sim')\geq\cl(\sim')$, by lemma \ref{lemma3.5}. Let us assume that $\ic(f\mid\sim')>\cl(\sim')$. 
As $\ic((f\mid\sim')\mid\sim)=\cl(\sim)$, by Lemmas \ref{lemma3.5} and \ref{lemma3.2}, 
\[\ic((f/{\sim})\mid\overline{\sim'})=\ic((f\mid\sim')/{\sim})\geq\ic(f\mid\sim')>\cl(\sim'),\]
so $\overline{\sim'}\notin\calE^W(f/{\sim})$. We proved that $1.\Longrightarrow 2$. 
However, $2.$ does not imply $1.$, as can be seen in Example \ref{ex3.1} below.\\

4. Firstly, $=_X$ and $\sim_f^i$ belong to $\calE^W(f)$ for any $f\in\rmbool(X)$. If $\sim\in\calE^W(f)$, such that $f\mid\sim$ is modular, then
\[\ic(f\mid\sim)=\cl(\sim)=|X|,\]
as, firstly, $\sim\in\calE^W(f)$ and, secondly, $f\mid\sim$ is modular. So $\sim$ is $=_X$, which implies that $\epsilon_\delta$ is a right counit. See Example \ref{ex3.2} for the problem on the left.
\end{proof}

\begin{example}\label{ex3.1}
Let $\sim\in\rmbool([4])$, such that
\begin{align*}
&\forall i,j\in [4],\mbox{ with }i\neq j,&f(\{i,j\})&\neq f(\{i\})+f(\{j\}),\\
&&f(\{1,2,3\})&=f(\{1,2\})+f(\{3\}).
\end{align*}
By Proposition \ref{prop2.17} with any $x$, if $X\subseteq [4]$, nonempty, then $f_{\mid X}$ is indecomposable, so $\calE^W(f)=\calE([4])$. Let $\sim\in\calE([4])$ whose classes are $\overline{1}=\{1,2\}$, $\overline{3}=\{3\}$, and $\overline{4}=\{4\}$,
and $\sim'\in\calE([4])$, whose classes are $\{1,2,3\}$ and $\{4\}$. Then $\sim\subseteq\sim'$ and the previous observation shows that $\sim'\in\calE^W(f)$ and $\sim\in\calE^W(f\mid\sim')$. However,
$\{\overline{1},\overline{3}\}$ is a class of $\overline{\sim'}$ and
\[f/{\sim}(\{\overline{1},\overline{3}\})=f(\{1,2,3\})=f(\{1,2\})+f(\{3\})=f/{\sim}(\{\overline{1}\})+f/{\sim}(\{\overline{3}\}),\]
so $(f/{\sim})_{\mid\{\overline{1},\overline{3}\}}$ is decomposable by Proposition \ref{prop2.18}, and $\overline{\sim'}\notin\calE^W(f/{\sim})$.
\end{example}

\begin{example}\label{ex3.2}
Let $f\in\rmbool([3])$ such that
\begin{align*}
&\forall i,j\in [3],\mbox{ with }i\neq j,&f(\{i,j\})&\neq f(\{i\})+f(\{j\}),\\
&&f(\{1,2,3\})&=f(\{1,2\})+f(\{3\}).
\end{align*}
By Lemma \ref{prop2.17} with any $i\in [3]$, for any nonempty $Y\subseteq [3]$, $f_{\mid Y}$ is indecomposable. So $\calE^W(f)=\calE(X)$. Let $\sim\in\calE(X)$ whose classes are $\overline{1}=\{1,2\}$ and $\overline{3}=\{3\}$. Then $f\in\calE^W(f)$. Moreover,
\begin{align*}
f/{\sim}(\{\overline{1},\overline{3}\})&=f(\{1,2,3\})=f(\{1,2\}+f(\{3\})=f/{\sim}(\{\overline{1}\})+f/{\sim}(\{\overline{3}\}),
\end{align*} 
so $f/{\sim}$ is modular, whereas $\sim\neq\sim_f^i$.
\end{example}

Applying the functor $\calF$:

\begin{cor}
We define a coproduct on $\calH_\bfbool$ as follows: 
\begin{align*}
&\forall f\in\rmbool([n]),&\delta^W(\overline{f})&=\sum_{\sim\in\calE^W(f)}\overline{f/{\sim}}\otimes\overline{f\mid\sim}.
\end{align*}
Then:
\begin{enumerate}
\item $\delta^W$ is an algebra morphism. 
\item$\delta^W$ is not coassociative. 
\item $\epsilon_\delta$ is a right counit for $\delta^W$, but $\delta^W$ has no left counit.
\item $\delta^W$ and $\Delta$ are compatible.
\item$\delta^W$ and $\varepsilon_\Delta$ are compatible. 
\end{enumerate}\end{cor}

\begin{example}
Let us illustrate the lack of left counit for $\delta^W$. Let us assume that $\mu$ is a left counit of $\delta^W$. 
For any modular boolean function $f$, 
\[\left(\mu\otimes\id_{\calH_\bfbool}\right)\circ\delta^W(\overline{f})=\mu(\overline{f})\overline{f}=\overline{f},\]
so $\mu(\overline{f})=1$. Let $f\in\rmbool([3])$ of Example \ref{ex3.2}. 
By Proposition \ref{prop2.17}, for any $X\subseteq [3]$, nonempty, $f_{\mid X}$ is indecomposable, so $\calE^W(f)=\calE([3])$. Let us consider $\sim\in\calE([3])$ whose classes are $\overline{1}=\{1,2\}$ and $\overline{3}=\{3\}$. Then 
\[f/{\sim}(\{\overline{1},\overline{3}\})=f(\{1,2,3\})=f(\{1,2\})+f(\{3\})=f/{\sim}(\{\overline{1}\})+f/{\sim}(\{\overline{3}\}),\]
so $f/{\sim}$ is modular, and therefore $\mu(f/{\sim})=1$: $\overline{f\mid\sim}$ appears in $\left(\mu\otimes\id_{\calH_{\bfbool}}\right)\circ\delta^W(\overline{f})=\overline{f}$. But $f\mid\sim$ is not indecomposable whereas $f$ is, so $\overline{f\mid\sim}\neq\overline{f}$: this is a contradiction. So $\delta^W$ has no left counit. 
\end{example}

\subsection{A second example: strong equivalences}

\begin{defi}
Let $f\in\rmbool(X)$. We put
\[\calE^S(f)=\{\sim\in\calE(X)\mid\ic(f\mid\sim)=\cl(\sim),\:\ic(f/{\sim})=\ic(f)\}.\]
The associated coproduct is denoted by $\delta^S$, instead of $\delta^{\calE^S}$.
\end{defi}

\begin{remark}
For any $f\in\rmbool(X)$, $\calE^S(f)\subseteq\calE^W(f)$. This inclusion can be strict, see Example \ref{ex3.4} below.
\end{remark}

\begin{example}\label{ex3.4}
Let us consider the boolean function $f\in\rmbool([3])$ of Example \ref{ex3.1}:
\begin{align*}
&\forall i,j\in [3],\mbox{ with }i\neq j,&f(\{i,j\})&\neq f(\{i\})+f(\{j\}),\\
&&f(\{1,2,3\})&=f(\{1,2\})+f(\{3\}).
\end{align*}
We already observed that $\calE^W(f)=\calE([3])$. Let $\sim\in\calE(X)$ whose classes are $\overline{1}=\{1,2\}$ and $\overline{3}=\{3\}$. Then $f\in\calE^W(f)$. Moreover,
\begin{align*}
f/{\sim}(\{\overline{1},\overline{3}\})&=f(\{1,2,3\})=f(\{1,2\}+f(\{3\})=f/{\sim}(\{\overline{1}\})+f/{\sim}(\{\overline{3}\}),
\end{align*} 
so $f/{\sim}$ is modular, and $\ic(f/{\sim})=2>1=\ic(f)$. So $\sim\notin\calE^S(f)$.
\end{example}

\begin{theo}\label{theo3.17}\begin{enumerate}
\item $\calE^S$ satisfies the $\star_1$ condition.
\item $\calE^S$ does not satisfy the $\Delta$ condition.
\item $\calE^S$ satisfies the $\delta$ condition.
\item $\calE^S$ satisfies the $\epsilon$ condition.
\end{enumerate}\end{theo}

\begin{proof}
1. Obviously, $\sim_\emptyset\in\calE^S(1)$. Let $X,Y$ be disjoint finite sets, $f\in\rmbool(X)$, $g\in\rmbool(Y)$, $\sim\in\calE(X\sqcup Y)$. We put $\sim_X=\sim\cap X^2$ and $\sim_Y=\sim\cap Y^2$. 
If $\sim\neq\sim_X\sqcup\sim Y$, then, as $\calE^W$ satisfy the $\star_1$ condition, $\sim\notin\calE^W(f)$. As $\calE^S(f)\subseteq\calE^W(f)$, $\sim\notin\calE^S(f)$. Let us assume that $\sim=\sim_X\sqcup\sim_Y$.
By lemma \ref{lemma3.4},
\begin{align*}
(f\star_1 g)/{\sim}&=(f/{\sim}_X)\star_1 (g/{\sim}_Y),&(f\star_1 g)\mid\sim&=(f\mid\sim_X)\star_1 (g\mid\sim_Y).
\end{align*} 
Moreover, by Lemma \ref{lemma3.5}, $\ic(f\mid\sim_X)\geq\cl(\sim_X)$ and, if the equality is satisfied, then $\ic(f/{\sim}_X)\geq\ic(f)$. The same for $g$ and $Y$. Therefore,
\begin{align*}
\sim\in\calE^S(f\star_1 g)&\Longleftrightarrow\begin{cases}
\ic(f\mid\sim_X)+\ic(g\mid\sim_Y)=\cl(\sim_X)+\cl(\sim_Y)\\
\ic(f/{\sim}_X)+\ic(g/{\sim}_Y)=\ic(f)+\ic(g)\\
\end{cases}\\
&\Longleftrightarrow\begin{cases}
\ic(f\mid\sim_X)=\cl(\sim_X)\\
\ic(g\mid\sim_Y)=\cl(\sim_Y)\\
\ic(f/{\sim}_X)=\ic(f)\\
\ic(g/{\sim}_Y)=\ic(g)\\
\end{cases}\\
&\Longleftrightarrow\sim_X\in\calE^S(f)\mbox{ and }\sim_Y\in\calE^S(g).
\end{align*}
So $\calE^S$ satisfies the $\star_1$ condition.\\

2. This is illustrated by Examples \ref{ex3.5} and \ref{ex3.6} below.\\

3. Let $X$ be a finite set, $\sim\subseteq\sim'\in\calE(X)$ and $f\in\rmbool(X)$. 
Firstly, using Lemma \ref{lemma3.2}, observe that
\begin{align*}
1.&\Longleftrightarrow\begin{cases}
\ic(f\mid\sim)=\cl(\sim)\\
\ic(f/{\sim})=\ic(f)\\
\ic((f/{\sim})\mid\overline{\sim'})=\cl(\overline{\sim'})\\
\ic((f/{\sim})/\overline{\sim'})=\ic(f/{\sim})
\end{cases}
&\hspace{-10mm}&\Longleftrightarrow\begin{cases}
\ic(f/{\sim})=\ic(f/{\sim}')=\ic(f)\\
\ic(f\mid\sim)=\cl(\sim)\\ 
\ic((f\mid\sim')/{\sim})=\cl(\sim').
\end{cases}\\
2.&\Longleftrightarrow\begin{cases}
\ic(f/{\sim}')=\ic(f)\\
\ic(f\mid\sim')=\cl(\sim')\\
\ic((f\mid\sim')/{\sim})=\ic(f\mid\sim')\\
\ic((f\mid\sim')\mid\sim)=\cl(\sim)
\end{cases}
&\hspace{-10mm}&\Longleftrightarrow\begin{cases}
\ic(f/{\sim}')=\ic(f)\\
\ic(f\mid\sim')=\cl(\sim')\\
\ic((f\mid\sim')/{\sim})=\cl(\sim')\\
\ic(f\mid\sim)=\cl(\sim).
\end{cases}
\end{align*}

$1.\Longrightarrow 2$. It remains to prove that $\ic(f\mid\sim')=\cl(\sim')$. By Lemma \ref{lemma3.5}, $\ic(f\mid\sim')\geq\cl(\sim')$. 
If $\ic(f\mid\sim')>\cl(\sim')$, as $\ic((f\mid\sim')\mid\sim)=\ic(f\mid\sim)=\cl(\sim)$, by Lemma \ref{lemma3.5},
\[\cl(\sim')=\ic((f\mid\sim')/{\sim})\geq\ic(f\mid\sim')>\cl(\sim'),\]
which is a contradiction. So $\ic(f\mid\sim')=\cl(\sim')$.\\

$2.\Longrightarrow 1$. It remains to prove that $\ic(f/{\sim})=\ic(f)$. As $\ic(f\mid\sim)=\cl(\sim)$, by Lemma \ref{lemma3.5}, $\ic(f/{\sim})\geq\ic(f)$. If $\ic(f/{\sim})>\ic(f)$, as
\[\ic((f/{\sim})\mid\overline{\sim'})=\ic((f\mid\sim')/{\sim})=\cl(\overline{\sim'}),\]
Lemma \ref{lemma3.5} gives
\[\ic(f)=\ic(f/{\sim}')=\ic((f/{\sim})/\overline{\sim'})\geq\ic(f/{\sim})>\ic(f),\]
which is a contradiction. So $\ic(f/{\sim})=\ic(f)$. As a conclusion, $\calE^S$ satisfies the $\delta$ condition.\\

4. Let $f\in\rmbool(X)$. Then $f/{=_X}=f$ and $f\mid =_X$ is modular, so 
\begin{align*}
\ic(f/{=_X})&=\ic(f),&\ic(f\mid =_X)&=|X|=\cl(=_X).
\end{align*}
Hence, $=_X\in\calE^S(f)$. Moreover, $f/{\sim}_f^i$ is modular and $f\mid\sim_f^i=f$, so 
\begin{align*}
\ic(f/{\sim}_f^i)&=|X/{\sim}_f^i|=\ic(f),&\ic(f\mid\sim_f^i)&=\ic(f)=\cl(\sim_f^i).
\end{align*}
So $\sim_f^i\in\calE^S(f)$. Let $\sim\in\calE^S(f)$, such that $f\mid\sim$ is modular. Then 
\[\ic(f\mid\sim)=|X|=\cl(\sim).\]
Therefore, $\sim$ is the equality of $X$. Let $\sim\in\calE^S(f)$, such that $f/{\sim}$ is modular. By Lemma \ref{lemma3.8}, $\sim\subseteq\sim_f^i$. Moreover, 
\[\ic(f/{\sim})=\ic(f)=\cl(\sim_f^i).\]
Note that $(f\mid\sim_f^i)\mid\sim=f\mid\sim$, so $\ic((f\mid\sim_f^i)\mid\sim)=\ic(f\mid\sim)=\cl(\sim)$, as $\sim\in\calE^S(f)$. By Lemma \ref{lemma3.5}, $\ic(f/{\sim})=\cl(\sim_f^i)\geq\cl(\sim)$, so, as $\sim\subseteq\sim_f^i$,
$\sim=\sim_f^i$. As a conclusion, $\calE^S$ satisfies the counit condition.
\end{proof}

\begin{example}\label{ex3.5}
This example illustrates that in the $\Delta$ condition for $\calE^S$, 1. does not imply 2.
Let us take $X=\{1\}$ and $Y=\{2,3\}$. We consider $f\in\rmbool([3])$ such that
\begin{align*}
f(\{2,3\})&\neq f(\{2\})+f(\{3\}),&f(\{1,2\})&\neq f(\{1\})+f(\{2\}),&f(\{1,2,3\})&=f(\{1\})+f(\{2,3\}).
\end{align*}
As $|X|=1$, $f_{\mid X}$ is obviously indecomposable. By Proposition \ref{prop2.17} with $x=2$, $f$ and $f_{\mid Y}$ are indecomposable. We obtain
\[\ic(f)=\ic(f_{\mid X})=\ic(f_{\mid Y})=1.\]
Let $\sim\in\calE(\{1,2,3\})$ whose classes are $X$ and $Y$. Then $\sim_X=\sim\cap X^2=X^2$ and $\sim_Y=\sim\cap Y^2=Y^2$.
As $|X/{\sim}_X|=|Y/{\sim}_Y|=1$, $f_{\mid X}/{\sim}_X$ and $f_{\mid Y}/{\sim}_Y$ are indecomposable, so
\begin{align*}
\ic(f_{\mid X}/{\sim}_X)&=\ic(f_{\mid X})=\ic(f_{\mid Y}/{\sim}_Y)=\ic(f_{\mid Y})=1.
\end{align*} 
Moreover, $f_{\mid X}/{\sim}_X=f_{\mid X}$ and $f_{\mid Y}=f_{\mid Y}$, so
\begin{align*}
\ic(f_{\mid X}\mid\sim X)&=\ic(f_{\mid X})=1=\cl(\sim_X),&\ic(f_{\mid Y}\mid\sim Y)&=\ic(f_{\mid Y})=1=\cl(\sim_Y),
\end{align*}
We obtain that $\sim_X\in\calE^S(f_{\mid X})$ and $\sim_Y\in\calE^SW(f_{\mid Y})$. 
Let us now compute $f/{\sim}$. 
\begin{align*}
f/{\sim}(\{\overline{1}\})&=f(\{1\}),&f/{\sim}(\{\overline{2}\})&=f(\{2,3\}),& f/{\sim}(\{\overline{1},\overline{2}\})&=f(\{1,2,3\})=f(\{1\})+f(\{2,3\}).
\end{align*}
so $f/\sim$ is modular, and $\ic(f/{\sim})=2>\ic(f)=1$: $\sim\notin\calE^S(f)$. 
\end{example}

\begin{example}\label{ex3.6}
This example illustrates that in the $\Delta$ condition for $\calE^S$, 2. does not imply 1. 
Let us take $X=[3]$ and $Y=\{4\}$. We consider $f\in\rmbool([4])$ such that
\begin{align*}
&\forall i,j\in [4],\mbox{ with }i\neq j,&f(\{i,j\})&\neq f(\{i\})+f(\{j\}),\\
&&f(\{1,2,3\})&=f(\{1,2\})+f(\{3\}),\\
&&f(\{1,2,4\})&\neq f(\{1,2\})+f(\{4\}).
\end{align*}
By Proposition \ref{prop2.17} with any $x\in [4]$, any restriction of $f$ is indecomposable, which implies that $\calE^W(f)=\calE([4])$. We consider the relation $\sim$ whose classes are $[2]$, $\{3\}$ and $\{4\}$. As usual, we put $\sim_X=\sim\cap X^2$, whose classes are $[2]$ and $\{3\}$
and $\sim_Y=\sim\cap Y^2=Y^2$. As $\calE^W(f)=\calE([4])$, $\ic(f\mid\sim)=\cl(\sim)$. Moreover,
\begin{align*}
f/{\sim}(\{\overline{1},\overline{4}\})&=f(\{1,2,4\})\neq f(\{1,2\})+f(\{4\})=f/{\sim}(\{\overline{1}\})+f/{\sim}(\{\overline{4}\}),\\
f/{\sim}(\{\overline{3},\overline{4}\})&=f(\{3,4\})\neq f(\{3\})+f(\{4\})=f/{\sim}(\{\overline{3}\})+f/{\sim}(\{\overline{4}\}).
\end{align*}
By Proposition \ref{prop2.17} with $x=\overline{4}$, $f/{\sim}$ is indecomposable, so $\ic(f/{\sim})=\ic(f)=1$. Therefore, $\sim\in\calE^S(f)$. 
On the other hand, 
\[f_{\mid X}/{\sim}_X(\{\overline{1},\overline{3}\})=f(\{1,2,3\})=f(\{1,2\})+f(\{3\})=f_{\mid X}/{\sim}_X(\{\overline{1}\})+f_{\mid X}/{\sim}_X(\{\overline{3}\}).\]
By Proposition \ref{prop2.18}, $f_{\mid X}/{\sim}_X$ is decomposable, so $\ic(f_{\mid X}/{\sim}_X)=2>\ic(f_{\mid X})$: $\sim_X\notin\calE^S(f_{\mid X})$. 
\end{example}

\begin{cor}
We define a coproduct on $\calH_\bfbool$ as follows: 
\begin{align*}
&\forall f\in\rmbool([n]),&\delta^S(\overline{f})&=\sum_{\sim\in\calE^S(f)}\overline{f/{\sim}}\otimes\overline{f\mid\sim}.
\end{align*}
Then:
\begin{enumerate}
\item $(\calH_\bfbool,\star_1,\delta^S)$ is a bialgebra, of counit $\epsilon_\delta$. 
\item $\delta^S$ and $\Delta$ are not compatible.
\item $\delta^S$ and $\varepsilon_\Delta$ are compatible.
\end{enumerate}\end{cor}

\subsection{Convenient boolean functions}

We now go back to our frame $\rmboolt$ and $\calE$, as exposed in Paragraph \ref{section3.2}.

\begin{prop}\label{prop3.19}
We assume that $\calE$ satisfies the $\star_1$, $\delta$, $\Delta$ and $\epsilon$ conditions on $\rmboolt$. 
Then for any $f\in\rmboolt$, $\calE(f)=\calE^S(f)=\calE^W(f)$.
\end{prop}

\begin{proof}
Let $X$ be a finite set, $f\in\rmbool(X)$ and $\sim\in\calE(f)$. As $\calE$ satisfies the $\star_1$ condition, by Lemma \ref{lemma3.8}, $\sim\subseteq\sim_f^i$. Moreover, $\sim_{f/{\sim}}^i\subseteq\overline{\sim_f^i}$. 
Let $\sim'$ be the unique element of $\calE(X)$ such that $\sim\subseteq\sim'$ and $\overline{\sim'}=\sim_{f/{\sim}}^i$. As $\calE$ satisfies the $\epsilon$ condition, $\overline{\sim'}\in\calE(f/{\sim})$.
As $\calE$ satisfies the $\delta$ condition ($1.\Longrightarrow 2.$), $\sim'\in\calE(f)$. By the $\epsilon$ condition, $\sim_f^i\in\calE(f)$. Moreover, $\sim'\in\calE(f)=\calE(f\mid\sim_f^i)$. By the $\delta$ condition
($2.\Longrightarrow 1$), $\overline{\sim_f^i}\in\calE(f/{\sim})$. By Lemma \ref{lemma3.8}, $\overline{\sim_f^i}\subseteq\sim_{f/{\sim}}^i$. We obtain that $\sim_{f/{\sim}}^i=\overline{\sim_f^i}$. Consequently,
\[\ic(f/{\sim})=\cl(\sim_{f/{\sim}}^i)=\cl(\overline{\sim_f^i})=\cl(\sim_f^i)=\ic(f).\] 
Let $Y\in X/{\sim}$ and $Z=X\setminus Y$. Then, by the $\Delta$ condition, $\sim_Y=\sim\cap Y^2\in\calE(f_{\mid Y})$, so $\sim_Y\subseteq\sim_{f_{\mid Y}}^i$. But $\sim_Y$ has a unique class, namely $Y$,
so $\sim_{f_{\mid Y}}$ has also a unique class: $f_{\mid Y}$ is indecomposable. Therefore, $\ic(f\mid\sim)=\cl(\sim)$. We obtain that $\calE(f)\subseteq\calE^S(f)$.\\

Let $\sim\in\calE^W(f)$. We now show that $\sim\in\calE(f)$. Let $Y\in X/{\sim}$. Then $\sim_Y=\sim\cap Y^2=Y^2=\sim_{f_{\mid Y}}^i$,
as $f_{\mid Y}$ is indecomposable. By the $\epsilon$ condition, this belongs to $\calE(f_{\mid Y})$. By the $\Delta$ condition ($2.\Longrightarrow 1.$), after iterations,
\[\sim=\bigsqcup_{Y\in X/{\sim}}\sim_Y\in\calE(f).\]
We obtain that $\calE^W(f)\subseteq\calE(f)$. For any boolean function, $\calE^S(f)\subseteq\calE^W(f)$, which gives the equality.
\end{proof}

\begin{cor}
There is no family $\calE$ on $\rmbool$ satisfying the $\star_1$, $\Delta$, $\delta$ and $\epsilon$ condition.
\end{cor}

\begin{proof}
As there exist boolean functions $f$ such that $\calE^S(f)\neq\calE^W(f)$, see Example \ref{ex3.4}.
\end{proof}

These results lead us to the following definition:

\begin{defi}\label{defi3.21}
Let $\rmboolt$ be a set subspecies of $\rmbool$.
We shall say that it is convenient if:
\begin{itemize}
\item $1\in\rmboolt(\emptyset)$
\item $\rmboolt$ is stable under $\star_1$: if $X,Y$ are disjoint finite sets, $f\in\rmboolt(X)$ and $g\in\rmboolt(Y)$, then $f\star_1 g\in\rmboolt(X\sqcup Y)$.
\item $\rmboolt$ is stable under $\Delta$: if $Y\subseteq X$ are two finite sets and $f\in\rmboolt(X)$, then $f_{\mid Y}\in\rmboolt(Y)$.
\item $\rmboolt$ is stable under $\delta^W$: if $X$ is a finite set, $f\in\rmboolt(X)$ and $\sim\in\calE^W(f)$, then $f/{\sim}\in\rmboolt(X/{\sim})$.
\item For any finite set $X$, for any $f\in\rmboolt(X)$, $\calE^W(f)=\calE^S(f)$.
\end{itemize}\end{defi}

\begin{theo}
Let $\rmboolt$ be a convenient subspecies of $\rmbool$. We denote by $\bfboolt$ its linearization and by $\calH_{\bfboolt}$ the subspace of $\calH_{\bfbool}$ generated by isoclasses of boolean functions in $\rmboolt$. 
Then:
\begin{enumerate}
\item $(\rmboolt,\star_1,\Delta)$ is a twisted bialgebra.
\item There exists a unique family $\calE$ of equivalences such that $\delta^\calE$ satisfies the $\star_1$, $\Delta$, $\delta$ and $\epsilon$ conditions on $\rmboolt$. Moreover, for any finite set $X$, and for any $f\in\rmboolt(X)$,
\begin{align*}
\calE(f)&=\calE^S(f)=\calE^W(f).
\end{align*}
\item $(\calH_{\bfboolt},\star_1,\Delta,\delta^W)=(\calH_{\bfboolt},\star_1,\Delta,\delta^S)$ is a double bialgebra. 
\end{enumerate}\end{theo}

\begin{proof}
1. Immediate, as $\rmboolt$ is stable under $\star_1$ and $\Delta$.\\

2. As $\calE^W(f)=\calE^S(f)$ for any $f\in\rmboolt(X)$,
the restriction of $\calE^W$ to $\rmboolt$ satisfies the $\star_1$ and $\Delta$ conditions (as does $\calE^W$) and the $\delta$ and $\epsilon$ condition (as does $\calE^S$). By Proposition \ref{prop3.19}, it is the unique family of equivalences satisfying all these properties.\\

3. is a direct consequence of 1. and 2. after application of the functor $\calF$. 
\end{proof}

Conversely: 

\begin{theo}
Let $\rmboolt$ be a set subspecies of $\rmbool$. We denote by $\bfboolt$ its linearization. We assume that:
\begin{enumerate}
\item $\bfboolt$ is a twisted subbialgebra of $(\bfbool,\star_1,\Delta)$.
\item There exists a family $\calE$ of equivalences defined on $\rmboolt$ which satisfies the $\star_1$, $\Delta$, $\delta$ and $\epsilon$ conditions.
\end{enumerate}
Then $\rmboolt$ is convenient, and for any finite set $X$, and any $f\in\rmboolt(X)$,
\[\calE(f)=\calE^W(f)=\calE^S(f).\]
\end{theo}

\begin{proof}
As $\bfboolt$ is a subbialgebra of $(\bfbool,\star_1,\Delta)$:
\begin{itemize}
\item $1\in\rmboolt(\emptyset)$.
\item If $X,Y$ are two disjoint finite sets, $f\in\rmboolt(X)$ and $g\in\rmboolt(Y)$, then $f\star_1 g\in\rmboolt(X\sqcup Y)$: $\rmboolt$ is stable under $\star_1$.
\item If $Y\subseteq X$ are two finite sets and $f\in\rmboolt(X)$, then $\Delta_{Y,X\setminus Y}(f)=f_{\mid Y}\otimes f_{\mid X\setminus Y}\in\bfboolt(Y)\otimes\bfboolt(X\setminus Y)$, so $f_{\mid Y}\in\rmboolt(Y)$: $\rmboolt$ is stable under $\Delta$.
\end{itemize}
By Proposition \ref{prop3.19}, for any $f\in\rmboolt(X)$, $\calE(f)=\calE^W(f)=\calE^S(f)$. Moreover, if $\sim\in\calE^W(f)$, $\delta^W_\sim(f)=f/{\sim}\otimes f\mid\sim\in\bfboolt(X/{\sim})\otimes\bfboolt(X)$, so $f/{\sim}\in\rmboolt(X/{\sim})$. We obtain that $\rmboolt$ is convenient.
\end{proof}

\begin{prop}\label{prop3.24}
There exists a convenient subspecies $\rmboolmax$, maximal for the inclusion.
\end{prop}

\begin{proof}
\textit{First step}. Let us define $\rmboolmax(X)$ by induction on $|X|$. If $|X|=0$, then $\rmboolmax(\emptyset)=\{1\}$. Let us assume that $\rmboolmax(Y)$ is defined for any set $Y$ with $|Y|<|X|$. Let $f\in\rmbool(X)$.
Then $f\in\rmboolmax(X)$ if the following conditions are verified:
\begin{enumerate}
\item $\calE^W(f)=\calE^S(f)$.
\item For any $Y\subsetneq X$, $f_{\mid Y}\in\rmboolmax(Y)$.
\item For any $\sim\in\calE^W(f)$, not equal to $=_X$, $f/{\sim}\in\rmboolmax(X/{\sim})$.
\end{enumerate}
This is well defined: if $Y\subsetneq X$, then $|Y|<|X|$ and if $\sim$ is not equal to $=_X$, then $|X/{\sim}|<|X|$. By definition, $\rmboolmax$ verifies the fourth item of Definition \ref{defi3.21}.
If $f\in\rmbool(X)$ and $Y\subseteq X$, then:
\begin{itemize}
\item If $Y=X$, $f_{\mid Y}=f\in\rmboolmax(X)$.
\item If $Y\subsetneq X$, then $f_{\mid Y}\in\rmboolmax(Y)$ by definition of $\rmboolmax$. 
\end{itemize}
So $\rmboolmax$ is stable under $\Delta$ (second item of Definition \ref{defi3.21}).
If $\sim\in\calE^W(f)$, then:
\begin{itemize}
\item If $\sim$ is $=_X$, then $f/{\sim}=f\in\rmboolmax(X)$.
\item Otherwise, $f/{\sim}\in\rmboolmax(X/{\sim})$ by definition of $\rmboolmax$. 
\end{itemize} 
So $\rmboolmax$ is stable under $\delta^W$ (third item of Definition \ref{defi3.21}).\\

\textit{Second step}. Let $\rmboolt$ be a convenient family. Let us prove that $\rmboolt(X)\subseteq\rmboolmax(X)$ for any finite set $X$ by induction on $|X|$.
If $X=\emptyset$, then $\rmboolt(\emptyset)=\{1\}=\rmboolmax(\emptyset)$. Let us assume the result for all sets $Y$ with $|Y|<|X|$. Let $f\in\rmboolt(X)$. 
As $\rmboolt$ is convenient, $\calE^S(f)=\calE^W(f)$. If $Y\subsetneq X$, then, as $\rmboolt$ is convenient, $f_{\mid Y}\in\rmboolt(Y)$. As $|Y|<|X|$, $f_{\mid Y}\in\rmboolmax(Y)$.
Let $\sim\in\calE^W(f)$, not equal to $=_X$. As $\rmboolt$ is convenient, $f/{\sim}\in\rmboolt(X/{\sim})$. As $|X/{\sim}|<|X|$, $f/{\sim}\in\rmboolmax(X/{\sim})$. Therefore, we obtain that $f\in\rmboolmax(X)$.\\

\textit{Last step}. It finally remains to prove that $\rmboolmax$ is stable under $\star_1$: together with the preceding steps, it will show that $\rmboolmax$ is indeed convenient (first and last steps), 
and maximal for the inclusion along convenient subspecies (second step). We consider the subspecies $\rmboolt$ such that for any nonempty finite set $X$,
\[\rmboolt(X)=\bigcup_{\pi\in\parti(X)}\prod_{Y\in\pi}^{\star_1}\rmboolmax(Y).\]
 We also take $\rmboolt(\emptyset)=\{1\}$. In other words, $\rmboolt$ is the twisted submonoid of $(\rmbool,\star_1)$ generated by $\rmboolmax$. It is by construction stable under $\star_1$, and contains $\rmboolmax$.

Let $f\in\rmboolt(X)$. There exist a partition $\{X_1,\ldots,X_k\}\in\parti(X)$ and $f_i\in\rmboolmax(f)$ for $i\in [k]$ such that $f=f_1\star_1\cdots\star_1 f_k$.
As $\calE^W$ and $\calE^S$ both satisfy the $\star_1$ condition, and by the first item of the definition of $\rmboolmax$ applied to $f_i$,
\begin{align*}
\calE^W(f)&=\{\sim_1\sqcup\cdots\sqcup\sim_k\mid\forall i\in [k],\:\sim_i\in\calE^W(f_i)\}\\
&=\{\sim_1\sqcup\cdots\sqcup\sim_k\mid\forall i\in [k],\:\sim_i\in\calE^S(f_i)\}=\calE^S(f).
\end{align*}
Let $Y\subseteq X$. Then
\[f_{\mid Y}={f_1}_{\mid X_1\cap Y}\star_1\ldots\star_1 {f_k}_{\mid X_k\cap Y}.\]
As $\rmboolmax$ satisfies the stability under $\Delta$ (see the first step), for any $i\in [k]$, ${f_i}_{\mid X_i\cap Y}\in\rmboolmax(X_i\cap Y)$, so $f_{\mid Y}\in\rmboolt(Y)$: $\rmboolt$ is stable under $\Delta$.
Let now $\sim\in\calE^W(f)$. We put $\sim_i=\sim\cap X_i^2$. As $\calE^W$ satisfies the $\star_1$ condition, $\sim=\sim_1\sqcup\cdots\sqcup\sim_k$, and
\[f/{\sim}=(f_1/{\sim}_1)\star_1\cdots\star_1 (f_k/{\sim}_k).\]
As $\rmboolmax$ is stable under $\delta^W$ (first step), for any $i\in [k]$, $f_i/{\sim}_i\in\rmboolmax(X_i/{\sim}_i)$, so $f/{\sim}\in\rmboolt(X/{\sim})$: $\rmboolt$ is stable under $\delta^W$, and finally $\rmboolt$ is convenient.
By the second step, $\rmboolt\subseteq\rmboolmax$. As a conclusion, $\rmboolmax=\rmboolt$, so $\rmboolmax$ is stable under $\star_1$.
\end{proof}

\section{Other convenient subspecies of boolean functions}

It seems rather difficult to explicitly describe $\rmboolmax$. We now look for smaller, but easier to handle with, convenient subspecies. Firstly, we work on the lack of counity of $\delta^W$.

\subsection{Counitary boolean functions}

\begin{lemma}\label{lemma4.1}
Let $f\in\rmbool(X)$ and $\sim\in\calE^W(f)$. Then 
\[\calE^W(f/{\sim})\subseteq\{\overline{\sim'}\mid\sim'\in\calE^W(f),\:\sim\subseteq\sim'\}.\]
\end{lemma}

\begin{proof}
Let $\overline{\sim'}\in\calE^W(f/{\sim})$. There exists a unique $\sim'\in\calE(X)$, containing $\sim$, such that for any $x,y\in X$, 
\[\varpi_\sim(x)\overline{\sim'}\varpi_\sim(y)\Longleftrightarrow x\sim' y.\]
Let us assume that $\sim'\notin\calE^W(f)$. By Lemma \ref{lemma3.5}, $\ic(f\mid\sim')>\cl(\sim')$. Moreover, by Lemma \ref{lemma3.2},
\[\ic((f\mid\sim')\mid\sim)=\ic(f\mid\sim)=\cl(\sim),\]
as $\sim\in\calE^W(f)$. By Lemma \ref{lemma3.5},
\[\ic((f\mid\sim')/{\sim})\geq\ic(f\mid\sim')>\cl(\sim').\]
On the other end, $\overline{\sim'}\in\calE^W(f/{\sim})$, so, by Lemma \ref{lemma3.2},
\[\ic((f/{\sim})\mid\overline{\sim'})=\ic((f\mid\sim')/{\sim})=\cl(\sim').\]
This is a contradiction, so $\sim'\in\calE^W(f)$. 
\end{proof}

In order to repair the lack of counity of $\delta^W$, we introduce counitary boolean functions:

\begin{defi}\label{defi4.2}
Let $f\in\rmbool(X)$. 
The following conditions are equivalent:
\begin{enumerate}
\item The unique $\sim\in\calE^W(f)$ such that $f/{\sim}$ is modular is $\sim_f^i$.
\item $\calE^W(f)=\calE^S(f)$.
\end{enumerate}
If $f$ satisfies these conditions, we shall say that $f$ is counitary. 
The subspecies of $\rmbool$ of counitary boolean functions is denoted by $\rmboolc$. The subspace of $\calH_\bfbool$ generated be classes of counitary boolean functions is denoted by $\calH_{\bfboolc}$.
The restrictions of $\delta^W$ and $\delta^S$ to $\bfboolc$ coincide. They have a counit, given by $\epsilon_\delta$.
\end{defi}

\begin{proof}
$2.\Longrightarrow 1$. If $\calE^W(f)=\calE^S(f)$, the fact that $\calE^S$ satisfies the $\epsilon$ condition gives immediately $1$.\\

$1.\Longrightarrow 2$. Let $\sim\in\calE^W(f)$. Let us consider the unique $\sim'\in\calE(X)$, such that $\sim\subseteq\sim'$ and $\overline{\sim'}=\sim_{f/{\sim}}^i$. 
By Lemma \ref{lemma4.1}, $\sim'\in\calE^W(f)$. Moreover, by Lemma \ref{lemma3.2},
\[f/{\sim}'=(f/{\sim})/\overline{\sim'}=(f/{\sim})/{\sim}^i_{f/{\sim}},\]
so $f/{\sim}'$ is modular. By 1., $\sim'=\sim_f^i$. Therefore,
\[\ic(f/{\sim})=\cl(\overline{\sim'})=\cl(\sim')=\cl(\sim_f^i)=\ic(f).\]
So $\sim\in\calE^S(f)$. 
\end{proof}

\begin{remark}
By definition, $\rmboolmax\subseteq\rmboolc$.
\end{remark}

\begin{lemma}
Let $f\in\rmboolc(X)$ and $\sim\in\calE^W(f)$. Then $f/{\sim}\in\rmboolc(X/{\sim})$, and $\sim_{f/{\sim}}^i=\overline{\sim^i_f}$. 
\end{lemma}

\begin{proof}
In the proof of $1.\Longrightarrow 2.$, of Definition \ref{defi4.2}, we obtained that $\sim_{f/{\sim}}^i=\overline{\sim^i_f}$. 
Let $\overline{\sim'}\in\calE^W(f)$, such that $(f/{\sim})/\overline{\sim'}$ is modular. By Lemmas \ref{lemma3.2} and \ref{lemma4.1}, $f/{\sim}'=(f/{\sim})/\overline{\sim'}$ is modular and $\sim'\in\calE^W(f)$. 
As $f$ is counitary, $\sim'=\sim_f^i$, so $\overline{\sim'}=\overline{\sim_f^i}=\sim_{f/{\sim}}^i$.
Therefore, $f/{\sim}$ is counitary.
\end{proof}

\begin{lemma}
Let $X,Y$ be two disjoint finite sets, $f\in\rmbool(X)$ and $g\in\rmbool(Y)$. Then $f\star_1 g$ is counitary if, and only if, both $f$ and $g$ are counitary. 
\end{lemma}

\begin{proof}
Recall that, as $\calE^W$ and $\calE^S$ satisfy the $\star_1$ condition,
\begin{align*}
\calE^W(f\star_1 g)=\{\sim_X\sqcup\sim_Y\mid\sim_X\in\calE^W(f),\:\sim_Y\in\calE^W(g)\},\\
\calE^S(f\star_1 g)=\{\sim_X\sqcup\sim_Y\mid\sim_X\in\calE^S(f),\:\sim_Y\in\calE^S(g)\}.
\end{align*}
The result is then immediate.
\end{proof}

\begin{remark}
As a consequence, $\bfboolc$ is a twisted subalgebra of $(\bfbool,\star_1)$. Unfortunately, it is not stable under $\Delta$ (see Example \ref{ex4.1} below), and we only have
\[\delta^W(\bfboolc)\subseteq\bfboolc\otimes\bfbool.\]
However, by definition of counitary boolean functions, $\epsilon_\delta$ is a counit for the restriction of $\delta^W$ or of $\delta^S$ to $\bfboolc$.
\end{remark}

\begin{example}\label{ex4.1}
Let $f\in\rmbool([4])$ such that
\[f(\{1,2,3\})=f(\{1,2\})+f(\{3\}),\]
and if $I,J$ are two nonempty disjoint subsets of $[4]$, with $\{I,J\}\neq\{\{1,2\},\{3\}\}$, then 
\[f(I\sqcup J)\neq f(I)+f(J).\]
Applying Proposition \ref{prop2.17} for any $x\in [4]$, we obtain that for any nonempty $Y\subseteq [4]$, $f_{\mid Y}$ is indecomposable. So $\calE^W(f)=\calE(X)$. 
If $\sim\in\calE(X)$, applying Proposition \ref{prop2.17} with any $\overline{x}\in [4]/{\sim}$ different from $\{1,2\}$ and $\{3\}$, we obtain that $f/{\sim}$ is indecomposable, so $\ic(f/{\sim})=\ic(f)=1$: $\calE^S(f)=\calE^W(f)$.
As a conclusion, $f\in\rmboolc([4])$. However, we proved in Example \ref{ex3.4} that $f_{\mid [3]}\notin\rmboolc([3])$, as $\calE^S(f_{\mid[3]})\neq\calE^W(f_{\mid[3]})$. So $\bfboolc$ is not stable under $\Delta$. Taking $\sim$ such that $[4]/\sim=\{[3],\{4\}\}$, 
$\sim\in \calE^W(f)$ and $f\mid \sim\notin \rmboolc([4])$. So $\bfboolc$ is not stable under $\delta^W$. 
\end{example}

\subsection{Rigid and hyper-rigid boolean functions}

Let us now define two convenient subspecies of $\rmbool$.

\begin{defi}\label{defi4.5}\begin{enumerate}
\item Let $f\in\rmbool(X)$ be indecomposable. We shall say that it is rigid if for any $A,B\subseteq X$, with $A\cap B=\emptyset$,
\begin{align*}
f(A\sqcup B)=f(A)+f(B)&\Longrightarrow\forall A'\subseteq A,\:\forall B'\subseteq B,\: f(A'\sqcup B')=f(A)+f(B),
\end{align*}
or, equivalently, 
\begin{align*}
f(A\sqcup B)=f(A)+f(B)&\Longrightarrow f_{\mid A\sqcup B}=f_{\mid A}\star_1 f_{\mid B}.
\end{align*}
\item Let $f\in\rmbool(X)$. We shall say that $f$ is rigid if for any indecomposable component $Y$ of $f$, $f_{\mid Y}$ is rigid. 
\end{enumerate}
The subspecies of $\rmbool$ of rigid boolean functions is denoted by $\rmbools$, and its linearization by $\bfbools$. The subspace of $\calH_{\bfbool}$ generated by isoclasses of rigid boolean functions is denoted by $\calH_{\bfbools}$. 
\end{defi}

\begin{defi}\label{defi4.6}\begin{enumerate}
\item Let $f\in\rmbool(X)$ be an indecomposable boolean function. We shall say that it is hyper-rigid if for any $A,B\subseteq X$, with $A\cap B=\emptyset$ and $A,B\neq\emptyset$, then $f(A\sqcup B)\neq f(A)+f(B)$.
\item Let $f\in\rmbool(X)$ be a boolean function. We shall say that it is hyper-rigid if for any indecomposable component $Y$ of $f$, $f_{\mid Y}$ is hyper-rigid. 
\end{enumerate}
The subspecies of $\rmbool$ of hyper-rigid boolean functions is denoted by $\rmboolhs$, and its linearization by $\bfboolhs$.
The subspace of $\calH_{\bfbool}$ generated by isoclasses of hyper-rigid boolean functions is denoted by $\calH_{\bfboolhs}$. 
\end{defi}

\begin{lemma}\label{lemma4.7}
$\rmboolhs\subseteq\rmbools$.
\end{lemma}

\begin{proof}
Let $f\in\rmboolhs(X)$, let us prove that it belongs to $\rmbools(X)$. We firstly assume that $f$ is indecomposable. If $A,B\subseteq X$, such that $f(A\sqcup B)=f(A)+f(B)$, as $f$ is indecomposable and hyper-rigid, $A=\emptyset$ or $B=\emptyset$. 
Let us assume for example that $A=\emptyset$. Then 
\[f_{\mid A\sqcup B}=f_{\mid B}=1\star_1 f_{\mid B}=f_{\mid A}\star_1 f_{\mid B},\]
So $f$ is rigid. In the general case, if $X_1,\ldots,X_k$ are the indecomposable components of $f$, then for any $i\in [k]$, $f_{\mid X_i}$ is indecomposable and hyper-rigid, so is indecomposable and rigid. 
As a conclusion, $f$ is rigid. 
\end{proof}

\begin{prop}\label{prop4.8}
Let $X$ be a set of cardinality 1 or 2. Then $\rmbool(X)=\rmbools(X)=\rmboolhs(X)$.
\end{prop}

\begin{proof}
It is obvious if $|X|=1$. Let us assume that $X=\{x,y\}$. Let $f\in\rmbool(X)$. If $f(\{x,y\})=f(\{x\})+f(\{y\})$, then $f$ is decomposable and its indecomposable components are $\{x\}$ and $\{y\}$. Obviously, $f_{\mid\{x\}}$ and $f_{\mid\{y\}}$ are hyper-rigid, so $f$ is hyper-rigid.
Otherwise, by Proposition \ref{prop2.18}, $f$ is indecomposable and obviously hyper-rigid.
\end{proof}

\begin{lemma}\label{lemma4.9}
Let $f\in\rmbool(X)$. Let us assume that for any $A,B\subseteq X$, with $A\cap B=\emptyset$,
\begin{align*}
f(A\sqcup B)=f(A)+f(B)&\Longrightarrow\forall A'\subseteq A,\:\forall B'\subseteq B,\: f(A'\sqcup B')=f(A)+f(B). 
\end{align*}
Then $f$ is rigid.
\end{lemma}

\begin{proof}
Let $Y$ be an indecomposable component of $f$, and let $A,B\subseteq Y$, disjoint, such that $f_{\mid Y}(A\sqcup B)=f_{\mid Y}(A)+f_{\mid Y}(B)$. This gives $f(A\sqcup B)=f(A)+f(B)$. 
 Let $A'\subseteq A$ and $B'\subseteq B$. By hypothesis on $f$,
\[f_{\mid Y}(A'\sqcup B')=f(A'\sqcup B')=f(A')+f(B')=f_{\mid Y}(A')+f_{\mid Y}(B').\]
So all the indecomposable components of $f$ are rigid, and $f$ is rigid. 
\end{proof}

\begin{remark}
The converse is false: here is a counterexample. Let $X=[2]$ and $Y=\{3,4\}$. Let $f\in\rmbool(X)$ and $g\in\rmbool(Y)$ such that
\begin{align*}
f(\{1,2\})&\neq f(\{1\})+f(\{2\}),\\
g(\{3,4\})&\neq g(\{3\})+g(\{4\}),\\
f(\{1,2\})+g(\{3,4\})&=f(\{1\})+f(\{2\})+g(\{3\})+g(\{4\}).
\end{align*}
By Proposition \ref{prop2.18}, $f$ and $g$ are indecomposable, and rigid by Proposition \ref{prop4.8}.
So $h=f\star_1 g$ is rigid. Let $A=\{1,3\}$ and $B=\{2,4\}$. Then 
\begin{align*}
h(A\sqcup B)&=h([4])\\
&=f(\{1,2\})+g(\{3,4\})\\
&=f(\{1\})+f(\{2\})+g(\{3\})+g(\{4\})=h(\{1,3\})+h(\{2,4\}).
\end{align*}
But for $A'=\{1\}$ and $B'=\{2\}$,
\[h(A'\sqcup B')=f(\{1,2\})\neq f(\{1\})+f(\{2\})=h(\{1\})+h(\{2\}).\] 
\end{remark}

Here is a wide family of rigid boolean functions. Recall that $\theta_1$ is defined in Theorem \ref{theo2.2}.

\begin{prop}\label{prop4.10}
Let $f\in\rmbool(X)$, such that for any $Y\subseteq X$, $f(Y)\in\N$. Then $\theta_1(f)$ is rigid. 
\end{prop}

\begin{proof}
Let $\{A,B\}\in\parti(X)$ such that $\theta_1(f)(A\sqcup B)=\theta_1(f)(A)+\theta_1(f)(B)$. Then
\begin{align*}
\theta_1(f)(A)+\theta_1(f)(B)&=\theta_1(f)(A\sqcup B)\\
&=\sum_{A'\subseteq A}f(A')+\sum_{B'\subseteq B}f(B')+\sum_{\substack{Y\subseteq A\sqcup B,\\ A\cap Y\neq\emptyset,\: B\cap Y\neq\emptyset}} f(Y)\\
&=\theta_1(f)(A)+\theta_1(f)(B)+\sum_{\substack{Y\subseteq A\sqcup B,\\ A\cap Y\neq\emptyset,\: B\cap Y\neq\emptyset}} f(Y).
\end{align*}
As $\theta_1(f)(Y)\in\N$ for any $Y\subseteq X$, if $Y\subseteq A\sqcup B$, $A\cap Y\neq\emptyset$ and $B\cap Y\neq\emptyset$, then $f(Y)=0$. As a consequence, if $A'\subseteq A$ and $B'\subseteq B$,
\begin{align*}
\theta_1(f)(A'\sqcup B')&=\sum_{A''\subseteq A'}f(A'')+\sum_{B''\subseteq B'}f(B'')+0=\theta_1(f)(A')+\theta_1(f)(B').
\end{align*}
By Lemma \ref{lemma4.9}, $\theta_1(f)$ is rigid. 
\end{proof}

\begin{lemma}\label{lemma4.11}\begin{enumerate}
\item Let $X$ and $Y$ be two disjoint finite sets, $f\in\rmbool(X)$ and $g\in\rmbool(Y)$. Then $f\star_1 g$ is rigid if, and only if, $f$ and $g$ are rigid. 
\item Let $f\in\rmbools(X)$. For any $Y\subseteq X$, $f_{\mid Y}$ is rigid.
\item Let $f\in\rmbools(X)$ and $\sim\in\calE^W(f)$. Then $f/{\sim}$ is rigid.
\end{enumerate}\end{lemma}

\begin{proof}
1. This comes from the fact that the indecomposable components of $f\star_1 g$ are the indecomposable components of $f$ and of $g$.\\

2. Let us denote by $X_1,\ldots,X_k$ the indecomposable components of $f$ and by $Y_1,\ldots,Y_l$ the ones of $f_{\mid Y}$. For any $i\in [l]$,
\[f_{\mid Y_i}=(f_{\mid X_1}\star_1\cdots\star_1 f_{\mid X_k})_{\mid Y_i}=f_{\mid X_1\cap Y_i}\star_1\cdots\star_1 f_{\mid X_k\cap Y_i}.\]
As $f_{\mid Y_i}$ is indecomposable, one, and only one, of the $X_j\cap Y_i$ is nonempty. Therefore, there exists a unique $\sigma(i)\in [k]$ such that $Y_i\subseteq X_{\sigma(i)}$.
Let $A,B\subseteq Y_i$, such that $A\cap B=\emptyset$ and $f_{\mid Y_i}(A\sqcup B)=f_{\mid Y_i}(A)+f_{\mid Y_i}(B)$, and let $A'\subseteq A$, $B'\subseteq B$. Firstly,
\[f_{\mid X_{\sigma(i)}}(A\sqcup B)=f_{\mid Y_i}(A\sqcup B)=f_{\mid Y_i}(A)+f_{\mid Y_i}(B)=f_{\mid X_{\sigma(i)}}(A)+f_{\mid X_{\sigma(i)}}(B).\]
As $f_{\mid X_{\sigma(i)}}$ is indecomposable and rigid, 
\[=f_{\mid Y_i}(A'\sqcup B')=f_{\mid X_{\sigma(i)}}(A'\sqcup B')=f_{\mid X_{\sigma(i)}}(A')+f_{\mid X_{\sigma(i)}}(B')=f_{\mid Y_i}(A')+f_{\mid Y_i}(B').\]
So $f_{\mid Y_i}$ is rigid. By definition, $f_{\mid Y}$ is rigid.\\

3. Let $X_1,\ldots,X_k$ be the indecomposable components of $f$. As $\sim\in\calE^W(f)$, $\sim\subseteq\sim_f^i$, so, putting $\sim_i=\sim\cap X_i^2$ for any $i\in [k]$, $\sim=\sim_1\sqcup\cdots\sqcup\sim_k$, and by Lemma \ref{lemma3.4},
\[f/{\sim}=(f_{\mid X_1})/{\sim}_1\star_1\cdots\star_1 (f_{\mid X_k})/{\sim}_k.\]
As a consequence, if $Y$ is an indecomposable component of $f/{\sim}$, then there exists a unique $i\in [k]$ such that $\varpi_\sim^{-1}(Y)\subseteq X_i$. 
Let $\{A,B\}\in\parti(Y)$ such that $f/{\sim}(A\sqcup B)=f/{\sim}(A)+f/{\sim}(B)$. Let $A'\subseteq A$ and $B'\subseteq B$. Firstly,
\begin{align*}
f(\varpi_\sim^{-1}(A)\sqcup\varpi_\sim^{-1}(B))&=f/{\sim}(A\sqcup B)\\
&=f/{\sim}(A)+f/{\sim}(B)\\
&=f_{\mid X_i}(\varpi_\sim^{-1}(A))+f_{\mid X_i}(\varpi_\sim^{-1}(B)).
\end{align*} 
As $f_{\mid X_i}$ is rigid and indecomposable, 
\begin{align*}
f_/{\sim}(A'\sqcup B')&=f_{\mid X_i}(\varpi_\sim^{-1}(A')\sqcup\varpi_\sim^{-1}(B'))\\
&=f_{\mid X_i}(\varpi_\sim^{-1}(A'))+f_{\mid X_i}(\varpi_\sim^{-1}(B'))\\
&=f/{\sim}(A')+f/{\sim}(B').
\end{align*}
So $f/{\sim}$ is rigid.
\end{proof}

\begin{lemma}\label{lemma4.12}\begin{enumerate}
\item Let $X$ and $Y$ be two disjoint finite sets, $f\in\rmbool(X)$ and $g\in\rmbool(Y)$. Then $f\star_1 g$ is hyper-rigid if, and only if, $f$ and $g$ are hyper-rigid. 
\item Let $f\in\rmboolhs(X)$. For any $Y\subseteq X$, $f_{\mid Y}$ is hyper-rigid.
\item Let $f\in\rmboolhs(X)$ and $\sim\in\calE^W(f)$. Then $f/{\sim}$ is hyper-rigid.
\end{enumerate}\end{lemma}

\begin{proof}
1. This comes from the fact that the indecomposable components of $f\star_1 g$ are the indecomposable components of $f$ and $g$.\\

2. Let us denote by $X_1,\ldots,X_k$ the indecomposable components of $f$ and by $Y_1,\ldots,Y_l$ the ones of $f_{\mid Y}$. As in the proof of Lemma \ref{lemma4.11}, for any $i\in [l]$, there exists a unique $\sigma(i)\in [k]$ such that $Y_i\subseteq X_{\sigma(i)}$. Let $\{A,B\}\in\parti(Y_i)$. Then, as $f_{\mid X_{\sigma(i)}}$ is indecomposable and hyper-rigid,
\begin{align*}
f_{\mid Y_i}(A\sqcup B)&=f_{\mid X_{\sigma(i)}}(A\sqcup B)\neq f_{\mid X_{\sigma(i)}}(A)+f_{\mid X_{\sigma(i)}}(A)=f_{\mid Y_i}(A)+f_{\mid Y_i}(B),
\end{align*}
so $f_{\mid Y_i}$ is hyper-rigid, and finally $f_{\mid Y}$ is hyper-rigid.\\

3. Let $Y$ be an indecomposable component of $f/{\sim}$. As in the proof of Lemma \ref{lemma4.11}, there exists an indecomposable component $X_i$ of $f$ such that $\varpi_\sim^{-1}(Y)\subseteq X_i$. Let $\{A,B\}\in\parti(Y)$. As $f_{\mid X_i}$ is indecomposable and hyper-rigid,
\begin{align*}
f/{\sim}(A\sqcup B)&=f(\varpi_\sim^{-1}(A)\sqcup\varpi_\sim^{-1}(B))\\
&=f_{\mid X_i}(\varpi_\sim^{-1}(A)\sqcup\varpi_\sim^{-1}(B))\\
&\neq f_{\mid X_i}(\varpi_\sim^{-1}(A))+f_{\mid X_i}(\varpi_\sim^{-1}(B))=f/{\sim}(A)+f/{\sim}(B).
\end{align*}
So $f/{\sim}$ is hyper-rigid.
\end{proof}

\begin{lemma}\label{lemma4.13}
Let $f\in\rmbools(X)$. Then $\calE^W(f)=\calE^S(f)$. 
\end{lemma}

\begin{proof}
We already have $\calE^S(f)\subseteq\calE^W(f)$. Let $\sim\in\calE^W(f)$, let us prove that $\ic(f/{\sim})=\ic(f)$. As $\ic(f\mid\sim)=\cl(\sim)$, by Lemma \ref{lemma3.5}, $\ic(f/{\sim})\geq\ic(f)$. Let us assume that $\ic(f/{\sim})>\ic(f)$.
There exists an indecomposable component $Y\subseteq X$ of $f$, such that $f_{\mid Y}/{\sim}\cap Y^2$ is decomposable. Let us write $Y=A\sqcup B$, with $A$, $B$ non empty, such that
\[f_{\mid Y}/{\sim}\cap Y^2=(f_{\mid Y}/{\sim}\cap Y^2)_{\mid A}\star_1 (f_{\mid Y}/{\sim}\cap Y^2)_{\mid B}.\]
Then, putting $A'=\varpi_\sim^{-1}(A)$ and $B'=\varpi_\sim^{-1}(B)$,
\[f_{\mid Y}(A'\sqcup B')=f_{\mid Y}/{\sim}\cap Y^2(A\sqcup B)=f_{\mid Y}/{\sim}\cap Y^2(A)+f_{\mid Y}/{\sim}\cap Y^2(B)=f_{\mid Y}(A')+f_{\mid Y}(B').\]
As $Y$ is an indecomposable component of $f$ and $f$ is rigid, $f_{\mid Y}=f_{\mid A'}\star_1 f_{\mid B'}$ is decomposable: this is a contradiction. We deduce that $\ic(f/{\sim})=\ic(f)$, so $\sim\in\calE^S(f)$.\end{proof}

As a conclusion:

\begin{prop}\begin{enumerate}
\item The subspecies $\rmbools$ and $\rmboolhs$ are convenient. 
\item The subspaces $\calH_{\bfbools}$ and $\calH_{\bfboolhs}$ are subbialgebras of $(\calH_{\bfbool},\star_1,\Delta)$, stable under $\delta^S$ and $\delta^W$, on which these two coproducts coincide. 
Both $(\calH_{\bfbools},\star_1,\Delta,\delta^W)$ and $(\calH_{\bfboolhs},\star_1,\Delta,\delta^W)$ are double bialgebras.
\end{enumerate}\end{prop}

\begin{proof}
Immediate consequence of Lemmas \ref{lemma4.11}, \ref{lemma4.12} and \ref{lemma4.13}.
\end{proof}

\begin{prop}\label{prop4.15}
\begin{enumerate}
\item For any finite set $X$, 
\[\rmboolhs(X)\subseteq\rmbools(X)\subseteq\rmboolmax(X)\subseteq\rmboolc(X)\subseteq\rmbool(X).\]
\item Let $X$ be a finite set of cardinality $\leq 2$. Then 
\[\rmboolhs(X)=\rmbools(X)=\rmboolmax(X)=\rmboolc(X)=\rmbool(X).\]
\item Let $X=\{x,y,z\}$ be a set of cardinality 3 and $f\in\rmbool(X)$.
If $f$ is decomposable, then $f\in\rmboolhs(X)$. If $f$ is indecomposable, then 
\begin{align*}
f\in\rmboolhs(X)&\Longleftrightarrow
\left(\begin{array}{c}
f(\{x,y\})\neq f(\{x\})+f(\{y\})\\
\mbox{ and }f(\{x,y,z\})\neq f(\{x,y\})+f(\{z\})
\end{array}\right)\\
&\mbox{and }\left(\begin{array}{c}
f(\{x,z\})\neq f(\{x\})+f(\{z\})\\
\mbox{ and }f(\{x,y,z\})\neq f(\{x,z\})+f(\{y\})
\end{array}\right)\\
&\mbox{and }\left(\begin{array}{c}
f(\{y,z\})\neq f(\{y\})+f(\{z\})\\
\mbox{ and }f(\{x,y,z\})\neq f(\{y,z\})+f(\{x\})
\end{array}\right),\\
f\in\rmbools(X)&\Longleftrightarrow
\left(\begin{array}{c}
f(\{x,y,z\})\neq f(\{x,y\})+f(\{z\})
\end{array}\right)\\
&\mbox{and }\left(\begin{array}{c}
f(\{x,y,z\})\neq f(\{x,z\})+f(\{y\})
\end{array}\right)\\
&\mbox{and }\left(\begin{array}{c}
f(\{x,y,z\})\neq f(\{y,z\})+f(\{x\})
\end{array}\right),\\
f\in\rmboolmax(X)&\Longleftrightarrow
\left(\begin{array}{c}
f(\{x,y\})= f(\{x\})+f(\{y\})\\
\mbox{ or }f(\{x,y,z\})\neq f(\{x,y\})+f(\{z\})
\end{array}\right)\\
&\mbox{and }\left(\begin{array}{c}
f(\{x,z\})= f(\{x\})+f(\{z\})\\
\mbox{ or }f(\{x,y,z\})\neq f(\{x,z\})+f(\{y\})
\end{array}\right)\\
&\mbox{and }\left(\begin{array}{c}
f(\{y,z\})= f(\{y\})+f(\{z\})\\
\mbox{ or }f(\{x,y,z\})\neq f(\{y,z\})+f(\{x\})
\end{array}\right).
\end{align*}
Moreover, $\rmboolhs(X)\subsetneq\rmbools(X)\subsetneq\rmboolmax(X)=\rmboolc(X)\subsetneq\rmbool(X)$.
\end{enumerate}\end{prop}

\begin{proof}
We proved in Lemma \ref{lemma4.7} that for any finite set $X$, $\rmboolhs(X)\subseteq\rmbools(X)$. As $\rmbools$ is convenient, by maximality of $\rmboolmax$, $\rmbools(X)\subseteq\rmboolmax(X)$. By definition of $\rmboolc$, $\rmboolmax(X)\subseteq\rmboolc(X)$. 
Proposition \ref{prop4.8} implies the equality if $|X|=1$ or $2$.\\

Let us now consider the case $X=\{x,y,z\}$ of cardinality 3, and let $f\in\rmbool(X)$. If $f$ is indecomposable, let us take $Y,Z\subset X$ such that $f=f_{\mid Y}\star_1 f_{\mid Z}$. 
Then $|Y|,|Z|\leq 2$, so $f_{\mid Y}$ and $f_{\mid Z}$ are hyper-rigid, and by stability under $\star_1$, $f$ is hyper-rigid. Let us assume that $f$ is indecomposable. 
The definition of hyper-rigidity gives the first equivalence.\\

Let us assume that $f\in\rmbools(X)$. If $f(\{x,y,z\})=f(\{x,y\})+f(\{z\})$, as $f$ in indecomposable and rigid, then for any $A'\subseteq\{x,y\}$ and $B'\subseteq\{z\}$, $f(A'\sqcup B')=f(A')+f(B')$. We deduce that 
$f=f_{\mid\{x,y\}}\star_1 f_{\mid\{z\}}$, which contradicts the indecomposability of $f$. So $f(\{x,y,z\})\neq f(\{x,y\})+f(\{z\})$. The two other inequalities are proved in the same way.
Conversely, let $A,B\subseteq X$, with $A\cap B=\emptyset$, such that $f(A\sqcup B)=f(A)+f(B)$. The hypothesis implies that $|A|\leq 1$ and $|B|\leq 1$. If $A'\subseteq A$ and $B'\subseteq B$, 
then either $A'=A$ and $B'=B$, or $A'=\emptyset$, or $B'=\emptyset$. In all cases, as $f(\emptyset)=0$, $f(A'\sqcup B')=f(A')+f(B')$: $f$ is rigid.\\

 Let us now prove the converse implication. We shall use the characterization of $\rmboolmax$ of the proof of Proposition \ref{prop3.24}. 
As $\rmboolmax(X')=\rmbool(X')$ if $|X'|\leq 2$, items 2 and 3 are obviously satisfied for any $f\in\rmbool(X)$. So $f\in\rmboolmax(X)$ if, and only if, $\calE^W(f)=\calE^S(f)$, that is to say if, and only if, $f\in\rmboolc(X)$. 
Let us make a review of all possible elements of $\calE(X)$.
\begin{itemize}
\item $X/{\sim}=\{\{x,y,z\}\}$. Then, as $f$ is indecomposable and $f/{\sim}$ is also indecomposable, $\sim\in\calE^S(f)$.
\item $\sim$ is the equality of $X$. Then $\sim\in\calE^S(f)$.
\item $X/{\sim}=\{\{x,y\},\{z\}\}$. By Proposition \ref{prop2.18}, $\sim\in\calE^W(f)$ if, and only if, $f(\{x,y\})\neq f(\{x\})+f(\{y\})$. We put $\overline{x}=\{x,y\}$ and $\overline{z}=\{z\}$. Then
\begin{align*}
f/{\sim}(\{\overline{x},\overline{z}\})&=f(\{x,y,z\}),&f/{\sim}(\{\overline{x}\})&=f(\{x,y\}),&f/{\sim}(\{\overline{z}\})&=f(\{z\}).
\end{align*} 
By Proposition \ref{prop2.18}, $f/{\sim}$ is indecomposable if, and only if, $f(\{x,y,z\})\neq f(\{x,y\})+f(\{z\})$.
As a conclusion, $\sim\in\calE^W(f)\setminus\calE^S(f)$ if, and only if, $f(\{x,y\})\neq f(\{y\})+f(\{y\})$ and $f(\{x,y,z\})=f(\{x,y\})+f(\{z\})$.
\item $X/{\sim}=\{\{x,z\},\{y\}\}$. Similarly, $\sim\in\calE^W(f)\setminus\calE^S(f)$ if, and only if, $f(\{x,z\})\neq f(\{x\})+f(\{z\})$ and $f(\{x,y,z\})=f(\{x,z\})+f(\{y\})$.
\item $X/{\sim}=\{\{y,z\},\{x\}\}$. Similarly, $\sim\in\calE^W(f)\setminus\calE^S(f)$ if, and only if, $f(\{y,z\})\neq f(\{y\})+f(\{z\})$ and $f(\{x,y,z\})=f(\{y,z\})+f(\{x\})$.
\end{itemize}
By transposition, we obtain the announced equivalence. The inclusions are strict, as shown in the Examples below.
\end{proof}

\begin{example}
Let $f\in\rmbool([3])$, such that
\begin{align*}
f(\{1,2\})&=f(\{1\})+f(\{2\}),&f(\{1,2,3\})&\neq f(\{1,2\})+f(\{3\}),\\
f(\{1,3\})&\neq f(\{1\})+f(\{3\}),&f(\{1,2,3\})&\neq f(\{1,3\})+f(\{2\}),\\
f(\{2,3\})&\neq f(\{2\})+f(\{3\}),&f(\{1,2,3\})&\neq f(\{2,3\})+f(\{1\}).
\end{align*}
By Proposition \ref{prop2.17} with $x=3$, $f$ is indecomposable. Proposition \ref{prop4.15} shows that $f$ is rigid but not hyper-rigid.
\end{example}

\begin{example}
Let $f\in\rmbool([3])$, such that
\begin{align*}
f(\{1,2\})&=f(\{1\})+f(\{2\}),&f(\{1,2,3\})&=f(\{1,2\})+f(\{3\}),\\
f(\{1,3\})&\neq f(\{1\})+f(\{3\}),&f(\{1,2,3\})&\neq f(\{1,3\})+f(\{2\}),\\
f(\{2,3\})&\neq f(\{2\})+f(\{3\}),&f(\{1,2,3\})&\neq f(\{2,3\})+f(\{1\}).
\end{align*}
By Proposition \ref{prop2.17} with $x=3$, $f$ is indecomposable. Proposition \ref{prop4.15} shows that $f\in\rmboolmax(X)$ but is not rigid.
\end{example}

\begin{example}\label{ex4.4}
Let $f\in\rmbool([3])$, such that
\begin{align*}
f(\{1,2\})&\neq f(\{1\})+f(\{2\}),&f(\{1,2,3\})&=f(\{1,2\})+f(\{3\}).\\
f(\{1,3\})&\neq f(\{1\})+f(\{3\}),\\
f(\{2,3\})&\neq f(\{2\})+f(\{3\}).
\end{align*}
By Proposition \ref{prop2.17} with any $x$, $f$ is indecomposable. Proposition \ref{prop4.15} shows that $f\notin\rmboolmax(X)$.
\end{example}

\begin{example}
If $|X|\geq 4$, $\rmboolmax(X)\subsetneq\rmboolc(X)$, as we now show.
Let $f\in\rmbool([4])$ such that $f(\{1,2,3\})=f(\{1,2\})+f(\{3\})$ and if $\{A,B\}$ is a pair of two disjoint nonempty subsets of $[4]$, different from $\{\{1,2\},\{3\}\}$, then $f(A\sqcup B)\neq f(A)+f(B)$.
Let $\sim\in\calE([4])$. For any nonempty $Y\subseteq [4]$, by Proposition \ref{prop2.17} with any $x\in Y$, $f_{\mid Y}$ is indecomposable, so $\ic(f\mid\sim)=\cl(\sim)$. 
Let $\sim\in\calE^W(f)$. By Proposition \ref{prop2.17} with $x=\overline{4}\in X/{\sim}$, $f/{\sim}$ is indecomposable, so $\ic(f/{\sim})=1=\ic(f)$.
We obtain that $\calE^S(f)=\calE([4])$, so $\calE^S(f)=\calE^W(f)=\calE([4])$, and $f\in\rmboolc([4])$. But, as $f_{\mid [3]}\notin\rmboolmax([3])$ (see Example \ref{ex4.4}), $f\notin\rmboolmax([4])$.
\end{example}

\subsection{Examples of rigid boolean functions: hypergraphs}

Reminders on the twisted bialgebra $\bfHG$ have been given in Section \ref{section1}. 

\begin{prop}\label{prop4.16}
Let $H$ be a hypergraph on $X$. We define two boolean functions by
\begin{align*}
\iota(H)&:\left\{\begin{array}{rcl}
\calP(X)&\longrightarrow&\Z\\
A&\longmapsto&\begin{cases}
1\mbox{ if }A\in E(H),\\
0\mbox{ otherwise},
\end{cases}
\end{array}\right.\\
\gamma(H)=\theta_1\circ\iota(H)&:\left\{\begin{array}{rcl}
\calP(X)&\longrightarrow&\Z\\
A&\longmapsto&|\{Y\in E(H)\mid Y\subseteq A\}|.
\end{array}\right.\end{align*}
Then $\iota$ is a injective twisted bialgebra morphism from $(\bfHG,m,\Delta)$ to $(\bfbool,\star_0,\Delta)$ and $\gamma$ is an injective twisted bialgebra morphism from $(\bfHG,m,\Delta)$ to $(\bfbools,\star_1,\Delta)$. Moreover,
\begin{align*}
\epsilon_\delta\circ\gamma&=\epsilon_\delta.
\end{align*}\end{prop}

\begin{proof}
Obviously, $\iota$ is injective. Let $X,Y$ be two disjoint sets. If $G\in\rmHG(X)$ and $H\in\rmHG(Y)$, then $\iota(GH)=\iota(G)\star_0\iota(H)$, so $\iota$ is a twisted algebra morphism.
If $G\in\rmbool(X\sqcup Y)$, then $\iota(G_{\mid X})=\iota(G)_{\mid X}$ and $\iota(G_{\mid Y})=\iota(G)_{\mid Y}$: this implies that $\iota$ is a twisted coalgebra morphism. 
By proposition \ref{prop4.10}, $\gamma=\theta_1\circ\iota$ takes its values in $\bfbools$. By composition, $\gamma=\theta_1\circ\iota$
is an injective twisted bialgebra morphism from $(\bfHG,m,\Delta)$ to $(\bfbools,\star_1,\Delta)$.\\

Let $H$ be a hypergraph.
\begin{align*}
\epsilon_\delta\circ\gamma(H)
&=\begin{cases}
1\mbox{ if $\gamma(H)$ is modular},\\
0\mbox{ otherwise}
\end{cases}\\
&=\begin{cases}
1\mbox{ if }\forall A\subseteq X,\: |\{Y\in E(H)\mid Y\subseteq A\}|=|\{x\in A\mid\{x\}\in E(H)\}|,\\
0\mbox{ otherwise}
\end{cases}\\
&=\begin{cases}
1\mbox{ if }\forall Y\in E(H),\: |Y|=1,\\
0\mbox{ otherwise}.
\end{cases}\\
&=\epsilon_\delta(H).\qedhere
\end{align*}\end{proof}

If $H$ is a hypergraph, then $\gamma(H)$ is rigid, but generally not hyper-rigid, as shown by the following example.

\begin{example}
Let $H$ be a hypergraph on $[3]$ with hyperedges $\{1\}$, $\{2\}$, $\{3\}$ and $[3]$. This gives
\begin{align*}
\gamma(H)(\{1\})&=1,&\gamma(H)(\{1,2\})&=2,&\gamma(H)([3])&=4.\\
\gamma(H)(\{1\})&=1,&\gamma(H)(\{1,3\})&=2,\\
\gamma(H)(\{1\})&=1,&\gamma(H)(\{2,3\})&=2.
\end{align*}
Then $\gamma(H)$ is indecomposable (see Proposition \ref{prop2.18}). As $\gamma(H)(\{1,2\})=\gamma(H)(\{1\})+\gamma(H)(\{2\})$, $\gamma(H)$ is not hyper-rigid. 
\end{example}
Applying the functor $\calF$:

\begin{theo}\label{theo4.17}
The map $\gamma$ induces an injective bialgebra morphism, also denoted by $\gamma$, from $(\calH_{\bfHG},m,\Delta)$ to $(\calH_{\bfbools},\star_1,\Delta)$, sending any isoclass $\overline{G}$ of hypergraphs to $\overline{\gamma(G)}$. Moreover,
\begin{align*}
\epsilon_\delta\circ\gamma&=\epsilon_\delta.
\end{align*}
\end{theo}

\begin{remark}
Unfortunately, $\gamma(\bfHG)$ is not stable under $\delta$. Here is a counterexample. Let $H$ be the hypergraph on $[3]$ such that
\[E(H)=\{\{1\},\{2\},\{1,2\}\}.\]
Let $\sim\in\calE([3])$ whose classes are $\overline{1}=[2]$ and $\overline{3}=\{3\}$. Then 
\[\gamma(H)/{\sim}(\{\overline{1}\})=\gamma(H)([2])=3,\]
so $\gamma(H)/{\sim}$ is not the boolean function of any hypergraph on $\{\overline{1},\overline{3}\}$.
However,
\[\delta^W(\calH_{\bfHG})\subseteq\calH_{\bfbools}\otimes\calH_{\bfHG}.\]
\end{remark}

Let us recall the following classical definition on hypergraphs:

\begin{defi}
\item Let $H$ be a hypergraph on $X$. We shall say that it is connected if for any $X',X''$, nonempty, such that $X=X'\sqcup X''$, there exists $Y\in E(H)$ such that $Y\cap X'\neq\emptyset$ and $Y\cap X''\neq\emptyset$. 
\end{defi}

\begin{prop}\label{prop4.19}
Let $H$ be a hypergraph. Then $H$ is connected if, and only if, $\gamma(H)$ is indecomposable.
\end{prop}

\begin{proof}
$\Longrightarrow$. Let us assume that $\gamma(H)$ is decomposable and let us write $\gamma(H)=\gamma(H)_{\mid X'}\star_1\gamma(H)_{\mid X''}$, with $X'\sqcup X''=X$ and $X',X''\neq\emptyset$. As a consequence,
\[\gamma(H)(X)=\gamma(H)(X')+\gamma(H)(X'').\]
In other words, if $Y\in E(H)$, then $Y\subseteq X'$ or $Y\subseteq X''$, so $H$ is not connected.\\

$\Longleftarrow$. Let us assume that $H$ is not connected. Then we can write $H=H_{\mid X'}\star_1 H_{\mid X''}$, with $X'\sqcup X''=X$ and $X',X''\neq\emptyset$. As a consequence,
\[\gamma(H)=\gamma(H)_{\mid X'}\star_1\gamma(H)_{\mid X''},\]
so $\gamma(H)$ is decomposable.
\end{proof}


\subsection{Examples of rigid boolean functions: matroids}

\begin{defi}\label{defi4.20}
Let $f\in\rmbool(X)$. We shall say that it is a rank function of a matroid if:
\begin{enumerate}
\item For any $A\subseteq X$, $0\leq f(A)\leq |A|$.
\item $f$ is increasing: for any $A\subseteq B\subseteq X$, $f(A)\leq f(B)$.
\item $f$ is submodular: for any $A,B\subseteq X$, $f(A\cup B)+f(A\cap B)\leq f(A)+f(B)$.
\end{enumerate}
The subspecies of $\rmbool$ of matroid rank functions is denoted by $\rmboolm$. The subspace of $\calH_\bfbool$ generated by isoclasses of elements of $\rmboolm$ is denoted by $\calH_{\bfboolm}$.
\end{defi}

\begin{lemma}\label{lemma4.21}\begin{enumerate}
\item Let $X$, $Y$ be two disjoint finite sets, $f\in\rmbool(X)$ and $g\in\rmbool(Y)$. Then $f\star_1 g\in\rmboolm(X\sqcup Y)$ if, and only if, $f\in\rmboolm(X)$ and $g\in\rmboolm(Y)$.
\item Let $Y\subseteq X$ be two finite sets and $f\in\rmboolm(X)$. Then $f_{\mid Y}\in\rmboolm(Y)$.
\end{enumerate}\end{lemma}

\begin{proof}
Direct verifications.
\end{proof}

However, matroid rank functions are not stable under contractions:

\begin{lemma}\label{lemma4.22}
Let $f\in\rmboolm(X)$ and $\sim\in\calE(X)$. Then $f/{\sim}$ is increasing and submodular. It is a matroid rank function if, and only if, for any $Y\in X/{\sim}$, $f(Y)\leq 1$.
\end{lemma}

\begin{proof}
Let $A\subseteq B\subseteq X/{\sim}$. 
\[f/{\sim}(A)=f(\varpi_\sim^{-1}(A))\leq f(\varpi_\sim^{-1}(B))=f/{\sim}(B).\]
Let $A,B\subseteq X/{\sim}$. 
\begin{align*}
f/{\sim}(A\cup B)+f/{\sim}(A\cap B)&=f(\varpi_\sim^{-1}(A)\cup\varpi_\sim^{-1}(B))+f(\varpi_\sim^{-1}(A)\cap\varpi_\sim^{-1}(B))\\
&\leq f(\varpi_\sim^{-1}(A))+f(\varpi_\sim^{-1}(B))\\
&=f/{\sim}(A)+f/{\sim}(B).
\end{align*}
Let us assume that $f/{\sim}$ is a matroid rank function. For any $Y\in X/{\sim}$,
\[f/{\sim}(\{Y\})=f(Y)\leq |\{Y\}|=1.\]
Let us assume that for any $Y\in X/{\sim}$, $f(Y)\leq 1$, and let us prove that for any $Y\subseteq X/{\sim}$, $f/{\sim}(Y)\leq |Y|$. We proceed by induction on $n=|Y|$. It is immediate if $Y=\emptyset$, and the hypothesis on $f$ if $n=1$.
Let us assume the result at rank $n-1$, with $n\geq 2$. Let us pick an element $\overline{y}\in Y$ and put $Y'=Y\setminus\{\overline{y}\}$. Then, as $f/{\sim}$ submodular, by the induction hypothesis,
\begin{align*}
f/{\sim}(Y)&\leq f/{\sim}(Y')+f/{\sim}(\{\overline{y}\})-f/{\sim}(\emptyset)\leq n-1+f(\overline{y})-0\leq n.
\end{align*}
So $f/{\sim}$ is a matroid rank function. 
\end{proof}

Let us gives this classical result on separators of matroids \cite[Lemma 5.8]{Vyuka}.
For the sake of completeness, and to avoid the use of the specific vocabulary of matroids (cycles, independents, etc) as far as possible, we give a proof of this Proposition in the appendix.\\

\begin{prop}\label{prop4.23}
Let $f\in\rmboolm(X)$, and let $A,B\subseteq X$, disjoint, such that
$f(A\sqcup B)=f(A)+f(B)$. For any $A'\subseteq A$ and any $B'\subseteq B$,
\[f(A'\sqcup B')=f(A')+f(B').\] 
\end{prop}

 Lemma \ref{lemma4.21} implies that:

\begin{prop}
For any finite set $X$, $\rmboolm(X)\subseteq\rmbools(X)$.
\end{prop}

\begin{proof}
We shall use Lemma \ref{lemma4.9}. Let $f\in\rmboolm(X)$ and $\{A,B\}\in\parti(X)$ such that $f(A\sqcup B)=f(A)+f(B)$. We put $g=f_{\mid A\sqcup B}$. Then $g$ is the rank function of a certain matroid $M$ on $A\sqcup B$.
By Proposition \ref{prop4.23}, for any $A'\subseteq A$ and $B'\subseteq B$, 
\[f(A'\sqcup B')=g(A'\sqcup B')=g(A')+g(B')=f(A')+f(B').\]
So $f$ is rigid. 
\end{proof}

In general, matroid rank functions are not hyper-rigid, nor of the form $\theta_1(f')$ with $f'$ taking its values in $\N$ (as in Proposition \ref{prop4.10}), as shown by the following example.

\begin{example}
Let $f\in\rmboolm([3])$ defined by
\begin{align*}
f(\{1\})=f(\{2\})=f(\{3\})&=1,\\
f(\{1,2\})=f(\{1,3\})=f(\{2,3\})&=2,\\
f(\{1,2,3\})&=2.
\end{align*}
As $f(\{1,2,3\})\neq f(\{i\})+f(\{j,k\})$ if $\{i,j,k\}=\{1,2,3\}$, $f$ is indecomposable. As $f(\{1,2\})=f(\{1\})+f(\{2\})$, $f$ is not hyper-rigid. 
If $\theta_1(f')=f$, then $f'=\theta_{-1}(f)$, so
\begin{align*}
f'(\{1,2,3\})&=f(\{1,2,3\})-f(\{1,2\})-f(\{1,3\})-f(\{2,3\})+f(\{1\})+f(\{2\})+f(\{3\})-f(\emptyset)\\
&=2-2-2-2+1+1+1-0\\
&=-1,
\end{align*}
so $f'$ does not take its values in $\N$. 
\end{example}

As a consequence of all these results:

\begin{theo}\label{theo4.25}
The subspace $\calH_{\bfboolm}$ is a subbialgebra of $(\calH_{\bfbools},\star_1,\Delta)$, not stable under $\delta^W$. However, 
\[\delta^W\left(\calH_{\bfboolm}\right)\subseteq\calH_{\bfbools}\otimes\calH_{\bfboolm}.\]
\end{theo}

\begin{remark}
Let $f\in\rmboolm(X)$. Then $f$ is modular if, and only if, 
\begin{align*}
&\forall A\subseteq X,&f(A)&=|\{a\in A\mid f(\{a\})=1\}|.
\end{align*}\end{remark}

\subsection{Two subfamilies of matroids}

There are two classical ways to obtain matroids: from graphs (graphic matroids), and from families of vectors (linear matroids). Let us firstly detail graphic matroids.

\begin{defi}\label{defi4.26}\begin{enumerate}
\item Let $X$ be a finite set. We here call graph on $X$ a pair $G=(V(G),e_G)$, where $V(G)$ is a finite set (set of vertices of $G$) and $e_G$ is a map from $X$ to the set $\calP_2(V(G))$ of subsets of $V(G)$ of cardinality $2$. 
In other terms, a graph on $X$ is here a multigraph whose edges are indexed by the set $X$. We admit multiple edges, but not loops.
\item Let $G$ be a graph on $X$.
\begin{enumerate}
\item The number of connected components of $G$ is denoted by $\cc(G)$.
\item If $Y$ is a subset of $X$, we denote by $G_{\mid Y}$ the graph $(V(G_{\mid Y}),{e_G}_{\mid Y})$, with
\begin{align*}
V(G_{\mid Y})&=\bigcup_{y\in Y}e_G(y).
\end{align*}
\item For any $Y\subseteq X$, we put $\rk_G(Y)=|V(G_{\mid Y})|-\cc(G_{\mid Y})$. This defines a rank function of a certain matroid on $X$. These matroids are called graphic matroids. 
\end{enumerate}
\item The subspecies of $\rmboolm$ given by rank functions of graphic matroids is denoted by $\rmboolgm$. The subspace of $\calH_{\rmboolm}$ generated by isoclasses of rank functions of graphic matroids is denoted by $\calH_{\bfboolgm}$.
\end{enumerate}\end{defi}

\begin{prop}
The subspace $\calH_{\bfboolgm}$ is a subbialgebra of $(\bfH_{\bfboolm},\star_1,\Delta)$, not stable under $\delta^W$. However, 
\[\delta^W(\calH_{\bfboolgm})\subseteq\calH_{\bfbools}\otimes\calH_{\bfboolgm}.\]
\end{prop}

\begin{proof}
Let $X,Y$ be two disjoint finite sets and let $G,H$ be two graphs on respectively $X$ and $Y$. We can without loss of generality assume that $V(G)$ and $V(H)$ are disjoint. Let $GH$ be the graph such that $V(GH)=V(G)\sqcup V(H)$ and $e_{GH}$ defined by
\begin{align*}
&\forall x\in X\sqcup Y,&e_{GH}(x)&=\begin{cases}
e_G(x)\mbox{ if }x\in X,\\
e_H(x)\mbox{ if }x\in Y.
\end{cases}\end{align*}

Then $\rk_{GH}=\rk_G\star_1\rk_H$, so $\GM$ is stable under $\star_1$. 
Let $G$ be a graph on $X$ and $A\subseteq X$. Then
\[(\rk_G)_{\mid A}=\rk_{G_{\mid A}}.\]
These two properties imply the result.
\end{proof}

\begin{prop}
Let $G$ be a graph. Then $\rk_G$ is modular if, and only if, $G$ is a forest.
\end{prop}

\begin{proof}
Let us assume that $G$ is a forest. If $A\subseteq X$, then $G_{\mid A}$ is also a forest, so 
\[|V(G_{\mid A})|=|A|+\cc(G_{\mid A}).\]
So for any $A\subseteq X$, $\rk_G(A)=|A|$, and $\rk_G$ is obviously modular.\\

Let us assume that $\rk_G$ is modular. For any $x\in X$, $G_{\mid\{x\}}$ is a graph with two vertices and a single edge, so 
\[\rk_G(\{x\})=2-1=1.\]
As $\rk_G$ is modular, for any $A\subseteq X$, $\rk_G(A)=|A|$. In particular, $\rk_G(X)=|X|$.
Denoting by $G'$ the graph obtained from $G$ by deletion of the $k$ isolated vertices of $G$, 
\begin{align*}
G_{\mid X}&=G',&|V(G')|&=|V(G)|-k,&\cc(G')&=\cc(G)-k,
\end{align*}
which gives
\[\rk_G(X)=|V(G')|-\cc(G')=|E(G')|,\]
so $G'$ is a forest, which implies that $G$ is a forest too.
\end{proof}

\begin{remark}
The correspondence sending a graph to its associated graphic matroid is not injective. For example, if $F_1$ and $F_2$ are two forests with on the same set $X$, then $\rk_{F_1}=\rk_{F_2}$. 
\end{remark}

\begin{prop}
Let $G$ be a graph on $X$ and $\sim\in\calE(X)$. Then $\rk_G/{\sim}$ is the rank function of a matroid if, and only if,
\begin{align*}
&\forall x,y\in X,&x\sim y\Longrightarrow e_G(x)=e_G(y).
\end{align*}
If so, $\rk_{G/\sim}$ is the rank function of a graphic matroid.
\end{prop}

\begin{proof}
$\Longrightarrow$. If $\rk_G/{\sim}$ is a rank function of a matroid, then by Lemma \ref{lemma4.22}, for any $Y\in X/{\sim}$, $\rk_G(Y)\leq 1$, that is to say
\[V(G_{\mid Y})\in\{\cc(G_{\mid Y}),\:\cc(G_{\mid Y})+1\}.\]
If $V(G_{\mid Y})=\cc(G_{\mid Y})$, then $G_{\mid Y}$ has no edge, which is impossible as $Y\neq\emptyset$. So $V(G_{\mid Y})=\cc(G_{\mid Y})+1$, which implies that all the edges of $G_{\mid Y}$ have the same extremities: 
in other words, for any $x,y\in X$, if $x\sim y$, then $e_G(x)=e_G(y)$.\\

$\Longleftarrow$. If so, we can define a map $\overline{e}_G:X/{\sim}\longrightarrow\calP_2(V(G))$, such $\overline{e}_G\circ\varpi_\sim=e_G$. This defines a graph $G/{\sim}$ on $X/{\sim}$, with $V(G/{\sim})=V(G)$. 
Then $\rk_G/{\sim}=\rk_{G/{\sim}}$, so is the rank function of a graphic matroid.
\end{proof}

Let us now study linear matroids. We fix a commutative field $\L$.

\begin{defi}\label{defi4.30}\begin{enumerate}
\item Let $X$ be a finite set and $V=(v_x)_{x\in X}$ a family of vectors of a given $\L$-vector space $E$, called the underlying space of $V$. For any $Y\subseteq X$, we define
\[\rk_V(Y)=\rk((v_x)_{x\in Y}).\]
This defines a matroid rank function, associated to a certain matroid on $X$. These matroids are said to be linear over $\L$.
\item We denote by $\rmboollm$ the subspecies of $\rmboolm$ of rank functions of linear matroids over $\L$, and by $\calH_{\bfboollm}$ the subspace of $\calH_{\bfboolm}$ generated by isoclasses of rank functions of linear matroids over $\L$.
\end{enumerate}\end{defi}

\begin{prop}
The subspace $\calH_{\bfboollm}$ is a subbialgebra of $(\bfH_{\bfboolm},\star_1,\Delta)$, not stable under $\delta^W$. However, 
\[\delta^W(\calH_{\bfboollm})\subseteq\calH_{\bfbools}\otimes\calH_{\bfboollm}.\]
\end{prop}

\begin{proof}
Let $V,V'$ be two families of vectors indexed by $X$ and $X'$, of underlying spaces $E$ and $E'$.
We consider the family of vectors of $E\oplus E'$ defined by
\[V\oplus V'=(v_x)_{x\in X}\sqcup (v'_x)_{x\in X'}.\]
Then $\rk_{V\oplus V'}=\rk_V\star_1\rk_{V'}$. 
If $V$ is a family of vectors indexed by $X$ and $Y\subseteq V$, then, obviously, $(\rk_V)_{\mid Y}$ is the rank boolean function attached to $(v_x)_{x\in Y}$. These two properties imply the result.
\end{proof}

\begin{prop}
Let $V=(v_x)_{x\in X}$ be a family of vectors. Then $\rk_V$ is modular if, and only if, the family $V_0$ obtained from $V$ by deletion of the zero vectors is free.
\end{prop}

\begin{proof}
We denote by $X_0\subseteq X$ the subsets of indices $x$ such that $v_x=0$ and we put $X'=X\setminus X_0$. If $\rk_V$ is modular, then 
\[\rk((v_x)_{x\in X'})=\rk_V(X')=\sum_{x\in X'}\rk((v_x))=|X'|,\]
so $(v_x)_{x\in X'}$ is free. Conversely, if this family is free, then for any $Y\subseteq X'$,
\[\rk_V(Y)=|Y\cap X'|=\sum_{x\in Y\cap X'}\rk_V(\{x\})+\sum_{x\in Y\cap X_0}\rk_V(\{x\})=\sum_{x\in Y}\rk_V(\{x\}),\]
so $\rk_V$ is modular. 
\end{proof}

\begin{prop}
Let $V=(v_x)_{x\in X}$ be a family of vectors $\sim\in\calE(X)$. Then $\rk_V/{\sim}$ is the rank function of a matroid if, and only if,
\begin{align*}
&\forall x,y\in X,&x\sim y\Longrightarrow\mbox{$v_x$ and $v_y$ are colinear}.
\end{align*}
If so, $\rk_{G/\sim}$ is the rank function of a linear matroid.
\end{prop}

\begin{proof}
$\Longrightarrow$. If $\rk_V/{\sim}$ is a rank function of a matroid, then, by Lemma \ref{lemma4.22}, for any $Y\in X/{\sim}$, $\rk((v_x)_{x\in Y})\leq 1$. As a consequence, the vectors $v_x$, with $x\in Y$, are pairwise colinear.\\

$\Longleftarrow$. Then, for any $Y\in X/{\sim}$:
\begin{itemize}
\item If $\rk((v_y)_{y\in Y})=0$, we put $w_Y=0$.
\item If $\rk((v_y)_{y\in Y})=1$, let us choose a nonzero vector $w_Y$ such that for any $y\in Y$, $w_Y$ and $v_y$ are colinear.
\end{itemize}
This defines a family $W=(w_y)_{Y\in X/{\sim}}$, and $r_V/{\sim}=r_W$, so is the rank function of a linear matroid.
\end{proof}

\section{Fundamental polynomial invariant}

\subsection{Construction and examples}

\begin{notation}
For any $k\geq 1$, we put
\[\binom{T}{k}=\frac{T(T-1)\cdots (T-k+1)}{k!}\in\K[T].\]
By convention, $\displaystyle\binom{T}{0}=1$.
\end{notation}

Let us apply Theorem  \ref{theo1.3} on $\calH_{\bfbool}$, with the character $\epsilon_\delta$.

\begin{theo}\label{theo5.1}
We define a bialgebra morphism from $(\calH_{\bfbool},\star_1,\Delta)$ to $(\K[T],m,\Delta)$ as follows:
\begin{align*}
&\forall x\in\ker(\varepsilon_\Delta),&\Phi(x)&=\sum_{k=1}^\infty\epsilon_\delta^{\otimes k}\circ\tdelta^{(k-1)}(x)\binom{T}{k}.
\end{align*}
Its restriction to $\calH_{\bfboolmax}$ is the unique double bialgebra morphism from $(\calH_{\bfboolmax},\star_1,\Delta,\delta^W)$ to $(\K[T],m,\Delta,\delta)$.
\end{theo}

This map $\Phi$ has a combinatorial interpretation. If $f\in\rmbool(X)$, with $X$ nonempty, then by definition of $\epsilon_\delta$ and $\Delta$,
\begin{align*}
\Phi(\overline{f})&=\sum_{k=1}^\infty\sum_{\substack{X=X_1\sqcup\cdots\sqcup X_k,\\ X_1,\ldots,X_k\neq\emptyset}}\epsilon_\delta(f_{\mid X_1})\ldots\epsilon_\delta(f_{\mid X_k})\binom{T}{k}\\
&=\sum_{k=1}^\infty\sum_{\substack{X=X_1\sqcup\cdots\sqcup X_k,\\ X_1,\ldots,X_k\neq\emptyset,\\ f_{\mid X_1},\ldots f_{\mid X_k}\mbox{\scriptsize modular}}}\epsilon_\delta(f_{\mid X_1})\ldots\epsilon_\delta(f_{\mid X_k})\binom{T}{k}\\
&=\sum_{k=1}^\infty |\{c:X\twoheadrightarrow [k]\mid\forall i\in [k],\: f_{\mid c^{-1}(i)}\mbox{ modular}\}|\binom{T}{k}.
\end{align*}
Consequently:

\begin{prop}\label{prop5.2}
Let $f\in\rmbool(X)$ be a boolean function and $n\geq 1$.
Then $\Phi\left(\overline{f}\right)(n)$ is the number of maps $c:X\longrightarrow [n]$ such that for any $i\in [n]$,
$f_{\mid c^{-1}(i)}$ is modular. In particular:
\begin{enumerate}
\item If $H$ is a hypergraph, then $\Phi\left(\overline{\gamma(H)}\right)(n)$ is the number of maps $c:V(H)\longrightarrow [n]$ such that for any $Y\in E(H)$, with $|Y|\geq 2$, $c_{\mid Y}$ is not constant.
In other words, $\Phi\left(\overline{\gamma(H)}\right)$ is the chromatic polynomial $\Phichr(H)$. 
\item If $G$ is a graph, then $\Phi\left(\overline{\rk_G}\right)(n)$ is the number of maps $c:E(G)\longrightarrow [n]$ such that for any $i\in [n]$, $G_{\mid c^{-1}(i)}$ is a forest.
\item If $V=(v_x)_{x\in X}$ is a family of vectors, then $\Phi\left(\overline{\rk_V}\right)(n)$ is the number of maps $c:X\longrightarrow [n]$ such that for any $i\in [n]$, the family obtained from $(v_x)_{x\in c^{-1}(i)}$ by suppression of the eventual zero vectors is free. 
\end{enumerate}\end{prop}

This polynomial invariant does not distinguish indecomposable hyper-rigid boolean functions on the same set, as shown in the following result. 

\begin{prop}
Let $X$ be a finite set of cardinality $k$, and let $f\in\rmboolhs(X)$, indecomposable. Then
\[\Phi(\overline{f})=T(T-1)\cdots (T-k+1).\]
\end{prop}

\begin{proof}
As $f$ is indecomposable and hyper-rigid, for any nonempty disjoint subsets, $A,B\subseteq X$, $f(A\sqcup B)\neq f(A)+f(B)$, so, for any nonempty $Y\subseteq X$, $f_{\mid Y}$ is indecomposable. Consequently, $f_{\mid Y}$ is modular if, and only if, $|Y|\leq 1$. 
By Proposition \ref{prop5.2}, for any $n\geq 1$, $\Phi(\overline{f})(n)$ is the number of maps $c:X\longrightarrow [n]$ such that for any $i\in [n]$, $|c^{-1}(i)|\leq 1$, that is to say the number of injections $c:X\longrightarrow [n]$. Consequently, 
\[\Phi(\overline{f})(n)=n(n-1)\cdots (n-k+1).\]
This implies the result.\end{proof}

\begin{prop}
Let $X$ be a finite set and $f\in\rmbool(X)$. Then $\Phi(\overline{f})\in\Z[X]$.
\end{prop}

\begin{proof}
We have
\[\Phi(\overline{f})=\sum_{k=1}^\infty\sum_{\substack{X=X_1\sqcup\cdots\sqcup X_k,\\ X_1,\ldots,X_k\neq\emptyset}}\epsilon_\delta(f_{\mid X_1})\ldots\epsilon_\delta(f_{\mid X_k})\binom{T}{k}.\]
By cocommutativity,
\begin{align*}
\Phi(\overline{f})&=\sum_{k=1}^\infty\sum_{\{X_1,\ldots,X_k\}\:\mbox{\scriptsize partition of }X} k!\epsilon_\delta(f_{\mid X_1})\ldots\epsilon_\delta(f_{\mid X_k})\binom{T}{k}\\
\Phi(\overline{f})&=\sum_{k=1}^\infty\sum_{\{X_1,\ldots,X_k\}\:\mbox{\scriptsize partition of }X}\epsilon_\delta(f_{\mid X_1})\ldots\epsilon_\delta(f_{\mid X_k}) T(T-1)\cdots (T-k+1).
\end{align*}
The conclusion comes from $\epsilon_\delta(f_{\mid X_1})\ldots\epsilon_\delta(f_{\mid X_k})\in\{0,1\}$. 
\end{proof}

\begin{remark}
The polynomial $\phi(\overline{f})$ is not a polynomial with alternating signs, contrarily with the chromatic polynomial of graphs. Here is a counterexample. Let $n\geq 1$ and $f\in\rmbool([n])$ such that:
\begin{align*}
&\forall Y\subseteq X,&f(Y)&=\begin{cases}
|Y|\mbox{ if }Y\neq [n],\\
0\mbox{ otherwise}.
\end{cases}\end{align*}
Therefore, if $Y\subsetneq X$, $f_{\mid Y}$ is modular, whereas $f$ is not. Consequently, for any $k\geq 0$, $\Phi(\overline{f})(k)$ is the number of maps $c:[n]\longrightarrow [k]$ which are not constant, that is to say $k^n-k$. 
We obtain $\Phi(\overline{f})=T^n-T$. 
\end{remark}

As a consequence of Theorem \ref{theo1.2}, applied to the double bialgebra $\calH_{\bfboolmax}$: 

\begin{prop}\label{prop5.5}
For any $x\in\calH_{\bfbool}$, we put $\mu_{\bfbool}(x)=\Phi(x)(-1)$. This defines a character of $\calH_{\bfbool}$. We denote by $S$ the antipode of $(\calH_\bfbool,\star_1,\Delta)$. Then, for any $x\in\calH_{\bfboolmax}$,
\[S(x)=\left(\mu_{\bfbool}\otimes\id_{\calH_{\bfboolmax}}\right)\circ\delta^W(x).\]
\end{prop}

\subsection{Extension to counitary boolean functions}

In fact, we can do better for the compatibility of $\Phi$ and $\delta$:

\begin{prop}\label{prop5.6}
Let $X$ be a finite set, and $f\in \rmbool(X)$. We denote by $x$ its class in $\calH_{\bfboolc}$. Then 
$\delta\circ \Phi(x)=(\Phi\otimes \phi)\circ \delta^W(x)$ if, and only if, $f\in \rmboolc[X]$.
\end{prop}

\begin{proof}
As usual, we identify $\K[T]\otimes\K[T]$ and $T'\K[T,T']$, through the map
\[\left\{\begin{array}{rcl}
\K[T]\otimes\K[T]&\longrightarrow&\K[T,T']\\
P(T)\otimes Q(T)&\longmapsto&P(T)Q(T').
\end{array}\right.\]
Let us assume that $f\in \rmboolc(X)$. For any $k,l\geq 1$,
\[(\Phi\otimes\Phi)\circ\delta^W(x)(k,l)=\Phi(x)(kl).\]
We proceed by induction on $k$. If $k=1$, as $\epsilon_\delta$ is a left counit for $\delta^W_{\mid\calH_{\bfboolc}}$,
\begin{align*}
(\Phi\otimes\Phi)\circ\delta^W(x)(1,l)&=(\epsilon_\delta\circ\Phi\otimes\Phi)\circ\delta^W(x)(l)\\
&=(\epsilon_\delta\otimes\Phi)\circ\delta^W(x)(l)\\
&=\Phi(x)(l).
\end{align*}
Let us assume the result at rank $k$. As $\Phi$ is compatible with the product and the coproduct $\Delta$,
\begin{align*}
(\Phi\otimes\Phi)\circ\delta^W(x)(k+1,l)&=\left(\Delta\otimes\id_{\K[T]}\right)\circ(\Phi\otimes\Phi)\circ\delta^W(x)(1,k,l)\\
&=(\Phi\otimes\Phi\otimes\Phi)\circ\left(\Delta\otimes\id_{\calH_{\bfbools}}\right)\circ\delta^W(x)(1,k,l)\\
&=(\Phi\otimes\Phi\otimes\Phi)\circ(\star_1)_{1,3,24}\circ\left(\delta^W\otimes\delta^W\right)\circ\Delta(x)(1,k,l)\\
&=m_{1,3,24}\circ (\Phi\otimes\Phi\otimes\Phi\otimes\Phi)\circ\left(\delta^W\otimes\delta^W\right)\circ\Delta(x)(1,k,l)\\
&= (\Phi\otimes\Phi\otimes\Phi\otimes\Phi)\circ\left(\delta^W\otimes\delta^W\right)\circ\Delta(x)(1,l,k,l)\\
&= (\Phi\otimes\Phi)\circ\Delta(x)(l,kl)\\
&=\Phi(x)(l+kl).
\end{align*}
We used the induction hypothesis and the case $k=1$ for the sixth equality. As a conclusion, the two polynomials in two indeterminates $(\Phi\otimes\Phi)\circ\delta^W(x)$ and $\delta\circ\Phi(x)$ coincide on $\N^*\times\N^*$, so they are equal.\\

Let us now assume that $\delta\circ \Phi(x)=(\Phi\otimes \Phi)\circ \delta^W(x)$. Then
\begin{align*}
\Phi(x)&=(\epsilon_\delta\circ \id_{\K[X]})\circ \delta\circ\Phi(x)\\
&=(\epsilon_\delta\circ\Phi\otimes\Phi)\circ \delta^W(x)\\
&=(\epsilon_\delta\otimes\Phi)\circ \delta^W(x)\\
&=\sum_{\sim\in\calE^W(b)} \epsilon_\delta(f/\sim)\Phi(\overline{f\mid\sim})\\
&=\Phi(x)+\sum_{\sim\in\calE^W(f)\setminus\{\sim_f^i\}} \epsilon_\delta(f/\sim)\Phi(\overline{f\mid\sim})\\
&=\Phi(x)+\sum_{\substack{\sim\in\calE^W(f)\setminus\{\sim_f^i\},\\\mbox{\scriptsize$f/\sim$ modular}}} \Phi(\overline{f\mid\sim}).
\end{align*}
Consequently, 
\[\sum_{\substack{\sim\in\calE^W(b)\setminus\{\sim_f^i\},\\\mbox{\scriptsize$f/\sim$ modular}}} \Phi(\overline{f\mid\sim})=0.\]
For any modular function $g$, $\Phi(\overline{g})$ is a nonzero polynomial, with a positive leading coefficient. A nonempty sum of such polynomials is not zero: therefore, if $\sim'\in \calE^W(f)$ with $f/\sim'$ modular, then $\sim'=\sim_f^i$.\\

Let now $\sim\in \calE^W(f)$. Let $\sim'\in \calE[X]$, containing $\sim$, such that $\overline{\sim'}=\sim_{f/\sim}^i$. Then 
\[f/\sim'=(f/\sim)/\overline{\sim'}=(f/\sim)/\sim_{f/\sim}^i \mbox{ is modular}.\]
Let $Y\in X/\sim'$. We assume that $f_{\mid Y}=f_{\mid Y_1}\star f_{\mid Y_2}$, with $Y=Y_1\sqcup Y_2$. As $\sim\subseteq \sim'$, $Y$ is a disjoint union of classes of $\sim$.  
As the restriction of $f$ to any class of $\sim$ is indecomposable, $Y_1$ and $Y_2$ are also disjoint unions of classes of $\sim$. Consequently, 
\[(f/\sim)_{\mid Y/\sim\cap Y^2}=\left(f_{\mid Y}\right)/\sim\cap Y^2=\left(\left(f_{\mid Y_1}\right)/\sim\cap Y_1^2\right) \star\left(\left(f_{\mid Y_2}\right)/\sim\cap Y_2^2\right).\]
As $(f/\sim)_{\mid Y/\sim\cap Y^2}$ is indecomposable (as $\overline{\sim'}=\sim_{f/\sim}^i$), $Y_1=\emptyset$ or $Y_2=\emptyset$. Hence, $f_{\mid Y}$ is indecomposable, so $\sim'\in \calE^W(f)$. 
As $f/\sim'$ is modular, necessarily $\sim'=\sim_f^i$. Consequently, $\ic(f)=\cl(\sim')=\cl(\overline{\sim'})=\ic(f/\sim)$. So $\sim\in \calE^S(f)$. 
We proved that $\calE^W(f)=\calE^S(f)$, which gives that $f\in \rmboolc(X)$. \end{proof}

\begin{example} The map $\Phi$ is not compatible on the whole $\calH_{\bfbool}$ with $\delta^W$, nor $\delta^S$, as shown by the following example.
We consider again $f\in\rmbool([3])$, as in Examples \ref{ex3.2}, \ref{ex3.1}, \ref{ex3.4}, and \ref{ex4.1}: 
\begin{align*}
&\forall i,j\in [3],\mbox{ with }i\neq j,&f(\{i,j\})&\neq f(\{i\})+f(\{j\}),\\
&&f(\{1,2,3\})&=f(\{1,2\})+f(\{3\}).
\end{align*}
Direct computations show that
\begin{align*}
(\Phi\otimes\Phi)\circ\delta^W(\overline{f})&=T(T-1)(T-2)\otimes T^3+T(3T-2)\otimes T^2(T-1)+T\otimes T(T-1)(T-2),\\
(\Phi\otimes\Phi)\circ\delta^S(\overline{f})&=T(T-1)(T-2)\otimes T^3+2T(T-1)\otimes T^2(T-1)+T\otimes T(T-1)(T-2),\\
\delta\circ\Phi(\overline{f})&=T(T-1)(T-2)\otimes T^3+3T(T-1)\otimes T^2(T-1)+T\otimes T(T-1)(T-2).
\end{align*}\end{example}

\section{Appendix : rank functions of matroids are rigid boolean functions}

Proposition \ref{prop4.23} is a well-known result in the theory of matroids \cite[Lemma 5.8]{Vyuka}. For the sake of completeness, and to avoid to use the vocabulary of matroids (independents, bases, cycles, etc), we now give a proof within the vocabulary of rank functions, as far as possible.
Let us firstly introduce some vocabulary, and some preliminary results.

\begin{defi}
Let $X$ be a finite set, and $f\in\rmboolm(X)$. Let $Y\subseteq X$. A basis of $Y$ is a subset $B_Y$ of $Y$, such that $f(Y)=f(B_Y)=|B_Y|$. 
\end{defi}

\begin{lemma}\label{lemma6.2}
Let $X$ be a finite set, and $f\in\rmboolm(X)$. Let $Y\subseteq X$, such that $f(Y)=|Y|$. Then, for any $Z\subseteq Y$, $f(Z)=|Z|$.
\end{lemma}

\begin{proof}
By definition of a rank function, $f(Z)\leq |Z|$. Moreover, as $f$ is submodular, 
\[f(Y)=f(Z\cup (Y\setminus Z))+f(\emptyset)\leq f(Z)+f(Y\setminus Z)\leq f(Z)+|Y|-|Z|.\]
Therefore, as $f(Y)=|Y|$,
\[|Y|=|Z|+|Y\setminus Z|\leq f(Z)+|Y|-|Z|.\]
So $f(Z)\geq |Z|$. 
\end{proof}

\begin{lemma}\label{lemma6.3}
Let $X$ be a finite set, and $f\in\rmboolm(X)$. Let $Z\subseteq Y\subseteq X$, and $0\leq k\leq f(Y)-f(Z)$. There exists $Z'\subseteq Y\setminus Z$, such that $f(Z')=|Z'|=k$, and $f(Z\sqcup Z')=f(Z)+|Z'|$.
\end{lemma}

\begin{proof}
We proceed by induction on $k$. If $k=0$, take $Z'=\emptyset$. Let us prove the result for $k=1$. Let $z\in Y\setminus Z$. As $f$ is submodular and increasing, 
\[f(Z)\leq f(Z\sqcup\{z\})\leq f(Z)+f(\{z\})\leq f(Z)+1,\]
so $f(Z\sqcup\{z\})=f(Z)$ or $f(Z)+1$. Let us assume that for any $z\in Y\setminus Z$, $f(Z\sqcup\{z\})=f(Z)$. Let us prove that for any $Z'\subseteq Y\setminus Z$, $f(Z\sqcup Z')=f(Z)$. We proceed by induction on $|Z'|$.
If $Z'=\emptyset$, this is obvious. If $|Z'|=1$, this is our hypothesis. If $|Z'|\geq 2$, choosing $z_1\neq z_2\in Z'$, putting $Z'_i=Z\setminus\{z_i\}$ for $i\in [2]$, then 
\begin{align*}
|Z'_1|=|Z'_2|&=|Z'|-1,&Z'_1\cup Z'_2&=Z',&|Z'_1\cap Z'_2|&=|Z'|-2.
\end{align*}
As $f$ is submodular, 
\[f(Z\sqcup Z')+f(Z\sqcup (Z'_1\cap Z'_2))\leq f(Z\sqcup Z'_1)+f(Z\sqcup Z'_2).\]
By the induction hypothesis, 
\[f(Z\sqcup Z')+f(Z)\leq f(Z)+f(Z),\]
so $f(Z\sqcup Z')\leq f(Z)$. Moreover, as $f$ is increasing, $f(Z\sqcup Z')=f(Z)$.\\

From this result with $Z'=Y\setminus Z$, $f(Y)=f(Z)$, which contradicts $f(Y)-f(Z)\geq 1$. So, there exists $z\in Y\setminus Z$, such that $f(Z\sqcup\{z\})=f(Z)+1$.\\

Let us now assume the result at rank $k-1$. The induction hypothesis gives the existence of $Z''\subseteq Y\setminus Z$, of cardinality $k-1$, such that $f(Z\sqcup Z'')=f(Z)+k-1$. The case $k=1$ gives the existence of $z\in Y\setminus (Z\sqcup Z'')$, such that 
\[f(Z\sqcup Z''\sqcup\{z\})=f(Z\sqcup Z'')+1=f(Z)+k.\]
Take then $Z'=Z''\sqcup\{z\}$.\\

We proved the existence of $Z'\subseteq Z\setminus Y$, of cardinality $k$, such that $f(Z\sqcup Z')=f(Z)+k$. It remains to prove that $f(Z')=k$. We immediately have $f(Z')\leq |Z'|=k$. As $f$ is submodular,
\[f(Z\sqcup Z')=f(Z)+k\leq f(Z)+f(Z'),\]
so $f(Z')\geq k$.\end{proof}

\begin{prop}\label{prop6.4}
Let $X$ be a finite set, and $f\in\rmboolm(X)$. 
\begin{enumerate}
\item Any $Y\subseteq X$ has a basis.
\item Let $Z\subseteq Y\subseteq X$, and let $B_Z$ be a basis of $Z$. There exists $B_Y\subseteq Y\setminus Z$, such that $B_Z\sqcup B_Y$ is a basis of $Y$. 
\end{enumerate}
\end{prop}

\begin{proof}
1. Apply Lemma \ref{lemma6.3} with $\emptyset\subseteq Y$ and $k=f(Y)$. Then $B_Y=Z'$ is a basis of $Y$.\\

2. Apply Lemma \ref{lemma6.3} with $B_Z\subseteq Y$ and $k=f(Y)-f(Z)=f(Y)-f(B_Z)$. There exists $B_Y\subseteq Z\setminus B_Y$, such that 
\[f(B_Y)=|B_Y|=f(Y)-f(B_Z)=f(Y)-f(Z),\]
and 
\[f(B_Z\sqcup B_Y)=f(Z)+|B_Y|=f(B_Z)+|B_Y|=f(Y).\]
Note that $f(Y)=f(B_Z)+|B_Y|=|B_Z|+|B_Y|=|B_Z\sqcup B_Y|$, so $B_Z\sqcup B_Y$ is a basis of $Y$. It remains to show that $B_Y\cap Z=\emptyset$. As $f$ is submodular,
\[f(Z\cup B_Y)+f(Z\cap B_Y)\leq f(Z)+f(B_Y)=|B_Z|+|B_Y|\]
By Lemma \ref{lemma6.2}, as $f(B_Y)=|B_Y|$, $f(Z\cap B_Y)=|Z\cap B_Y|$. As $f$ is increasing, $f(Z\cup B_Y)\geq f(B_Z\sqcup B_Y)=|B_Z|+|B_Y|$. We obtain finally
\[|B_Z|+|B_Y|+|Z\cap B_Y|\leq |B_Z|+|B_Y|,\]
so $Z\cap B_Y=\emptyset$.\end{proof}

\begin{proof}(Proposition \ref{prop4.23}).
Let us assume that $f(A\sqcup B)=f(A)+f(B)$. Let us first prove that for any $B'\subseteq B$, $f(A\sqcup B')=f(A)+f(B')$. Let us choose a basis $B_{B'}$ of $B$' (Proposition {prop6.4}, first item).
There exists $B_A\subseteq A$, such that $B_A\sqcup B_{B'}$ is a basis of $A\sqcup B'$ (Proposition {prop6.4}, second item). There exists $B_{B''}\subseteq B\setminus B'$, such that $B_A\sqcup B_{B'}\sqcup B_{B''}$ is a basis of $A\sqcup B$. We obtain
\begin{align*}
f(B')&=f(B_{B'})=|B_{B'}|,\\
f(A\sqcup B')&=f(B_A\sqcup B_{B'})=|B_A|+|B_{B'}|,\\
f(A\sqcup B)&=f(B_A\sqcup B_{B'}\sqcup B_{B''})=|B_A|+|B_{B'}|+|B_{B''}|.
\end{align*}
By Lemma \ref{lemma6.2}, applied to $B_A\sqcup B_{B'}\sqcup B_{B''}$, 
\[f(B_{B'}\sqcup B_{B''})=|B_{B'}|+|B_{B''}|.\]
Moreover, as $f$ is increasing, $f(B_A)\leq f(A)$, and $f(B_{B'}\sqcup B_{B''})\leq f(B)$. By hypothesis, 
\[f(A\sqcup B)=|B_A|+|B_{B'}|+|B_{B''}|=f(B_A)+f(B_{B'}\sqcup B_{B''})=f(A)+f(B).\]
Combined with the two preceding inequalities, we deduce that $f(A)=f(B_A)=|B_A|$ (and also $f(B_{B'}\sqcup B_{B''})=f(B)$, which we won't use). Therefore, 
\[f(A\sqcup B')=|B_A|+|B_{B'}|=f(A)+f(B').\]

Similarly, we can prove that for any $A'\subseteq A$, $f(A'\sqcup B)=f(A')+f(B)$. If $A'\subseteq A$ and $B'\subseteq B$, we obtain that $f(A'\sqcup B)=f(A')+f(B)$. Applying the first part of the proof with $A'$ instead of $A$, we obtain that $f(A'\sqcup B')=f(A')+f(B')$.\end{proof}

\bibliographystyle{amsplain}
\bibliography{biblio}

\end{document}